\pdfoutput=1

\documentclass[12pt]{amsart}
\usepackage[top=2cm, bottom=2cm, left=2cm, right=2cm]{geometry}
\usepackage[colorlinks,citecolor = red, linkcolor=blue,hyperindex]{hyperref}
\usepackage{euscript,verbatim, mathrsfs}
\usepackage{eufrak}
\usepackage[psamsfonts]{amssymb}
\usepackage{bbm}
\usepackage{graphicx}
\usepackage{float}
\usepackage{float, tikz}
\usepackage[all, cmtip]{xy}
\usepackage{upref, xcolor, dsfont}
\usepackage{amsfonts,amsmath,amstext,amsbsy, amsopn,amsthm, amscd,tikz-cd}
\usepackage{mathdots}
\usepackage{enumerate}
\usepackage{url}
\usepackage{extpfeil}

\usepackage{mathtools}
\usepackage{bookmark}
 \usepackage{euscript}
\usepackage{mathptmx}       
\usepackage{helvet}        
\usepackage{courier}        
\usepackage{type1cm}        
\usepackage{multicol}        
\usepackage[bottom]{footmisc}
\usepackage{comment}
\usepackage{amscd,pifont}
\usepackage[curve]{xypic}
\usepackage{calligra}

\DeclareMathAlphabet{\mathcal}{OMS}{cmsy}{m}{n}
\DeclareSymbolFont{largesymbols}{OMX}{cmex}{m}{n}

\newtheorem{theorem}{Theorem}[section]
\newtheorem{mainthm}{Theorem}

\newtheorem{lemma}[theorem]{Lemma}

\newtheorem{proposition}[theorem]{Proposition}
\newtheorem{corollary}[theorem]{Corollary}
\newtheorem*{definition*}{Definition}

\newtheorem*{observation*}{Observation}

\newtheorem*{assumption*}{Assumption}
\newtheorem*{question*}{Question}
\newtheorem*{conjecture*}{Conjecture}

\theoremstyle{definition}
\newtheorem{definition}{Definition}

\theoremstyle{remark}
\newtheorem{remark}{Remark}[section]
\newtheorem*{remark*}{Remark}
\newtheorem{example}[remark]{Example}
\newtheorem{problem}{Problem}

\newcommand{\R}{\mathbb{R}}
\newcommand{\N}{\mathbb{N}}
\newcommand{\Z}{\mathbb{Z}}
\newcommand{\Q}{\mathbb{Q}}
\newcommand{\D}{\mathbb{D}}
\newcommand{\C}{\mathbb{C}}

\newcommand{\T}{\mathbb{T}}

\newcommand{\MM}{\mathcal{M}}

\newcommand{\Tr}{\mathrm{Tr}}
\newcommand{\Sol}{\mathrm{Sol}}

\newcommand{\spann}{\mathrm{span}}

\newcommand{\supp}{\mathrm{supp}}

\newcommand{\hank}{\mathrm{Hank}}

\newcommand{\BMOA}{\mathrm{BMOA}}
\newcommand{\diag}{\mathrm{diag}}
\newcommand{\adiag}{\mathrm{adiag}}

\newcommand{\oss}{\textit{o.s.s.\,\,}}

\newcommand{\an}{\text{\, and \,}}

\newcommand{\ch}{\mathbbm{1}}

\usepackage{footnote}
\usepackage{footmisc}
\makesavenoteenv{tabular}
\makesavenoteenv{table}

\renewcommand{\thefootnote}{\alph{footnote}}

\newcommand{\astfootnote}[1]{%
\let\oldthefootnote=\thefootnote%
\setcounter{footnote}{0}%
\renewcommand{\thefootnote}{\fnsymbol{footnote}}%
\footnote{#1}%
\let\thefootnote=\oldthefootnote%
}

\numberwithin{equation}{section}

\begin{document}

\title[SAP of Hankel systems]{Self-absorption of Hankel systems on monoids \\ --{\MakeLowercase{ a seemingly universal property}}}


\author
{Yong Han}
\address
{Yong HAN: School of Mathematical Sciences, Shenzhen University, Shenzhen 518060, Guangdong, China}
\email{hanyong@szu.edu.cn}

\author
{Yanqi Qiu}
\address
{Yanqi Qiu: School of Fundamental Physics and Mathematical Sciences, HIAS, University of Chinese Academy of Sciences, Hangzhou 310024, China; Institute of Mathematics, AMSS, Chinese Academy of Sciences, Beijing 100190, China}
\email{yanqi.qiu@hotmail.com, yanqiqiu@ucas.ac.cn}

\author{Zipeng Wang}
\address{Zipeng WANG: College of Mathematics and Statistics, Chongqing University, Chongqing
401331, China}
\email{zipengwang2012@gmail.com}

\begin{abstract}
Given any cancellative monoid $\mathcal{M}$, we study the Hankel system determined by its multiplication table (see \eqref{def-gamma-M}). We prove that the Hankel system admits self-absorption property (Definition~\ref{def-self-tensor}) provided that the monoid $\mathcal{M}$ has the  local algebraic structure: 
\[
\big(ax  = by, cx=dy, az=bw  \,\, \text{in $\mathcal{M}$}\big)\Longrightarrow \big(cz=dw \,\, \text{in $\mathcal{M}$}\big). 
\] Our result holds for all group-embeddable monoids and goes beyond. In particular, it works for all cancellative Abelian monoids and most common non-Abelian cancellative monoids such as 
$$
\mathrm{SL}_d(\N): = \big\{[a_{ij}]_{1\le i,j\le d}\in \mathrm{SL}_d(\Z)\big| a_{ij} \in \N\big\}. 
$$

The Hankel system determined by the multiplication table of a monoid  is further generalized to that determined by level sets of any abstract two-variable map (see \eqref{def-G-Phi}). We introduce an algebraic notion of lunar maps (Definition~\ref{def-SD}) and establish a stronger hereditary self-absorption property (Definition~\ref{def-hsap}) for the corresponding generalized Hankel systems. As a consequence, we prove the self-absorption property for arbitrary spatial compression of the regular representation system  $\{\lambda_G(g)\}_{g\in G}$ of any discrete group $G$,  as well as  the Hankel system $\{\Gamma_\ell^\Phi\}$  determined by the level sets of any rational map of the form $\Phi(x,y)=a x^m  + b y^n$ with $a,b,m,n\in \mathbb{Z}^*$: 
\[
\Gamma_\ell^{\Phi}(x, y)= \mathds{1}(a x^m + b y^n= \ell), \quad x, y\in \mathbb{N}^*, \, \ell\in \Phi (\mathbb{N}^*\times \mathbb{N}^*).
\]

Our result is already new in the classical setting of the additive monoid  $\mathbb{N}$ of non-negative integers.  The self-absorption property is  applied to the study of completely bounded Fourier multipliers between Hardy spaces.  In the case of endpoint spaces, we obtain the full  characterizations: 
\[
(H^1, \mathrm{BMOA})_{cb} = \mathrm{BMOA}  \an (H^1, H^2)_{cb}  =  (H^2, \mathrm{BMOA})_{cb} = \mathrm{BMOA}^{(2)}.
\] 
Then, using complex interpolation, we obtain, for any $1\le p\le 2 \le q<\infty$, 
\[
\mathrm{BMOA}^{(u)} \subset (H^p, H^q)_{cb} \quad \text{with $\frac{1}{u} = \frac{1}{p} - \frac{1}{q}$}.
\]
Here $\mathrm{BMOA}^{(u)}$ is a new  Banach space of  analytic functions on the unit circle $\T$ with 
\[
\| \varphi\|_{\BMOA^{(u)}}\approx   |\widehat{\varphi}(0)| +  \sup_{n\ge 1} \Big\{ \sum_{k=1}^\infty  \Big(\sum_{kn\le j < (k+1)n} | \widehat{\varphi}(j)|^u \Big)^2\Big\}^{1/2u}. 
\]
 Further applications are:  i) exact complete bounded norm of the Carleman embedding in any   dimension;  ii) mixed Fourier-Schur multiplier inequalities with critical exponent $4/3$;  iii) failure of hyper-complete-contractivity for  the Poisson semigroup. 
\end{abstract}


\maketitle

\setcounter{tocdepth}{1}

\tableofcontents

\setcounter{equation}{0}

 \section*{Foreword: What does absorption mean and why it is natural ?}
{\flushleft \bf I. Fell's absorption in group representation theory.}
The well-known Fell's aborption principle states that,  the left regular representation $\lambda_G$  of any discrete countable group $G$ absorbs all its unitary representations. In other words, for any unitary representation $\pi: G \rightarrow \mathcal{U}(H)$,   there exists a unitary operator $U \in \mathcal{U}(\ell^2(G)\otimes_2 H)$ such that 
\[
\lambda_G(g) \otimes \pi(g) =  U [\lambda_G(g) \otimes Id_H] U^{-1}, \quad \forall g \in G.  
\]
This means that all operators $\lambda_G(g)\otimes \pi(g)$ can be  simultaneously block-diagonalized, with all block-diagonal entries being $\lambda_G(g)$.  

{\flushleft \bf II. A new absorption in operator-space theory.}    We say that a  family of Hilbert-space operators $\{x_i: i\in I\} \subset B(H)$ contractively absorbs another family $\{y_i: i\in I\}\subset B(K)$   if there exist two contractive operators $U, V$ on $H\otimes_2 K$, such that all operators $x_i\otimes y_i$ simultaneously factor through block-diagonalized operators as: 
\begin{align}\label{xy-ab}
 x_i \otimes y_i     = U (x_i \otimes Id_K) V, \quad \forall i \in I. 
\end{align}

The main focus of this paper is the self-absorption property. Specifically,  we consider the problem when the family $\{x_i: i \in I\}$ contractively absorbs itself.  In fact, from the proofs of our main results (Theorem \ref{thm-lunar-monoid} in \S \ref{sec-intro} and Theorem \ref{thm-Phi} in \S \ref{sec-hank-monoid}), we obtain \eqref{xy-ab} for   very general  Hankel systems  with $U$ and $V$   even being  partial isometries (see \eqref{simul-diag} and \eqref{CD-Hab}). 

The absorption property mentioned above is an algebraic property of the  Hankel systems. However,  for simplicity, we shall state our main results  in the framework of the operator-space theory in a slightly weaker version. Indeed,  the simultaneous factorization \eqref{xy-ab} implies  the following contractivity inequalities:  
\begin{align}\label{cc-ineq}
\|\sum_{i} c_i \otimes x_i \otimes y_i\| \le  \|\sum_{i} c_i \otimes x_i \|, 
\end{align}
where $(c_i)_{i}$ is any finitely supported family in an arbitrary $C^*$-algebra.  Conversely, the fundamental factorization/extension theorem in operator-space theory (consequence of the Stinespring and Arveson's dialation theorems) says that \eqref{cc-ineq} is equivalent to the existence of a $C^*$-representation $\pi: B(H)\rightarrow B(\widehat{H})$ and two contractive operators $U, V\in B(\widehat{H} , H\otimes_2 K)$ with 
\[
x_i \otimes y_i   = U  \pi( x_i) V^*, \quad \forall i \in I. 
\]

The inequality \eqref{cc-ineq}  has close connections to harmonic analysis: it appears to be related to the well-known open problem of the existence of Nehari Theorem on higher dimensional torus (see Problem \ref{prob-d-torus} in \S \ref{sec-intro}),  and  it will be applied to the study of completely bounded Fourier multipliers between Hardy spaces (see Section \ref{sec-app}).  For instance,  we show that, within the framework of operator-space theory,  it is surprising that Janson's hypercontractivity fails on Hardy spaces (see Corollary \ref{cor-hyper}).

Finally, it is worth mentioning the following  universal property for an arbitrary family of unit norm Boolean operators $\{x_i\}$ (see \S \ref{sec-boolen-graph} for its definition) :   under an appropriate positivity condition on the coefficients $c_i$, we always have (see Proposition \ref{prop-positive})
 \[
\big\| \sum_{i} c_i \otimes   x_i^{\otimes m} \big\| =  \big\| \sum_{i} c_i \otimes   x_i  \big\|. 
\]

{\flushleft \bf  III.  Relation with Graph-theory. }  For  any family $\{x_i: i\in I\}$ of  unit norm Boolean operators with disjoint supports (see \S \ref{sec-boolen-graph}),  one can naturally associate an oriented colored graph $\mathcal{G}$.  Similarly, an associated oriented colored graph  $\widehat{\mathcal{G}}$ is constructed from $\{x_i \otimes x_i: i\in I\}$.  In a subsequent paper, we show that, if all connected components  of $\widehat{\mathcal{G}}$  (after certain natural operations)  can be reduced to saturated subgraphs of $\mathcal{G}$, then $\{x_i: i \in I\}$ contractively absorbs itself.  This realization allows us, on the one hand,  to verify the self-absorption  property for various specific  examples of families of Boolean operators, and on the other hand, provides a conjectured equivalent graph-theory criterion of the self-aborption property in the context of Boolean operators.

\section{Introduction}\label{sec-intro}

This paper is mainly devoted to the study of Hankel systems on abstract monoids.  The key ingredients are the notion of   self-absorption property (Definition~\ref{def-self-tensor}), lunar monoids (Definition~\ref{def-AAM}), SAP monoids (Definition~\ref{def-SAP-monoid}) and the coupled foliations (Figure~\ref{fig-foliation}).  

We begin by introducing the definition of the Hankel systems on monoids. Let $\MM$ be a cancellative monoid  (not necessarily Abelian). The  associated Hankel system 
$\{\Gamma_t^\MM\}_{t\in \MM} \subset B(\ell^2(\MM))$  is a collection of operators defined as follows: 
\begin{align}\label{def-gamma-M}
\Gamma_t^\MM(s_1, s_2)= \mathds{1}(s_1s_2 = t),  \quad s_1, s_2 \in \MM. 
\end{align}
Here, we recall that a cancellative monoid $\MM$ refers to a semigroup with a  unit element in which  the implication $(abc = adc) \Longrightarrow (b=d)$ holds in $\MM$.  Note that   $\MM$ is cancellative if and only if   $\| \Gamma_t^\MM\|_{B(\ell^2(\MM))} = 1$ for all $t\in \MM$.

For instance, if $\MM = \N^d$ (resp. $\N^{(\infty)}$) is the free Abelian monoid of $d$ generators (resp. countably infinite generators),  then  the operator $\Gamma_t^{\MM}$ is unitarily equivalent,  via Fourier transform, to an elementary small Hankel operator on the Hardy space of the   $d$-dimensional torus $\T^d$ (resp. infinite dimensional torus $\T^\N$).   Our work on self-absorptioin property of the Hankel system on $\N^d$ is closely related to the open problem of the existence of higher dimensional Nehari-Sarashon-Page Theorem, see Problem~\ref{prob-d-torus} and Problem~\ref{prob-LF}  and the discussions there. 

 The Hankel operators or Hankel matrices associated with the  system $\{\Gamma_t^\MM\}_{t\in\MM}$ are defined as follows: 
\begin{align}\label{hankel-matrix}
\sum_{t\in \MM} c_t\otimes \Gamma_t^\MM  = [c_{s_1 s_2}]_{s_1, s_2\in \MM}, 
\end{align}
where $c_t$ is either in $\C$ or belongs to a $C^*$-algebra.  In the classical setting of the monoid $\N$, the study of  Hankel operators is closely related to the theory of  one-dimensional discrete time Gaussian processes (see \cite[Chapters 8 and 9]{Peller-book}).   In order to investigate Gaussian processes on abstract monoids, one needs to study positive definite Hankel matrices of the form  \eqref{hankel-matrix} or  similar forms, see \cite{duke-hankel-semigroup, mathz-semigroup, Bernoulli-hankel-semigroup} for details.  

In the case of free monoids (or more general trace monoids), Fliess \cite{Fliess} studied the Hankel matrices  of the form \eqref{hankel-matrix} with coefficients   $c_t$ belonging to  an abstract semiring. These Hankel matrices on free monoids also play a role in quantized calculus for non-commutatative geometry.  Connes \cite[Section 4.5]{Connes} conjectured a non-commutative analogue of the  well-known Kronecker theorem on the relation between finite rank Hankel matrices and rational functions. Duchamp and Reutenauer \cite{invention-connes} confirmed  Connes' conjecture and their proof is based on Hankel matrices  on the free monoids (see \cite[Thm. 2.1.6]{book-control}).  See also \cite{MIYAGAWA2022109609} for similar phenomenon  of free semicircular elements.

In this paper, the Hankel systems on monoids are further generalized to the non-classical ones determined by level sets of abstract two-variable maps. More precisely, given  three sets $\mathcal{A}, \mathcal{X}, \mathcal{L}$ and a  map $\Phi: \mathcal{A}\times \mathcal{X}\rightarrow \mathcal{L}$, we consider the non-classical  Hankel system   $\{\Gamma_\ell^\Phi\}_{\ell\in \mathcal{L}}$  with $\Gamma_\ell^\Phi$ the matrices defined by
\begin{align}\label{non-cl-hankel}
\Gamma_\ell^\Phi(a, x)= \mathds{1}(\Phi(a, x)= \ell), \quad (a, x)\in \mathcal{A}\times \mathcal{X}. 
\end{align}
The associated Hankel matrices are then defined as
\begin{align}\label{g-h-M}
\sum_{\ell\in \mathcal{L}} c_\ell \otimes \Gamma_\ell^\Phi = [c_{\Phi(a,x)}]_{a\in \mathcal{A}, \, x\in \mathcal{X}},
\end{align}
where $c_\ell$ is either in $\C$ or belongs to a $C^*$-algebra.

 The consideration of such non-classical Hankel systems $\{\Gamma_\ell^\Phi\}_{\ell\in \mathcal{L}}$ is not merely a pure mathematical curiosity,   but rather  it is essential for proving the hereditary self-absorption property  (Definition~\ref{def-hsap})  for the Hankel system on  the simplest monoid $\N$, for the big Hankel system on the higher dimensional torus $\T^d$ or the infinite dimensional torus $\T^\N$  and for the regular representation system of any discrete group (see Corollaries \ref{cor-comp}, \ref{cor-gp-compression} and  \ref{cor-big-hankel}).  In particular,  our result implies a recent result of Katsoulis \cite[Cor. 5.3]{Katsoulis:2023aa}.  Note  also that the Hankel matrices studied  in \cite{duke-hankel-semigroup} are  associated to the monoids with an involution and they  are of the form  \eqref{g-h-M} for a suitable map $\Phi$.

For simplicity, in the remaining part of the introduction, we mainly concentrate on the Hankel systems on monoids. The results concerning non-classical hankel systems will be given in \S \ref{sec-hank-monoid}.

{\flushleft \bf I.  Self-absorption property.}
The following definition is inspired by Pisier's operator Hilbert space $OH$.
 Here we recall that $OH$ is the Hilbert space $\ell^2= \ell^2(\N)$ equipped with the unique  {\it operator space structure} (abbrev. \oss,  see  \cite[Chapters 1-5]{Pisier-Operator-space-book} and \cite{Ruan-JFA-Subspces, Ruan-book} for the background of \oss) such that  for  any orthonormal vectors $(v_i)_{i\in I}$ in $OH$ and  any finitely supported sequence $(a_i)_{i\in I}$ in any  $C^*$-algebra $A$, 
\begin{align}\label{def-OH}
\Big\| \sum_{i\in I} v_i \otimes a_i\Big\|_{OH\otimes_{min} A}  = \Big\| \sum_{i\in I} a_i \otimes \overline{a}_i\Big\|_{A\otimes_{min}\overline{A}}^{1/2},
\end{align}
where $\otimes_{min}$ denotes the standard minimal tensor product and $\overline{A}$ denotes the complex conjugate of the $C^*$-algebra $A$ (see  \cite[Section 2.3]{newbook-pisier}).

\begin{definition}[Self-absorption property]\label{def-self-tensor}
Given $\kappa \ge 1$. A family (also called a system) $\mathcal{F}=\{x_i\}_{i\in I}$  in a $C^*$-algebra $A$ or in an operator space is said to have  $\kappa$-self-absorption property (abbrev. $\kappa$-SAP) if the following inequalities hold for any finitely supported family $\{b_i\}_{i\in I}$ in  any other $C^*$-algebra $B$:  
\begin{align}\label{def-kappa}
 \frac{1}{\kappa} \Big \| \sum_{i\in I}  b_i \otimes x_i \Big \|_{B\otimes_{min} A}  \le  \Big \| \sum_{i\in I}  b_i \otimes x_i \otimes \overline{x}_i\Big \|_{B\otimes_{min} A \otimes_{min} \overline{A}}  \le    \kappa \Big\| \sum_{i\in I}  b_i \otimes x_i  \Big\|_{B\otimes_{min} A}. 
\end{align}
For simplicity, we denote $1$-SAP just by SAP. 
\end{definition}

\begin{definition}[Lunar monoids]\label{def-AAM}
We say that a monoid $\MM$ is a {\it lunar monoid} (or we say $\MM$ is lunar) if it is cancellative and satisfies the algebraic  lunar condition :
\begin{align}\label{def-lunar-intro}
\big(ax  = by, cx=dy, az=bw  \,\, \text{in $\MM$}\big)\Longrightarrow \big(cz=dw \,\, \text{in $\MM$}\big). 
\end{align}
\end{definition}

Our terminology ``lunar condition" comes from the theory of group-embeddability of monoids (\cite{Malcev-I, Malcev-II} and \cite{Lambk-immersibility}).  Indeed, the condition \eqref{def-lunar-intro} is the first-order one of a series of  conditions (called lunar conditions there) in Jackson's  thesis with his thesis supervisor Halperin (see \cite{Bush-George-TAMS}).

\begin{mainthm}\label{thm-lunar-monoid}
For any lunar monoid $\MM$, the Hankel system $\{\Gamma_t^\MM\}_{t\in \MM}$  has SAP. 
\end{mainthm}

Theorem~\ref{thm-lunar-monoid} follows from  a stronger result in Theorem~\ref{thm-lunar-sap}, which states that any lunar monoid induces  a Hankel system with the stronger {\it hereditary SAP} (see Definition~\ref{def-hsap}).   With  extra efforts,   our formalism can be adapted to operators on $\ell^p$-spaces or on more general symmetric sequence-spaces. 

Note that the lunar condition of a monoid is a {\it local  algebraic condition}: to check whether  a monoid  is lunar or not, it suffices to study the configurations of all $4\times 4$ blocks in its multiplication table. Theorem~\ref{thm-lunar-monoid} means that,  such local algebraic  condition of a monoid implies  a global analytic  property --the SAP--for the associated Hankel system.  Furthermore,  we conjecture that, the above local algebraic condition is equivalent to the global analytic one.  See Problem~\ref{prob-sap-lunar}  and the more tractable  Problem~\ref{prob-hsap-lunar}, Problem~\ref{prob-hsap}  for related discussions. 

It can be checked directly that the lunar condition holds true for all  group-embeddable monoids. Therefore, we obtain for instance the SAP for the Hankel systems associated with the following monoids: the free monoids (submonoids of  free groups), the submonoid $\mathrm{SL}_d(\N)$ of $\mathrm{SL}_d(\Z)$ defined by
\[
\mathrm{SL}_d(\N): = \big\{[a_{ij}]_{1\le i,j\le d}\in \mathrm{SL}_d(\Z)\big| a_{ij} \in \N\big\}. 
\]
Note also that, any cancellative Abelian monoid is  lunar, since it can be embedded into its Grothendieck group (see, e.g., \cite[Section 2.A, page 52]{Bruns-Polytopes-K-theory}). 
\begin{corollary}\label{cor-abelian-sap}
The Hankel system on any cancellative Abelian monoid  has SAP. 
\end{corollary}

 As a consequence, we obtain the SAP (and in fact the much stronger hereditary SAP in the sense of Definition~\ref{def-hsap}) for  both the system of small Hankel operators and that of big Hankel operators on the higher dimensional torus  $\T^d$ ($d\ge 1$), as well as on  the infinite dimensional torus $\T^\N$.  Note that, by the unique factorization integers theorem, the  small Hankel operators on $\T^\N$  can be transformed to multiplicative Hankel operators associated to the multiplicative monoid  of positive integers $\N^* = \N\setminus \{0\}$ (see \cite{Helson2010,Ortega-2012,Brevig-2021}).

It should be noted that  Theorem~\ref{thm-lunar-monoid} goes beyond the class of group-embeddable monoids. In fact,  Bush  \cite{Bush-George-TAMS}  showed that the lunar condition is not sufficient for the group-embeddability.

Theorem~\ref{thm-lunar-monoid} holds  in  a much more general abstract setting (see Theorem~\ref{thm-Phi} below). Namely, by introducing a key algebraic notion of {\it lunar maps}  (see Definition~\ref{def-SD}),  we obtain SAP for non-classical Hankel systems \eqref{non-cl-hankel} associated with any   lunar map. Here we only mention two concrete such examples:
\begin{itemize}
\item Arbitrary spatial compression  of the regular representation system of any discrete group. See Corollaries~\ref{cor-comp} and \ref{cor-gp-compression}. 
\item  The Hankel system $\{\Gamma_\ell^\Phi\}$  determined by the level sets of any rational map $\Phi(x,y)=a x^m  + b y^n$ with $a,b,m,n\in \mathbb{Z}^* = \Z\setminus \{0\}$: 
\[
\Gamma_\ell^{\Phi}(x, y)= \mathds{1}(a x^m + b y^n= \ell), \quad x, y\in \mathbb{N}^* = \N\setminus\{0\}, \, \ell\in \Phi (\mathbb{N}^*\times \mathbb{N}^*).
\]
See Example~\ref{ex-polynomial} for details.  It seems to be a natural problem of determine all the lunar rational maps or lunar polynomial maps, see Problem~\ref{prob-lunar-poly} in \S \ref{sec-lunar-map}.
\end{itemize}

We conjecture that lunar condition on a map is equivalent to the SAP of its associated Hankel system, see Problem~\ref{prob-hsap}.  For a concrete example of non-lunar map inducing a non-SAP Hankel system, see the map induced by a multicolor checkerboard in Example~\ref{ex-3times3}.

\medskip

{\flushleft \bf II.  Harmonic analysis applications of SAP.}   In the classical setting of the additive monoid of non-negative integers $\N = \{0, 1, 2, \cdots\}$, as well as all the more general free Abelian monoids $\N^d$, the self-absorption property is  applied to the study of completely bounded Fourier multipliers between Hardy spaces (see Section \ref{sec-app}). The reader is refered to \cite{arnold2022s1bounded,arnold2023hankel,Parcet-annals} for recent works about completely bounded multipliers.      

All notations below are classical and will be  recalled in \S\ref{sec-notation}. All subspaces or quotient spaces of $L^p$ are equipped with their standard  \oss, see  \cite[Chapters 1-5]{Pisier-Operator-space-book} and  \cite[Chapters 1-2]{Pisier-Non-Commutative-vector}.   For simplifying notation, throughout the paper, by writing $L^p, H^p$ etc, we mean $L^p(\T), H^p(\T)$ respectively.   We identify $L^\infty/\overline{H_0^\infty}$ with $\BMOA$.  All spaces of completely bounded Fourier multipliers (such as $(H^p, H^q)_{cb}$,  $(H^p, \BMOA)_{cb}$ etc.) are equipped with their natural \oss 

Based on Grothendieck’s work \cite{Grothendieck-1953}, Pisier \cite{Pisier-similar-cb} obtained a striking description of $(H^1,H^1)_{cb}$ as follows: $\varphi \in (H^1, H^1)_{cb}$ if and only if there exist $v, w \in \ell^\infty(\ell^2)$ such that the  matrix $[\langle v_m, w_n\rangle]_{m,n\in\N}$ is of  Hankel type and 
\begin{align}\label{phi-dec}
\widehat{\varphi}(m+n)  = \langle v_m, w_n\rangle_{\ell^2} \quad \text{for all $m,n\in \N$}. 
\end{align}
   Very recently, in the same spirit, Arnold, Le Merdy and  Zadeh  \cite{arnold2022s1bounded,arnold2023hankel} described  $(H^p(\R),H^p(\R))_{cb}$ for $1\leq p<\infty$. 
However, in general, it seems difficult to determine pairs $(v,w)$  of vectors in $\ell^\infty(\ell^2)$ such that $[\langle v_m, w_n\rangle]_{m,n\in\N}$  is of Hankel type.  Therefore,  it would be interesting to give an effective method to determine when $\varphi$ satisfies \eqref{phi-dec}.

On the other hand, the question of characterizing $(H^p, H^q)_{cb}$ seems open  for any other non-trivial distinct pairs of $(p,q)$ for $1\leq p, q<\infty$. 

To state our new findings, we introduce analytic function spaces  $\BMOA^{(p)}$ on $\T$ as follows. For any $1\le p<\infty$, define 
\begin{align}\label{def-bmoap}
\BMOA^{(p)}: = \Big\{\varphi = \sum_{n=0}^\infty \widehat{\varphi}(n) e^{in\theta}\Big|  \| \varphi\|_{\BMOA^{(p)}} = \Big\| \sum_{n=0}^\infty |\widehat{\varphi}(n)|^p e^{in\theta}\Big\|_{\BMOA}^{1/p} <\infty \Big\}.
\end{align}
By a well-known result of Fefferman on $\BMOA$ functions with non-negative Fourier coefficients (see, e.g., \cite[Thm. 6.4.3]{Mulitipier-Fefferman-condition}),
\[
\| \varphi\|_{\BMOA^{(p)}}\approx   |\widehat{\varphi}(0)| +  \sup_{n\ge 1} \Big\{ \sum_{k=1}^\infty  \Big(\sum_{kn\le j < (k+1)n} | \widehat{\varphi}(j)|^p \Big)^2\Big\}^{1/2p}. 
\]
By convention, we set 
\begin{align}\label{def-bmoa-inf}
\BMOA^{(\infty)}: = \Big\{\varphi = \sum_{n=0}^\infty \widehat{\varphi}(n) e^{in\theta}\Big|  \| \varphi\|_{\BMOA^{(\infty)}} =  \sup_{n\ge 0}|\widehat{\varphi}(n)|<\infty\Big\}.
\end{align}

\begin{corollary}[see Proposition \ref{thm-hardy} for its complete isometric version]
We have the following isometric isomorphisms:
\[
(H^1, \mathrm{BMOA})_{cb} = \mathrm{BMOA}  \an (H^1, H^2)_{cb}  =  (H^2, \mathrm{BMOA})_{cb} = \mathrm{BMOA}^{(2)}.
\]
\end{corollary}

We then prove  in Lemmas \ref{prop-bmoap} and \ref{prop-scale} that,  $\BMOA^{(p)}$ are Banach spaces and the family $\{\BMOA^{(p)}: 1\le p\le \infty\}$ forms a  complex interpolation scale.   
\begin{corollary}[see Theorem \ref{cor-interpolation}]\label{cor-interpolation}
For any $1\le p\le 2 \le q<\infty$, 
 \begin{align*}
\BMOA^{(u)} \xhookrightarrow[\textit{inclusion}]{\textit{bounded}} &  (H^p, H^q)_{cb} \quad \text{with $\frac{1}{u} = \frac{1}{p} - \frac{1}{q}$}. 
\end{align*} 
\end{corollary}

Further applications are  as follows. 

\begin{itemize}
\item[1)] Fourier-Schur multiplier inequalities with critical exponent $4/3$ (Corollary \ref{cor-FS-mul} and Proposition \ref{prop-43-optimal}): For any  $\ell^2$-column-vector-valued function $f\in H^1(\ell^2)$,
\[
\Big\| \Big(\int_\T \big[(f*\varphi)  (f*\varphi)^*\big] dm \Big)^{1/2} \Big\|_{S_{4/3}} \le   \| \varphi \|_{\BMOA^{(2)}} \cdot \| f\|_{H^1(\ell^2)}. 
\]
The exponent $4/3$ is critical: it can not be replaced by any exponent $p< 4/3$. 
\item[2)]  Exact  complete bounded norm of the Carleman embedding (see Corollary~\ref{cor-carleman} for the one-dimensional case and Corollary~\ref{cor-higher-carleman} for its higher dimensional analogue):  For any $f\in H^1(S_1)$ and $g \in  \C[z] \otimes S_\infty$ (the constant $\sqrt{\pi}$ below is optimal), 
\[
\Big| \int_\D \Tr \big(f(z) \overline{g(z)}\big) dA(z) \Big| \le \sqrt{\pi} \| f\|_{H^1(S_1)}  \Big\|  \int_\D  \big[ g(z) \otimes \overline{g(z)}  \big] dA(z) \Big\|_{B(\ell^2\otimes_2 \ell^2)}^{1/2}. 
\]
\item[3)] Failure of hyper-complete-contractivity for  the Poisson semigroup (Corollary \ref{cor-hyper}): Consider the  Poisson convolution $P_r$  on $\T$.    Janson's hyper-contractivity \cite{Janson-Poisson-hypercontractivity} states that $P_r$ is contractive from $H^1$ to $H^2$ if and only if $r\le \sqrt{1/2}$.  However,  
\[
\text{$\| P_r\|_{(H^1, H^2)_{cb}} = (1-r^4)^{-1/2}>1$ for any $r\in (0, 1)$}.
\]   
\item[4)]An inequality for Hankel operators (see Corollary~\ref{cor-S4-hankel} for its higher dimensional analogue):  Let $\varphi\in H^1$ and set $\varphi^{\dag}(e^{i\theta}) = \overline{\varphi(e^{-i \theta}})$, then for any mutually orthogonal functions  $(f_k)_{k\ge 1}$  in $H^2$, 
\[
\big\|\sum_{k=1}^\infty (\Gamma_{f_k*\varphi})^*  \Gamma_{f_k*\varphi} \big\|_{B(\overline{H^2})}^{1/2} \le  \|\Gamma_{\varphi*\varphi^\dag} \|_{\hank(\overline{H^2}, H^2)}^{1/2} \big(\sum_{k=1}^\infty \| f_k\|_{H^2}^4\big)^{1/4},  
\]
where $\hank(\overline{H^2}, H^2)$ denotes the space of Hankel operators and $\Gamma_\psi$ denotes the corresponding Hankel operator with symbol $\psi$ (see \S \ref{sec-nsp} for the notation). 
\end{itemize}

\medskip

{\flushleft \bf  III. Coupled foliations.} The main ingredient in the proof of Theorem~\ref{thm-lunar-monoid} is the construction and analysis of the {\it coupled foliations} illustrated in Figure~\ref{fig-foliation} for the following two subsets in $\MM^2$ (in Abelian case, they both equal $\MM^2$, the roles of these two subsets will be explained soon): 
\begin{align*}
&(\MM^2)_{\mathrm{left}\,} \,: = \{(x,y)\in \MM^2:\exists (a,b), \, s.t.\,\, ax = by\};
\\
&(\MM^2)_{\mathrm{right}}:  = \{(a,b)\in \MM^2:\exists (x,y), \, s.t.\,\, ax = by\}.
\end{align*}
The lunar condition  \eqref{def-lunar-intro} is crucial in our construction. 
\begin{figure}[h]
\begin{center}\label{fig-foliation}
\includegraphics[width=0.5\linewidth]{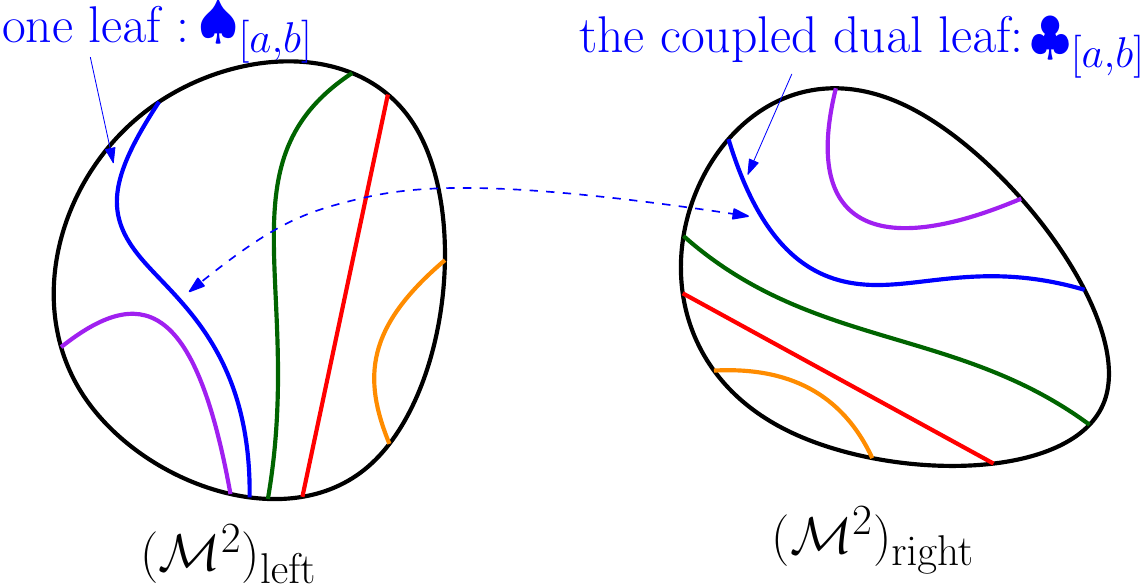}
\caption{The coupled foliation structures.}
\end{center}
\end{figure}
  Under this lunar condition, we may define an equivalence relation on $(\MM^2)_{\mathrm{right}}$ by 
\[
(a,b)\sim (c,d) \stackrel{def}{\Longleftrightarrow} \exists (x,y), \,s.t.\, ax = by \,\, \&\, \, cx=dy. 
\]
Denote by  $[\MM^2]_{\mathrm{right}} = (\MM^2)_{\mathrm{right}}/\!{\sim}$ the quotient space and by   $[a,b] \in [\MM^2]_{\mathrm{right}}$ the  equivalence class of  $(a,b) \in (\MM^2)_{\mathrm{right}}$. Then   the coupled foliation decompositions are given by 
\begin{align}\label{coup-fol}
(\MM^2)_{\mathrm{left}} = \bigsqcup_{[a,b]\in [\MM^2]_{\mathrm{right}}} \spadesuit_{[a,b]} \an (\MM^2)_{\mathrm{right}} = \bigsqcup_{[a,b]\in [\MM^2]_{\mathrm{right}}} \clubsuit_{[a,b]}
\end{align}
with  $\spadesuit_{[a,b]}: = \{(x,y)\in \MM^2: ax  = by\}$
 and $\clubsuit_{[a,b]}: =  \{(c,d)\in \MM^2:  (c,d)\sim (a,b)\}$.  The lunar condition ensures that the definition of $\spadesuit_{[a,b]}$ depends only on the equivalence class of $(a,b)$.

A key property shared by all these leaves $\spadesuit_{[a,b]}, \clubsuit_{[a,b]}$ is the ``{\it one-dimensionality}". More precisely,  let $\pi_1, \pi_2$ be the usual coordinate  projections on $\MM^2$,  then  the lunar condition combined with the cancellativity implies that  all the following maps 
\[
\pi_1|_{\spadesuit_{[a,b]}}, \quad \pi_2|_{\spadesuit_{[a,b]}}, \quad  \pi_1|_{\clubsuit_{[a,b]}}, \quad \pi_2|_{\clubsuit_{[a,b]}}
\]
 are injective  for any $[a,b]\in [\MM^2]_{\mathrm{right}}$. 

\begin{remark}
Consider the dual  equivalence relation on $(\MM^2)_{\mathrm{left}}$ defined by:   $(x,y) \sim' (z,w)$ if and only if there exists $(a,b)$ such that $ax = by \, \&\, az=bw$.   The relation $\sim'$ induces a natural decomposition of $(\MM^2)_{\mathrm{left}}$. One can show that,  under the lunar condition on $\MM$, this decomposition   coincides with  the one given in \eqref{coup-fol}. However, it should be emphasized that, the one-to-one correspondance given by  \eqref{coup-fol} between the leaves in $(\MM^2)_{\mathrm{left}}$ and $(\MM^2)_{\mathrm{right}}$  is important for us. 
\end{remark}

{\flushleft \bf IV.  How coupled foliations  lead to SAP ? }  The central step to SAP is the   {\it simultaneous block-diagonalizations} of 
 operators $\Gamma_t^\MM\otimes \Gamma_t^\MM$ for all $t\in \MM$.  Such simultaneous block-diagonalizations will be explicitly realized by the coupled foliations in \eqref{coup-fol}. 

Now, we explain  the roles of $(\MM^2)_{\mathrm{left}}$ and $(\MM^2)_{\mathrm{right}}$: they induce two subspaces $\ell^2 ((\MM^2)_{\mathrm{left}})$ and  $\ell^2(  (\MM^2)_{\mathrm{right}})$ of $\ell^2(\MM^2)$. Moreover, using the canonical identification of $\ell^2(\MM^2)$ with $\ell^2(\MM)\otimes_2 \ell^2(\MM)$ and hence the canonical  identifications of their subspaces,   one can show that 
\begin{itemize}
\item[(i)] the orthogonal complement  of $\ell^2 ((\MM^2)_{\mathrm{left}})$ in $\ell^2(\MM^2)$ lies in the  kernel of  $\Gamma_t^\MM\otimes \Gamma_t^\MM$ simultaneously for all $t\in \MM$: 
\[
[ \ell^2 ((\MM^2)_{\mathrm{left}})]^\perp = \ell^2(\MM^2)\ominus \ell^2 ((\MM^2)_{\mathrm{left}}) \subset  \bigcap_{t\in \MM}\ker (\Gamma_t^\MM\otimes \Gamma_t^\MM);
\] 
\item[(ii)] the images of $\ell^2(\MM^2)$ under all the maps $\Gamma_t^\MM\otimes \Gamma_t^\MM$  are contained in $\ell^2 ((\MM^2)_{\mathrm{right}})$: 
\[
  \overline{\spann} \big\{ \bigcup_{t\in \MM}\mathrm{Im} (\Gamma_t^\MM\otimes \Gamma_t^\MM)  \big\} \subset  \ell^2 ((\MM^2)_{\mathrm{right}}). 
\]
\end{itemize}

Next we explain how   the coupled foliations \eqref{coup-fol} are used to construct simultaneous block-diagonalization of all $\Gamma_t^\MM\otimes \Gamma_t^\MM$. 
More precisely, the coupled foliations induce the following coupled orthogonal decompositions: 
\[
\ell^2 ((\MM^2)_{\mathrm{left}})  =  \bigoplus_{[a,b]\in [\MM^2]_{\mathrm{right}}}  \ell^2(\spadesuit_{[a,b]})  \an \ell^2 ((\MM^2)_{\mathrm{right}})  = \bigoplus_{[a,b]\in [\MM^2]_{\mathrm{right}}} \ell^2(\clubsuit_{[a,b]}). 
\]
We further show that
\[
(\Gamma_t^\MM\otimes \Gamma_t^\MM) (\ell^2(\spadesuit_{[a,b]})) \subset \ell^2(\clubsuit_{[a,b]}), \quad \text{for all $t\in \MM$. }
\]

Finally,  by the ``one-dimensional structure" of  all the leaves in the foliations \eqref{coup-fol}, for  each pair of the   coupled leaves $(\spadesuit_{[a,b]}, \clubsuit_{[a,b]})$, the family  $\{(\Gamma_t^\MM\otimes \Gamma_t^\MM)|_{ \ell^2(\spadesuit_{[a,b]})}\}_{t\in \MM}$ is shown to  satisfy  the following 
 simultaneous commutative diagram (as before, here `` simultaneous" means that the contractive interwining operators $R^{[a,b]}$ and $S^{[a,b]}$ only depend on the equivalence class $[a,b]\in [\MM^2]_{\mathrm{right}}$ and are both independent of $t\in \MM$): 
\begin{equation*}
\begin{tikzcd}
\ell^2(\spadesuit_{[a,b]}) \arrow{r}{  \Gamma_t^\MM \otimes \Gamma_t^\MM} \arrow[d, "R^{[a,b]}" swap]& \ell^2(\clubsuit_{[a,b]})
\\
\ell^2(\MM) \arrow{r}{\Gamma_t^\MM} & \ell^2(\MM)  \arrow[u, "S^{[a,b]}" swap]
\end{tikzcd},
\end{equation*}
with  $R^{[a,b]}, S^{[a,b]}$ two explicitly constructed operators such that  
\[
\| R^{[a,b]} \| \le 1 \an \| S^{[a,b]}\|\le 1.
\] 

As a result,  all operators $\Gamma_t^\MM\otimes \Gamma_t^\MM$ are  simultaneously  block-diagonalized,  with notably  simplified diagonal-blocks. In notation, we establish the decomposition
\[
\Gamma_t^\MM \otimes \Gamma_t^\MM =  \Big([ \ell^2 ((\MM^2)_{\mathrm{left}})]^\perp \xrightarrow{0} 0\Big) \oplus \bigoplus_{[a,b]\in [\MM^2]_{\mathrm{right}}}   \Big( \ell^2(\spadesuit_{[a,b]})  \xrightarrow{S^{[a,b]}  \Gamma_t^\MM   R^{[a,b]}} \ell^2(\clubsuit_{[a,b]}) \Big). 
\]

\medskip

{\flushleft \bf V.  Open problems and further discussions.} The SAP for the Hankel system on $\N^d$ established in this paper (Corollary \ref{cor-abelian-sap}) is closely  related to the well-known open problem on the existence of a higher dimensional  Nehari-Sarason-Page Theorem  on $\T^d$ ($d\ge 2$). Indeed, by  the SAP for the Hankel system on $\N^d$,  a positive answer to the following  Problem~\ref{prob-d-torus} would follow from   the existence of such a higher dimensional NSP theorem. 

\begin{problem}\label{prob-d-torus}
Let $d\ge 2$ be an integer and consider  the operator space $L^\infty(\T^d)/\overline{H_0^\infty(\T^d)}$ (equipped with its natural  quotient \oss).   Does the family  $\{e^{in \cdot \theta}\}_{n\in \N^d} \subset L^\infty(\T^d)/\overline{H_0^\infty(\T^d)}$ have $\kappa$-SAP for a finite $\kappa$ ?  
\end{problem}

We say that a  multiplier  $T$ from  $H^1(\T^d)$ to  $L^\infty(\T^d)/\overline{H_0^\infty(\T^d)}$  is {\it liftable} if it satisfies the commutative diagram: 
\begin{equation*}
\begin{tikzcd}
L^1 (\T^d) \arrow{r}{K}& L^\infty (\T^d) \arrow[d, "Q" swap, "\textit{quotient}"]
\\
H^1(\T^d) \arrow{r}{T}\arrow[u, "\textit{inclusion}" swap, "\mathcal{I}"]&  L^\infty(\T^d)/\overline{H_0^\infty(\T^d)} 
\end{tikzcd},
\end{equation*}
where $K$ is a bounded convolution operator. 

Using similar arguments in the proof of Corollary \ref{cor-hardy-one-inf} for the one-dimensional case,  one can show that Problem~\ref{prob-d-torus} is  equivalent to  the following 

\begin{problem}\label{prob-LF}
 Is  every  multiplier   in  $(H^1(\T^d), L^\infty(\T^d)/\overline{H_0^\infty(\T^d)})_{cb}$  liftable ? 
\end{problem}

\begin{definition}[SAP monoids]\label{def-SAP-monoid}
A monoid is called a {\it SAP monoid} if its associated Hankel system  has SAP.  In such situation, we also say that the monoid has SAP. 
\end{definition}
The  SAP of a monoid can be seen as an analytic property (the operator space structure) of an algebraic object (the associated Hankel system), while the  lunar condition for a monoid  is a pure algebraic property. Theorem~\ref{thm-lunar-monoid} says that,  this  algebraic property implies the  analytic one.  

\begin{problem}\label{prob-sap-lunar}
Does the class of SAP monoids coincide with that of lunar ones ? 
\end{problem}

 It is not even known to us whether SAP is preserved by taking submonoids. 

\begin{problem} Can  SAP always be inherited by submonoids ? 
\end{problem}

The following table provides informations of some classical  operations  preserving the class of lunar monoid and that of   SAP monoid respectively.   

\begin{center}
\begin{tabular}{|c|c|c|c|}
\hline
    &  Cartesian product &  free product   &   submonoid  \\
\hline
lunar monoid (algebraic property)  & $\checkmark$ & $\checkmark$ & $\checkmark$
\\
\hline
SAP monoid (analytic property) & $\checkmark$ &  $\checkmark$ \footnote{The proof that free product preserves the class of SAP monoids requires much more technical details, so  we leave it to a separate work.} & ? 
\\
\hline
\end{tabular}
\end{center}

We believe that not all cancellative monoids have SAP. However, at this moment, we are not aware of such examples. 

\begin{problem} 
Do there exist non-SAP cancellative monoids ? 
\end{problem}

{\flushleft\bf Acknowledgements.} This work is supported by the National Natural Science Foundation of China (No.12288201).  YH is supported by the grant NSFC 12131016 and 12201419,  ZW is supported by NSF of Chongqing (CSTB2022BSXM-JCX0088,
CSTB2022NSCQ-MSX0321) and FRFCU (2023CDJXY-043).

\section{The simplest SAP monoid: an illustration}\label{sec-one-dim}

To illustrate our main idea, we include an algebraic proof of the SAP for the simplest additive monoid $(\N, +)$ of non-negative integers. See Proposition~\ref{thm-doubling} and Corollary~\ref{cor-simple-n}. 

By identifying the Hankel operators on $\ell^2(\N)$ with the standard Hankel operators on the Hardy spaces of $\T$, one can prove the SAP for the Hankel system on $\N$ in a pure analytic method by  applying the celebrated Nehari-Sarason-Page Theorem (which is briefly recalled in \S \ref{sec-nsp}).  

However, the Nehari-Sarason-Page Theorem is not available in general situations.  Our algebraic proof is self-contained  and can be easily adapted to all cancellative Abelian monoids and more general non-Abelian monoids (see \S \ref{sec-real-hankel-monoid} and \S \ref{sec-hank-monoid}).  Moreover,  even in the setting of $\N$,   this algebraic  approach leads to a stronger hereditary SAP  (see Definition~\ref{def-hsap}, Theorem~\ref{thm-lunar-sap} and Corollary~\ref{cor-comp})  and it can even be adapted to operators on $\ell^p$-spaces or on more general symmetric sequence-spaces, both of which clearly do not follow from the Nehari-Sarason-Page Theorem.

\subsection{The simplest SAP monoid $(\N,+)$} Throughout the paper, $B(\mathscr{H}, \mathscr{K})$ denotes the bounded linear operators  from a Hilbert space $\mathscr{H}$ to another one $\mathscr{K}$ and $B(\mathscr{H}) = B(\mathscr{H}, \mathscr{H})$ and  $\mathscr{H}\otimes_2\mathscr{K}$ denotes the Hilbertian tensor product of $\mathscr{H}$ and $\mathscr{K}$.

 For simplifying notation, we always write $\ell^2  = \ell^2(\N)$. 
The standard Hankel matrix associated to  a complex sequence  $a= (a_i)_{i\ge 0} \in \C^\N$ is defined by 
\[
\Gamma_a = [a_{i+j}]_{i,j\ge 0}. 
\]  
Let $\hank(\ell^2)$ denote the collection  of all bounded $\Gamma_a \in B(\ell^2)$.
More generally, for any sequence $(x_n)_{n\ge 0}$ in $B(\mathscr{H})$, we define a block Hankel matrix (which a priori does not represent an element in $B(\mathscr{H}\otimes_2 \ell^2)$): 
\[
\Gamma_x  = \Gamma_{(x_n)_{n\in\N}}: = [x_{i+j}]_{i,j\ge 0}.
\]
Informally, we may write 
$
\Gamma_x = \sum_{n=0}^\infty x_n \otimes \Gamma_n, 
$
where  the operator $\Gamma_n$ (for any non-negative integer $n\ge 0$) is defined by 
\begin{align}\label{def-gamma-n}
\Gamma_n(i,j)= \ch(i+j = n), \quad i, j \ge 0. 
\end{align}
In particular, for any $\Gamma_a \in \hank(\ell^2)$, define a block Hankel matrix by the informal series: 
\[
J_2(\Gamma_a) : =   \Gamma_{(a_n \Gamma_n)_{n\in\N}}  = \sum_{n=0}^\infty a_n \Gamma_n \otimes \Gamma_n. 
\]

\begin{proposition}\label{thm-doubling}
Let  $\Gamma_a \in \hank(\ell^2)$. Then $J_2(\Gamma_a)\in B(\ell^2\otimes_2\ell^2)$.  Moreover,   the linear map $J_2: \hank(\ell^2) \rightarrow B(\ell^2\otimes_2 \ell^2)$ is a complete isometric embedding. 
\end{proposition}

\begin{corollary}\label{cor-simple-n}
The monoid $(\N,+)$ has SAP. 
\end{corollary}

\begin{remark}
Let $C^*(\{\Gamma_n\}_{n\in\N}) \subset B(\ell^2)$ (resp. $C^*(\{\Gamma_n \otimes\}_{n\in\N}) \subset B(\ell^2\otimes_2\ell^2)$)  denote the $C^*$-algebra generated by the family $\{\Gamma_n\}_{n\in\N}$ (resp. $\{\Gamma_n\otimes \Gamma_n\}_{n\in\N}$). Then one may check that the map $\Gamma_n\mapsto \Gamma_n\otimes \Gamma_n (\forall n\in\N)$ can not be extended to a $C^*$-representation from $C^*(\{\Gamma_n\}_{n\in\N})$ to $C^*(\{\Gamma_n\otimes \Gamma_n\}_{n\in\N})$.  Indeed,     by direct computation, one obtains   
\[ 
\|  2\Gamma_0\otimes \Gamma_0  - 2 \Gamma_1\otimes \Gamma_1    - (\Gamma_1\otimes \Gamma_1)^2\|= 3  > \| 2 \Gamma_0 -  2 \Gamma_1  - \Gamma_1^2 \|=\sqrt{5}.
\]
The same computation also implies that the system $\{Id\} \cup\{\Gamma_n\}_{n\in\N}$ has no SAP since 
\[
\|  2\Gamma_0\otimes \Gamma_0  - 2 \Gamma_1\otimes \Gamma_1    - Id\otimes Id\|= 3  > \| 2 \Gamma_0 -  2 \Gamma_1  - Id \|=\sqrt{5}.
\]
\end{remark}

\subsection{Proof of Proposition \ref{thm-doubling}}

Let $\{e_i: i \in  \N\}$ be the standard orthonormal basis of  $\ell^2$. For any integer  $k \ge 0$, define closed subspaces $\mathcal{H}_k$  and  $\mathcal{H}_{-k}$ of $\ell^2\otimes_2 \ell^2$ as 
\begin{align*}
\mathcal{H}_k = \overline{\spann} \{e_i \otimes e_{i+k}: i\in \N \} \an 
\mathcal{H}_{-k} = \overline{\spann} \{e_{i+k} \otimes e_{i}: i \in \N\}.
\end{align*}
 Clearly, $\mathcal{H}_\ell$ are mutually  orthogonal and 
\begin{align}\label{double-ortho-H}
\ell^2\otimes_2\ell^2 = \bigoplus_{\ell\in \Z} \mathcal{H}_\ell. 
\end{align}
Under the above orthogonal decomposition,  any operator $A \in B(\ell^2\otimes_2\ell^2)$ has  a block-matrix representation: 
\[
A = [A_{\ell, \ell'}]_{\ell, \ell'\in\Z} \,\,\text{\,with\,}\, \,A_{\ell, \ell'} \in B(\mathcal{H}_{\ell'}, \mathcal{H}_\ell). 
\]
In particular, an operator $D \in B(\ell^2\otimes_2\ell^2)$ (resp. $D' \in B(\ell^2\otimes_2 \ell^2)$) is called of block diagonal form (resp. of block anti-diagonal form), if 
\[
D  =
\left[
\begin{array}{ccccc}
\ddots& & & &
\\
&D_{-1}  & &&
\\
&& D_0 &&
\\
&&&D_1&
\\ 
&&&&\ddots
\end{array}
\right], \quad 
D' = \left[
\begin{array}{ccccc}
&&&&\iddots
\\
&&&D_{-1}'&
\\
&&D_0'&&
\\
&D_{1}'&&&
\\
\iddots&&&&
\end{array}
\right], 
\]
where $D_\ell \in B(\mathcal{H}_\ell, \mathcal{H}_\ell)$ and $D_{\ell}'\in B(\mathcal{H}_\ell, \mathcal{H}_{-\ell})$. For simplification, we write 
\begin{align}\label{def-diag-anti}
D = \diag  (D_\ell)_{\ell\in\Z} \an D' =  \adiag(D_\ell')_{\ell\in\Z}.
\end{align}

\begin{lemma}\label{lem-diag}
The subspace $\mathcal{H}_0 \subset \ell^2\otimes_2 \ell^2$ is $\Gamma_n \otimes \Gamma_n$-invariant for all $n\in \N$. Moreover, we have the following commutative diagram:
\begin{equation}\label{diag-cd}
\begin{tikzcd}
\mathcal{H}_0 \arrow{r}{ \qquad \Gamma_n\otimes \Gamma_n\qquad} \arrow[d, "\simeq", "U" swap]& \mathcal{H}_0 \arrow[d, "U", "\simeq" swap]
\\
\ell^2 \arrow{r}{\quad \Gamma_n \quad} & \ell^2 
\end{tikzcd},
\end{equation} 
where $U$ is the unitary operator such that 
\begin{align}\label{def-U}
\begin{array}{cccl}
U: & \mathcal{H}_0 &\rightarrow & \ell^2
\\
 & e_i\otimes e_i &  \mapsto & e_i, \quad i \in \N.
\end{array}
\end{align}
 In other words, for all $n\in \N$, 
\begin{align}\label{diag-interwining}
(\Gamma_n\otimes \Gamma_n)|_{\mathcal{H}_0} =   U^{-1} \Gamma_n U.
\end{align}
\end{lemma}

\begin{proof}
Fix $n\in \N$. By  the definition of  $\Gamma_n$, we have  
\begin{align}\label{gamma-map}
\Gamma_n e_i   = \mathds{1}(i\le n) e_{n-i}, \quad  i \in \N. 
\end{align}
Hence,  for any $i \in \N$, 
\[
 (\Gamma_n \otimes \Gamma_n)  (e_i\otimes e_i) = \mathds{1}(i\le n) e_{n-i} \otimes e_{n-i}.   
\]
It follows that $\mathcal{H}_0$ is $\Gamma_n \otimes \Gamma_n$-invariant.  Now by \eqref{def-U}, for any $i \in \N$, 
\begin{align*}
\Gamma_n U(e_i\otimes e_i)  = &  \Gamma_n e_i = \mathds{1}(i\le n) e_{n-i}, 
\\
U (\Gamma_n \otimes \Gamma_n) (e_i \otimes e_i) & = U (\Gamma_n e_i \otimes \Gamma_n e_i) = \mathds{1}(i\le n) e_{n-i}. 
\end{align*}
The desired equality $\Gamma_n U = U (\Gamma_n \otimes \Gamma_n)$ then follows.  
\end{proof}

For any integer $k\ge 0$, define operators  $W_{\pm k}: \mathcal{H}_{\pm k} \rightarrow \mathcal{H}_0$  and $V_{\pm k}: \mathcal{H}_0 \rightarrow \mathcal{H}_{\pm k}$   as  follows: for any $i\in\N$, 
\begin{align}\label{def-WV}
\left\{ 
\begin{array}{l}
W_k(e_i \otimes e_{i+k}) = e_{i+k} \otimes e_{i+k},
\\
W_{-k}(e_{i+k} \otimes e_{i}) = e_{i+k} \otimes e_{i+k},  
\\
V_k(e_i\otimes e_i) =   e_{i}\otimes e_{i+k},
\\
V_{-k} (e_i\otimes e_i)= e_{i+k}\otimes e_i.
\end{array}
\right.
\end{align}
  Note that $W_{\pm k}$ are isometric embeddings and $V_{\pm k}$ are unitary operators.  In particular,  $W_0 = V_0 = I_{\mathcal{H}_0}$, where $I_{\mathcal{H}_0}$ denotes the identity map on $\mathcal{H}_0$.

\begin{remark}\label{rem-easy}
The operation rules in \eqref{def-WV} are easy to remember. Indeed, these operations share a common rule of increasing  only one of the indices  of a tensor $e_i\otimes e_j$ (even for the operators with negative indices as $W_{-k}$ and $V_{-k}$) in an appropriate way.  
\end{remark}

\begin{lemma}\label{lem-offdiag}
For any $\ell\in \Z$ and any $n\in \N$,  the operator $\Gamma_n\otimes \Gamma_n$ maps  $\mathcal{H}_\ell$ into  $\mathcal{H}_{-\ell}$: 
\[
(\Gamma_n \otimes \Gamma_n)|_{\mathcal{H}_\ell} \in B(\mathcal{H}_\ell, \mathcal{H}_{-\ell}).
\]
 Moreover, we have the following commutative diagram: 
\begin{equation*}
\begin{tikzcd}
\mathcal{H}_\ell \arrow[r, "\Gamma_n \otimes \Gamma_n"]   \arrow[d, "W_\ell" swap]&  \mathcal{H}_{-\ell}
\\
\mathcal{H}_0 \arrow[r, "\Gamma_n\otimes \Gamma_n"] & \mathcal{H}_0  \arrow[u,"V_{-\ell}" swap]
\end{tikzcd}.
\end{equation*}
In other words, 
\begin{align}\label{interwining-VW}
 (\Gamma_n \otimes \Gamma_n)|_{\mathcal{H}_\ell}= V_{-\ell}  \big [ (\Gamma_n\otimes \Gamma_n)|_{\mathcal{H}_0}\big]  W_\ell. 
\end{align}
Consequently, with respect to the orthogonal decomposition \eqref{double-ortho-H},  the operator $\Gamma_n\otimes \Gamma_n$ is of block anti-diagonal form.  According to  the notation \eqref{def-diag-anti}, we can write
\[
\Gamma_n\otimes \Gamma_n =  \adiag \Big ( V_{-\ell}  \big [ (\Gamma_n\otimes \Gamma_n)|_{\mathcal{H}_0}\big]  W_\ell\Big)_{\ell\in\Z}.
\]
\end{lemma}

\begin{proof}
The case $\ell=0$ is trivial.  Now take any positive integer $k \ge 1$. Then for all $i \in \N$,
\begin{align*}
(\Gamma_n \otimes \Gamma_n) (e_i\otimes e_{i+k}) & = \mathds{1}(i+k\le n) e_{n-i} \otimes e_{n-i-k} \in \mathcal{H}_{-k},
\\
(\Gamma_n \otimes \Gamma_n) (e_{i+k} \otimes e_i) &= \mathds{1}(i+k\le n) e_{n-i-k} \otimes e_{n-i} \in \mathcal{H}_{k}.
\end{align*}
It follows  that $(\Gamma_n\otimes \Gamma_n)(\mathcal{H}_\ell) \subset \mathcal{H}_{-\ell}$ for any integer $\ell\in \Z\setminus \{0\}$.   By the definitions of $W_{\pm k}, V_{\pm k}$ in \eqref{def-WV} and  Remark~\ref{rem-easy}, we have 
\begin{align*}
V_{-k} (\Gamma_n\otimes \Gamma_n) W_k   (e_i\otimes e_{i+k}) & =  V_{-k} (\Gamma_n\otimes \Gamma_n)  (e_{i+k}\otimes e_{i+k}) 
\\
&  = V_{-k} (\mathds{1}(i+k\le n) e_{n-i-k}\otimes e_{n-i-k})
\\
& = \mathds{1}(i+k\le n)  e_{n-i}\otimes e_{n-i-k}
\end{align*}
and 
\begin{align*}
V_k(\Gamma_n\otimes \Gamma_n) W_{-k} (e_{i+k}\otimes e_i) &= V_k(\Gamma_n\otimes \Gamma_n) (e_{i+k}\otimes e_{i+k})  
\\
& = V_k ( \mathds{1}(i+k\le n) e_{n-i-k}\otimes e_{n-i-k}) 
\\
& = \mathds{1}(i+k\le n) e_{n-i-k}\otimes e_{n-i}.
\end{align*}
Therefore, for all $i \in \N$, 
\begin{align*}
V_{-k} (\Gamma_n\otimes \Gamma_n) W_k   (e_i\otimes e_{i+k}) &= (\Gamma_n\otimes \Gamma_n) (e_i \otimes e_{i+k}),
\\
V_k(\Gamma_n\otimes \Gamma_n) W_{-k} (e_{i+k}\otimes e_i) & = (\Gamma_n\otimes \Gamma_n) (e_{i+k}\otimes e_i).
\end{align*}
The desired equality \eqref{interwining-VW} follows. 
\end{proof}

Let $\bigoplus_{\ell\in \Z} \mathcal{H}_0$ denote the Hilbert space 
\[
 \bigoplus_{\ell\in \Z} \mathcal{H}_0 = \Big\{ (v_\ell)_{\ell\in\Z}\Big|  \text{$v_\ell\in \mathcal{H}_0$ for all $\ell\in\Z$ and $\sum_{\ell\in\Z} \| v_\ell\|^2 <\infty$} \Big\} 
\]
carrying the natural inner product. The unitary operator $U:   \mathcal{H}_0 \rightarrow \ell^2$ defined in \eqref{def-U} induces a faithful $C^*$-representation:
\begin{align}\label{def-rhoU}
\begin{array}{cccl}
\mathcal{R}_U:  &B(\ell^2) &\xrightarrow[\textit{representation}]{\textit{faithful}}& B\big(\bigoplus_{\ell\in \Z} \mathcal{H}_0\big)
\vspace{2mm}
\\
& T & \mapsto&  \mathcal{R}_U(T)=  (U^{-1} T U)^{\oplus \Z}
\end{array},
\end{align}
where the operator $(U^{-1} T U)^{\oplus \Z}$ is given by $(U^{-1} T U)^{\oplus \Z} ((v_\ell)_{\ell\in\Z})  = ( U^{-1} T U v_\ell)_{\ell\in\Z}$. 

Recall the orthogonal decomposition in \eqref{double-ortho-H}:
$\ell^2\otimes_2\ell^2 = \bigoplus_{\ell\in \Z} \mathcal{H}_\ell.$
\begin{proposition}\label{prop-mul-repn}
For any $n\in \N$, the operator $\Gamma_n\otimes \Gamma_n$ admits the factorization:
\begin{align}\label{double-gamma-dec}
\Gamma_n\otimes \Gamma_n =  \mathcal{V}  \mathcal{F}  \mathcal{R}_U(\Gamma_n)  \mathcal{W}, 
\end{align}
where $\mathcal{V}, \mathcal{W}$ are defined by  (recall the definitions of $W_{\pm k}, V_{\pm k}$ in \eqref{def-WV})
\[
   \bigoplus_{\ell\in \Z} \mathcal{H}_0  \xrightarrow{ \mathcal{V} = \bigoplus_{\ell\in\Z} V_\ell} \bigoplus_{\ell\in\Z} \mathcal{H}_\ell  
 , \quad 
    \bigoplus_{\ell\in\Z} \mathcal{H}_\ell \xrightarrow{\mathcal{W} = \bigoplus_{\ell\in\Z} W_\ell} \bigoplus_{\ell\in \Z} \mathcal{H}_0
\]
and $\mathcal{F}: \bigoplus_{\ell\in \Z} \mathcal{H}_0 \rightarrow \bigoplus_{\ell\in \Z} \mathcal{H}_0$ is the unitary operator defined by
$
\mathcal{F}( (v_\ell)_{\ell\in\Z})= (v_{-\ell})_{\ell\in\Z}$. 
\end{proposition}

\begin{proof}
By Lemmas \ref{lem-diag} and \ref{lem-offdiag}, 
\begin{align*}
 \Gamma_n\otimes \Gamma_n & =  \adiag \Big( V_{-\ell}  \big [ (\Gamma_n\otimes \Gamma_n)|_{\mathcal{H}_0}\big]  W_\ell \Big)_{\ell\in\Z} 
\\
& = \adiag \Big( V_{-\ell}  U^{-1}\Gamma_n U W_\ell \Big)_{\ell\in\Z}. 
\end{align*}
Take an arbitray $\ell_0\in\Z$ and  $\xi \in \mathcal{H}_{\ell_0}$. First of all, 
\[
(\Gamma_n\otimes \Gamma_n)(\xi) = V_{-{\ell_0}} U^{-1} \Gamma_n U W_{\ell_0}(\xi) \in \mathcal{H}_{-{\ell_0}}.
\]
On the other hand, by the definitions of $\mathcal{V}, \mathcal{F}, \mathcal{R}_U$ and $\mathcal{W}$, 
\begin{align*}
[\mathcal{V}\mathcal{F}\mathcal{R}_U(\Gamma_n)\mathcal{W} ] (\xi) & = \Big[\Big( \bigoplus_{\ell\in\Z} V_\ell\Big)  \circ \mathcal{F}  \circ (U^{-1} \Gamma_n U)^{\oplus \Z} \circ  \Big(\bigoplus_{\ell\in\Z} W_\ell\Big) \Big](\xi)
\\
&= \Big[\Big( \bigoplus_{\ell\in\Z} V_\ell\Big)  \circ \mathcal{F} \Big]  (\cdots, 0,  \underbrace{U^{-1} \Gamma_n U W_{\ell_0} (\xi)}_{\text{at the $\ell_0$-th position}}, 0, \cdots) 
\\
 & = \Big( \bigoplus_{\ell\in\Z} V_\ell\Big)  (\cdots, 0,  \underbrace{U^{-1} \Gamma_n U W_{\ell_0} (\xi)}_{\text{at the $(-\ell_0)$-th position}}, 0, \cdots)
\\
& = (\cdots, 0,  \underbrace{ V_{-\ell_0}U^{-1} \Gamma_n U W_{\ell_0} (\xi)}_{\text{at the $(-\ell_0)$-th position}}, 0, \cdots)
\\
&= V_{-{\ell_0}} U^{-1} \Gamma_n U W_{\ell_0}(\xi) \in \mathcal{H}_{-{\ell_0}}.
\end{align*}
Hence 
\[
(\Gamma_n\otimes \Gamma_n)(\xi)  = [\mathcal{V}\mathcal{F}\mathcal{R}_U(\Gamma_n)\mathcal{W} ] (\xi). 
\]
Since $\ell_0\in \Z$ and $\xi \in \mathcal{H}_{\ell_0}$ are chosen arbitrarily, the desired equality  \eqref{double-gamma-dec} follows.  
\end{proof}

Now we are ready to prove Proposition~\ref{thm-doubling}.   We need the following elementary equality.

\begin{lemma}\label{lem-Pr}
For any Hilbert space $\mathscr{H}$ and any bounded sequence $(x_n)_{n\ge 0}$ in $B(\mathscr{H})$,
\begin{align}\label{dr-eq}
\Big\|\sum_{n=0}^\infty x_n \otimes \Gamma_n\Big\|_{B(\mathscr{H} \otimes_2 \ell^2)} = \sup_{0<r<1} \Big\|\sum_{n=0}^\infty  r^n x_n \otimes \Gamma_n\Big\|_{B(\mathscr{H} \otimes_2 \ell^2)} \in [0, \infty].
\end{align}
\end{lemma}

\begin{proof}
For any $r\in (0,1)$, define a block diagonal operator 
$
D_r= \diag(I_{\mathscr{H}}, rI_{\mathscr{H}}, r^2I_{\mathscr{H}}, \cdots)$  on $\mathscr{H} \otimes_2\ell^2 = \ell^2(\mathscr{H})$, where $I_{\mathscr{H}}$ is the identity map on $\mathscr{H}$. Clearly, $\| D_r\|=1$. Hence  
\[
\sup_{0<r<1} \| \Gamma_{(r^n x_n)_{n\in\N}}\|  =  \sup_{0<r<1} \|D_r \Gamma_{(x_n)_{n\in\N}} D_r \| \le \| \Gamma_{(x_n)_{n\in\N}}\|. 
\]
 Conversely,  by  the standard approximation,
\begin{align*}
\|\Gamma_{(x_n)_{n\in\N}}\|   = \sup_{L\ge 0}\,\, \sup_{0< r<1} \| [r^{i+j}x_{i+j}]_{0\le i, j \le L}\|
  \le \sup_{0<r<1}\| [r^{i+j}x_{i+j}]_{i, j \ge 0}\| =  \sup_{0<r<1} \| \Gamma_{(r^n x_n)_{n\in\N}}\|.
\end{align*}
The desired equality then follows. 
\end{proof}

\begin{remark}\label{rem-Pr}
The equality \eqref{dr-eq}  applied to  $(x_n\otimes \Gamma_n)_{n\in\N}$ in $B(\mathscr{H}\otimes_2 \ell^2)$ yields
\[
\Big\|\sum_{n=0}^\infty x_n  \otimes \Gamma_n \otimes \Gamma_n\Big\|_{B(\mathscr{H} \otimes_2 \ell^2 \otimes_2 \ell^2)} = \sup_{0<r<1} \Big\|\sum_{n=0}^\infty  r^n x_n \otimes \Gamma_n \otimes \Gamma_n\Big\|_{B(\mathscr{H}\otimes_2 \ell^2 \otimes_2 \ell^2)}.
\] 
Note that if one side of the above equality is finite, then the sequence $(x_n)_{n\in\N}$ must be bounded. 
\end{remark}

\begin{proof}[Proof of Proposition \ref{thm-doubling}]
Consider the following  linear subspace of $\hank(\ell^2)$ (which is clearly not norm-closed nor norm-dense): 
\[
\hank^{(\infty)}(\ell^2)  = \Big\{ \sum_{n=0}^\infty a_n \Gamma_n\Big| a_n\in \C \an \sum_{n=0}^\infty |a_n|<\infty\Big\} \subset \hank(\ell^2).
\]
By Lemma \ref{lem-Pr} and Remark \ref{rem-Pr}, we only need to show that the restriction of $J_2$ on $\hank^{(\infty)}(\ell^2)$ is completely isometric. By Proposition \ref{prop-mul-repn},  restricted on $\hank^{(\infty)}(\ell^2)$,  
\[
J_2 \Big(\sum_{n=0}^\infty a_n \Gamma_n\Big)  =  \sum_{n=0}^\infty a_n \Gamma_n\otimes  \Gamma_n = \sum_{n=0}^\infty a_n   \mathcal{V}  \mathcal{F}  \mathcal{R}_U(\Gamma_n)  \mathcal{W} =   \mathcal{V}  \mathcal{F}  \mathcal{R}_U\Big(\sum_{n=0}^\infty a_n \Gamma_n\Big)  \mathcal{W}. 
\]
The operator  $\mathcal{R}_U$ defined in \eqref{def-rhoU} is a  $C^*$-representation, hence $\| \mathcal{R}_U\|_{cb}\le 1$. Since 
$\| \mathcal{F}\| \le 1, \| \mathcal{V}\|\le 1$ and $\| \mathcal{W}\| \le 1$, 
we have  (see, e.g., \cite[Thm. 1.6]{Pisier-Operator-space-book})
\[
\|J_2: \hank^{(\infty)}(\ell^2) \rightarrow B(\ell^2 \otimes_2 \ell^2)\|_{cb} \le  \| \mathcal{V} \mathcal{F}\| \| \mathcal{R}_U\|_{cb} \| \mathcal{W}\| \le 1. 
\]
Conversely,  by the equality \eqref{diag-interwining}, 
\[
 \sum_{n=0}^\infty a_n \Gamma_n   =  U \Big[\Big(\sum_{n=0}^\infty a_n \Gamma_n \otimes \Gamma_n\Big)\Big|_{\mathcal{H}_0}  \Big] U^{-1}= U\Big[ J_2   \Big(\sum_{n=0}^\infty a_n \Gamma_n \Big)\Big|_{\mathcal{H}_0}\Big]U^{-1}. 
\]
It follows that the map 
\[
\begin{array}{ccc}
J_2 ( \hank^{(\infty)}(\ell^2) )  & \xrightarrow{J_2^{-1}} &   \hank^{(\infty)}(\ell^2)
\vspace{3mm}
\\
 J_2   \Big(\sum_{n=0}^\infty a_n \Gamma_n \Big) &  \mapsto  & \sum_{n=0}^\infty a_n \Gamma_n 
\end{array}
\]
is completely contractive and we complete the whole proof. 
\end{proof}

\section{Self-absorption property: general results}\label{sec-st}

Recall the notion  of SAP introduced in Definition \ref{def-self-tensor}. 
In this section, we give several general properties of SAP: 
\begin{itemize}
\item SAP is preserved by tensor product (Proposition \ref{prop-tensor}). 
\item  Any system of Boolean operators has a restricted-version SAP (Proposition \ref{prop-positive}). 
\item In general,  SAP is neither preserved by direct sum (Example \ref{ex-overlap}) nor  preserved by spatial compression (Example \ref{ex-compression}).
\item The standard Fell's absorption principle implies that the regular representation system of any group has SAP (Example~\ref{ex-regular}).  However, it is quite non-trivial to prove that any spatial compression of the regular representation system of a group also has SAP (Corollary~\ref{cor-gp-compression}). Note that our result implies in particular a recent result of Katsoulis \cite[Cor. 5.3]{Katsoulis:2023aa}. 
\end{itemize}

\subsection{Tensor product preserves SAP}
Given two families $\mathcal{F}_1 = \{x_i\}_{i\in I}$ and $\mathcal{F}_2 = \{y_j\}_{j\in J}$ in two $C^*$-algebras $A_1$ and $A_2$ respectively, slightly abusing the notation $\otimes_{min}$, we denote
\[
\mathcal{F}_1 \otimes_{min} \mathcal{F}_2  = \{x_i \otimes y_j\}_{ (i,j)\in I\times J} \subset A_1 \otimes_{min} A_2. 
\]

Denote by $\kappa_{SA}(\mathcal{F})$ the best constant in the defining inequality  \eqref{def-kappa} of $\kappa$-SAP of a family $\mathcal{F}$. Hence $\mathcal{F}$ has SAP if and only if $\kappa_{SA}(\mathcal{F}) = 1$.

\begin{proposition}\label{prop-tensor}
For any two families $\mathcal{F}_1 \subset A_1$ and $\mathcal{F}_2\subset A_2$ in $C^*$-algebras $A_1$ and $A_2$, 
\[
\kappa_{SA}(\mathcal{F}_1\otimes_{min} \mathcal{F}_2)\le \kappa_{SA}(\mathcal{F}_1) \kappa_{ST}(\mathcal{F}_2).
\]
In particular, if both $\mathcal{F}_1$ and $\mathcal{F}_2$ have SAP, then so does $\mathcal{F}_1\otimes_{min} \mathcal{F}_2$. 
\end{proposition}

\begin{proof}
Write $\mathcal{F}_1 = \{x_i\}_{i\in I}$ and $\mathcal{F}_2 = \{y_j\}_{j\in J}$. 
Take any finitely supported sequence $\{b_{i,j}\}$ in a $C^*$-algebra $B$, by using the canonical $C^*$-isomorphism between  $C_1\otimes_{min} C_2$ and $C_2\otimes_{min}C_1$ for any given $C^*$-algebras $C_1, C_2$, we have  (the norms are taken in the corresponding $C^*$-algebras)
\begin{align*}
& \Big\| \sum_{i\in I, j\in J}b_{i,j} \otimes x_i \otimes y_j \otimes \overline{x}_i\otimes \overline{y}_j\Big\| =  \Big\| \sum_{j\in J} \Big(\sum_{i\in I} b_{i,j} \otimes x_i  \otimes   \overline{x}_i\Big)\otimes y_j \otimes \overline{y}_j\Big\| 
\\
\le&  \kappa_{SA}(\mathcal{F}_2) \Big\| \sum_{j\in J} \Big(\sum_{i\in I} b_{i,j} \otimes x_i  \otimes   \overline{x}_i\Big)\otimes y_j \Big\|
 =  \kappa_{SA}(\mathcal{F}_2) \Big\| \sum_{i\in I}\Big(\sum_{j\in J}  b_{i,j}  \otimes y_j \Big)\otimes x_i  \otimes   \overline{x}_i\Big\| 
\\
  \le& \kappa_{SA}(\mathcal{F}_2) \kappa_{SA}(\mathcal{F}_1) \Big\| \sum_{i\in I}\Big(\sum_{j\in J}  b_{i,j}  \otimes y_j \Big)\otimes x_i \Big\|   = \kappa_{SA}(\mathcal{F}_1) \kappa_{SA}(\mathcal{F}_2) \Big\| \sum_{i\in I, j\in J }  b_{i,j}  \otimes x_i \otimes y_j  \Big\|. 
\end{align*}
The converse inequality is proved similarly. 
\end{proof}

\subsection{SAP--an restricted version}\label{sec-boolen-graph}
An operator $a\in B(\ell^2)$ is called Boolean if its standard matrix representation $[a(i,j)]_{i,j\in \N}$  satisfies $a(i,j)\in \{0,1\}$ for all $i,j$.  The support of a Boolean operator is defined by 
\[
\supp(a) = \{(i, j)\in \N^2: a(i,j)\ne 0\}.
\] 

Note  that  any non-zero element in  a SAP system has unit norm.
Proposition \ref{prop-positive} suggests that  any family of unit norm Boolean operators has a  restricted-version SAP if we require  certain positivity condition on the  coefficients. This result is probably known, but we have not found it in the literature.

\begin{proposition}\label{prop-positive}
Fix any family $\mathcal{F}$ of unit-norm Boolean operators in $B(\ell^2)$. Then for any integer $m\ge 2$ (the norms below are taken in the corresponding $C^*$-algebras), the equality
\[
\big\| \sum_{x\in \mathcal{F}} c_x \otimes   x^{\otimes m} \big\| =  \big\| \sum_{x\in \mathcal{F}} c_x \otimes   x  \big\| \quad \text{with $x^{\otimes m} = \underbrace{x \otimes \cdots \otimes x}_{\text{$m$ times}}$}
\]
holds for any finitely supported family $(c_x)_{x\in \mathcal{F}}$ of  $N\times N$ matrices  satisfying any one of the following conditions:
\begin{itemize}
\item[(c1)]  all matrices $c_x$ have non-negative coefficients; 
\item[(c2)]  $N= k^2$ and all matrices $c_x$ have the form  $c_x = b_x \otimes \overline{b}_x$ with $b_x$ an $k\times k$ matrix; 
\item[(c3)]  $N= k^2$ and all matrices $c_x$ have the form  $c_x = b_x \otimes b_x^*$ with $b_x$ an $k\times k$ matrix.
\end{itemize}
\end{proposition}

 The following lemma is elementary. 

\begin{lemma}\label{lem-norm-one}
A Boolean matrix $a=[a(i,j)]_{i,j\in \N}$ has  operator norm $\|a\|=1$ if and only if there exists a bijection $\sigma: I_\sigma \rightarrow J_\sigma$ between two subsets $I_\sigma, J_\sigma \subset \N$ such that
\[
a(i,j)=  \mathds{1}(i\in I_\sigma) \mathds{1}(j= \sigma(i)). 
\]
\end{lemma}
\begin{proof}
It suffices to notice that a non-zero Boolean matrix has operator norm $1$ if and only if  it has at most one non-zero entry in every row  and in every column. 
\end{proof}

The unit norm Boolean  matrix $a$ in Lemma~\ref{lem-norm-one} will be denoted by $a_\sigma$ and  it is given by 
\begin{align}\label{M-sigma-sum}
 a_\sigma  = \sum_{i\in I_\sigma} e_{i, \sigma(i)} = \sum_{j\in J_\sigma} e_{\sigma^{-1}(j), j}, 
\end{align}
where $e_{i,j}$ denotes the Boolean matrix whose all entries are zero except the $(i,j)$-entry equals $1$.
Let $S_{bool}(\ell^2)$ denote the set of all unit norm Boolean operators in $B(\ell^2)$. The following lemma is simple and we omit its proof (the unit norm assumption is important). 
\begin{lemma}\label{lem-closed}
If $a_\sigma, a_\tau\in S_{bool}(\ell^2)$, then  $ a_\sigma^* \in S_{bool}(\ell^2)$ and $a_\sigma a_\tau \in S_{bool}(\ell^2)\cup \{0\}$. 
\end{lemma}

\begin{proof}[Proof of Proposition \ref{prop-positive}]
By a standard truncation-approximation argument, we may assume that $\mathcal{F}$ is a family of unit norm Boolean matrices of a fixed finite size $d\times d$.  Then, by Lemma~\ref{lem-norm-one}, each Boolean matrix in $\mathcal{F}$ corresponds to a bijection $\sigma: I_\sigma \rightarrow J_\sigma$ with $I_\sigma, J_\sigma$ two subsets of $\{1, 2, \cdots, d\}$. Hence we can  write $\mathcal{F} = \{a_\sigma\}_\sigma$.  Denote  
\[
T_m: = \big\| \sum_{x\in \mathcal{F}} c_x \otimes   x^{\otimes m} \big\| =  \big\| \sum_{\sigma} c_\sigma \otimes  a_\sigma ^{\otimes m} \big\| . 
\]
Our goal is to prove the equality 
\[
T_m= T_1.
\] 
By the classical formula for spectral radius, we have  
\begin{align*}
T_m &  = \lim_{n\to\infty}\Big\{ \Tr   \Big[\Big( \sum_{\sigma} c_\sigma \otimes    a_\sigma ^{\otimes m} \Big)^*  \Big( \sum_{\sigma} c_\sigma \otimes    a_\sigma ^{\otimes m} \Big)\Big]^n \Big\}^{1/2n}
\\
 & = \lim_{n\to\infty}\Big\{   \sum_{\sigma_1, \cdots, \sigma_n,\atop \tau_1, \cdots, \tau_n} \Tr( c_{\sigma_1}^* c_{\tau_1} \cdots c_{\sigma_n}^* c_{\tau_n}) \cdot [\Tr (a_{\sigma_1}^* a_{\tau_1} \cdots a_{\sigma_n}^* a_{\tau_n})]^{ m} \Big\}^{1/2n}.
\end{align*}
By Lemma \ref{lem-closed}, for any $n\ge 1$ and any  $\sigma_1, \cdots, \sigma_n, \tau_1, \cdots, \tau_n$,  either $a_{\sigma_1}^* a_{\tau_1} \cdots a_{\sigma_n}^* a_{\tau_n} = 0$ or $a_{\sigma_1}^* a_{\tau_1} \cdots a_{\sigma_n}^* a_{\tau_n}  = a_\sigma$ for  a certain bijection $\sigma: I_\sigma \rightarrow J_\sigma$. Therefore, 
\[
\Tr(a_{\sigma_1}^* a_{\tau_1} \cdots a_{\sigma_n}^* a_{\tau_n})\in \{0, 1, \cdots, d\}. 
\] 
Under any one of the conditions (c1), (c2) and (c3), we have 
\[\Tr( c_{\sigma_1}^* c_{\tau_1} \cdots c_{\sigma_n}^* c_{\tau_n})\ge 0.
\] Consequently, using $\lim_{n\to\infty} d^{(m-1)/2n} = 1$, we obtain 
\begin{align*}
T_m &  \le  \limsup_{n\to\infty}     \Big\{    d^{m-1} \sum_{\sigma_1, \cdots, \sigma_n,\atop \tau_1, \cdots, \tau_n} \Tr( c_{\sigma_1}^* c_{\tau_1} \cdots c_{\sigma_n}^* c_{\tau_n}) \cdot \Tr (a_{\sigma_1}^* a_{\tau_1} \cdots a_{\sigma_n}^* a_{\tau_n}) \Big\}^{1/2n}
\\
&= \limsup_{n\to\infty}     \Big\{   \sum_{\sigma_1, \cdots, \sigma_n,\atop \tau_1, \cdots, \tau_n} \Tr( c_{\sigma_1}^* c_{\tau_1} \cdots c_{\sigma_n}^* c_{\tau_n}) \cdot \Tr (a_{\sigma_1}^* a_{\tau_1} \cdots a_{\sigma_n}^* a_{\tau_n}) \Big\}^{1/2n} = T_1.
\end{align*}
Conversely,  consider the diagonal subspace $
D_m: = \spann\{e_i^{\otimes m}: 1\le i \le d\} \subset   (\C^d)^{\otimes m}
$
and the corresponding orthogonal projection $P_m:(\C^d)^{\otimes m} \rightarrow D_m $. Let $U_m: \C^d \rightarrow D_m$ be the natural unitary operator sending each  $e_i$ to $e_i^{\otimes m}$.  Take any unit norm Boolean matrix $a_\sigma\in \mathcal{F}$ associated to a bijection $\sigma: I_\sigma \rightarrow J_\sigma$.  Using \eqref{M-sigma-sum}, we obtain 
\[
a_\sigma^{\otimes m} (e_i^{\otimes m}) =   (a_\sigma e_i )^{\otimes m } = (\mathds{1}(i \in J_\sigma)  e_{\sigma^{-1}(i)})^{\otimes m } = \mathds{1}(i \in J_\sigma) \cdot  e_{\sigma^{-1}(i)}^{\otimes m }, \quad i = 1, \cdots, d. 
\]
Hence
$
U_m^{-1}P_m (a_\sigma^{\otimes m }) P_m   U_m  =  a_\sigma$ and 
\[
\sum_{\sigma} c_\sigma \otimes  a_\sigma = \sum_{\sigma} c_\sigma \otimes  U_m^{-1}P_m (a_\sigma^{\otimes m }) P_m   U_m =  (Id \otimes U_m^{-1}P_m)  \Big(\sum_{\sigma} c_\sigma \otimes   a_\sigma^{\otimes m } \Big) ( Id\otimes P_m   U_m). 
\]
It follows immediately  that $T_m \ge T_1$. 
\end{proof}

\subsection{Examples and counter-examples}

\begin{example}\label{ex-disjoint}
Let $e_{i,j} \in B(\ell^2)$ denote the operator corresponding to the elementary  matrix whose all entries are zero except the $(i,j)$-entry equals $1$.  It is a standard fact that  
$
\{e_{i,j}: (i,j)\in \N^2\} \subset B(\ell^2)
$
 has SAP.
\end{example}

\begin{example}[{\it SAP is not preserved by direct sum}]\label{ex-overlap}
Given any two operators $x,y$, define  
$
x \oplus y=\Big[\begin{matrix}x&0\\ 0& y\end{matrix} \Big].
$  Consider  two families with SAP:  $\{x_1 = 1, x_2 =1, x_3 =0\} \subset \C$ and $\{y_1 = 1, y_2 = 0, y_3 = 1\}\subset \C$. Note that the  family 
\begin{align}\label{2-matrix}
\Big\{ z_1= x_1 \oplus y_1  = \Big[\begin{matrix} 1&0\\ 0&1 \end{matrix}\Big],  \, z_2 = x_2 \oplus y_2 =\Big[ \begin{matrix} 1&0\\ 0& 0\end{matrix} \Big], \, z_3 = x_3\oplus y_3 = \Big[\begin{matrix} 0&0\\ 0&1 \end{matrix}\Big]
\Big\}
\end{align}
does not have SAP, since $
0 = \| z_1 - z_2 - z_3\|  \ne \|z_1\otimes z_1 - z_2 \otimes z_2 - z_3 \otimes z_3\| = 1. $
\end{example}

\begin{example}[{\it SAP is not preserved by  spatial compression}]\label{ex-compression}
Consider the family 
\begin{align}\label{4-matrix}
\{X_1 = \diag (1,1,1,1), X_2 = \diag (1, 0, 1, 0), X_3 = \diag (0, 1, 1, 0)\}. 
\end{align}
It has SAP since for any $b_1, b_2, b_3$ in a $C^*$-algebra $B$, 
\[
\Big\| \sum_{i=1}^3 b_i \otimes X_i\Big\| =\Big \| \sum_{i=1}^3 b_i \otimes X_i \otimes X_i\Big\| = \max_{\varepsilon, \varepsilon'\in \{0, 1\}}\| b_1 + \varepsilon b_2  +\varepsilon' b_3\|. 
\]
Note that the non-SAP family \eqref{2-matrix} is a spatial compression of the  SAP family \eqref{4-matrix}. 
\end{example}

\begin{example}[{\it Regular representation system has SAP}]\label{ex-regular}
Let $G$ be a discrete group  and $\lambda_G: G\rightarrow B(\ell^2(G))$ be its left regular representation. We call $\{\lambda_G(g)\}_{g\in G}$ the regular representation system of $G$.   The standard Fell's absorption principle (see \cite[Prop. 8.1]{Pisier-Operator-space-book}) implies that $\{\lambda_G(g)\}_{ g\in G}$ has SAP.  Recall the definition \eqref{def-gamma-M} for the Hankel operators. One can check that $
\Gamma_g^G = \lambda_G(g) U_{flip}$ with $U_{flip}$ the unitary operator in $B(\ell^2(G))$ sending each $\delta_s$ to $\delta_{s^{-1}}$ for all $s\in G$. It follows that $\{\Gamma^G_g\}_{g \in G}$ has SAP.  
\end{example}

\begin{example}[{\it A non-SAP Hankel system inducing by a multicolored checkerboard}]\label{ex-3times3}
 The  multicolored checkerboard  in Figure~\ref{fig-3table} has the property that every color appears at most once in each column and each row  and it  corresponds to a coordinatewise injective (see Definition~\ref{def-SD} below) map 
\[
\Phi: \{1,2,3\}\times \{1,2,3\}\rightarrow  \{\text{red, orange, blue, purple, grey}\}.
\]  
\begin{figure}[h]
\begin{center}
\includegraphics[width=0.1\linewidth]{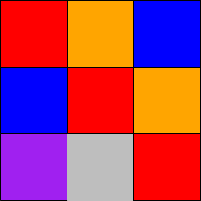}
\caption{A multicolored checkerboard.}\label{fig-3table}
\end{center}
\end{figure}

\noindent This colored checkerboard induces a  Hankel system in the $C^*$-algebra $M_3(\C)$ of  
 all $3\times 3$ complex matrices: 
\[
\{\Gamma_c^\Phi| c\in \{\text{red, orange, blue, purple, grey}\}\}\subset M_3(\C).
\]
For instance, $\Gamma_{\text{orange}}^\Phi  =  e_{1,2} + e_{2,3}$, $\Gamma_{\text{blue}}^\Phi = e_{1,3}+e_{2,1}$, etc.  Note that the coordinatewise injectivity of $\Phi$ implies that all $\Gamma_c^\Phi$ are of unit norm.   The self-tensorized system 
\[
\{\Gamma^\Phi_{c} \otimes \Gamma_c^\Phi\}_{c}\subset M_3(\C)\otimes M_3(\C) \simeq M_9(\C)
\] is induced by  the multicolored checkerboard in Figure~\ref{fig-9table}. 
\begin{figure}[h]
\begin{center}
\includegraphics[width=0.6\linewidth]{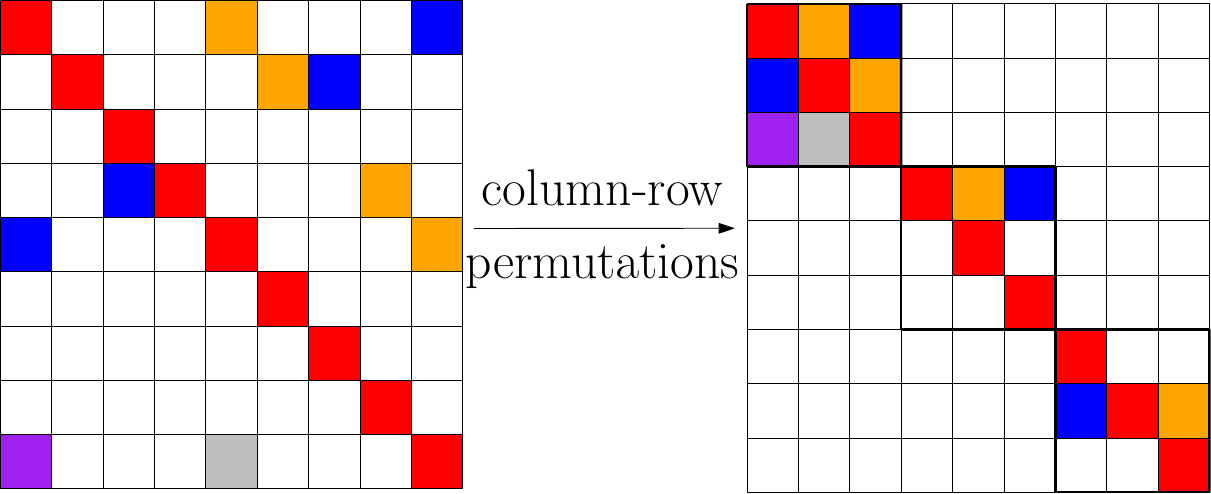}
\caption{The tensorized multicolored checkerboard.}\label{fig-9table}
\end{center}
\end{figure}
 To verify that $\{\Gamma_c^\Phi\}_c$ is non-SAP,  we consider the following substitution
\[
(\text{red, orange, blue, purple, grey}) \to (4,2,-1, 0, 0).
\] More precisely, set
 \[
A:=4\Gamma^\Phi_{\text{red}}+2\Gamma^\Phi_{\text{orange}}-\Gamma^\Phi_{\text{blue}}=\left[
\begin{array}{ccc}
 4 & 2 & -1 \\
 -1 & 4 & 2 \\
 0 & 0 & 4 \\
\end{array}
\right].
\]
Then the tensorized matrix is given by
\[
B:=4\Gamma^\Phi_{\text{red}}\otimes \Gamma^\Phi_{\text{red}}+2\Gamma^\Phi_{\text{orange}}\otimes\Gamma^\Phi_{\text{orange}}-\Gamma^\Phi_{\text{blue}}\otimes \Gamma^\Phi_{\text{blue}}\xrightarrow[\textit{permutations}]{\textit{column-row}}\left[
\begin{array}{ccc}
 A & 0 & 0 \\
 0 & A_1 & 0  \\
 0 & 0 & A_2 \\
\end{array}
\right],
\]
where 
\[
A_1=\left[
\begin{array}{ccc}
 4 & 2 & -1 \\
 0 & 4 & 0 \\
 0 & 0 & 4 \\
\end{array}
\right],\quad A_2=\left[
\begin{array}{ccc}
 4 & 0 & 0 \\
 -1 & 4 & 2 \\
 0 & 0 & 4 \\
\end{array}
\right].
\]
It can be checked that  $\|A\|=3\sqrt{3}$ and $\|A_1\|=\|A_2\|=\sqrt{\frac{1}{2} \left(\sqrt{345}+37\right)}>3\sqrt{3}$. Hence $\|B\|=\max\{\|A\|,\|A_1\|,\|A_2\|\}>\|A\|$. This implies that $\{\Gamma_c^\Phi\}_c$  does not have SAP.
\end{example}

\section{Hankel systems on monoids}\label{sec-real-hankel-monoid}
 All monoids below will be assumed to be cancellative.   For the reader's convenience, we recall the following definitions and give a new definition of {\it hereditary SAP monoids}: 
\begin{itemize}
\item Lunar monoids (Definition~\ref{def-AAM}):  a monoid $\MM$ is called a lunar monoid if it satisfies the lunar condition :
\[
\big(ax  = by, cx=dy, az=bw  \,\, \text{in $\MM$}\big)\Longrightarrow \big(cz=dw \,\, \text{in $\MM$}\big). 
\]
\item Hankel system on a monoid $\MM$: by the  Hankel system on  $\MM$, we mean the family  $\{\Gamma_t^\MM\}_{t\in \MM} \subset B(\ell^2(\MM))$ consisting of Hankel operators $\Gamma_t^\MM$ defined by 
\[
\Gamma_t^\MM(s_1, s_2)= \mathds{1}(s_1s_2 = t),  \quad s_1, s_2 \in \MM. 
\]
\item SAP monoids (Definition~\ref{def-SAP-monoid}): a monoid  $\MM$ is called a SAP monoid if its associated Hankel system $\{\Gamma_t^\MM\}_{t\in \MM} \subset B(\ell^2(\MM))$ has SAP. 
\item  Hereditary SAP monoids: we say that a monoid $\MM$ is a hereditary SAP monoid if for any  subsets $S_1,S_2 \subset \MM$, the compressed system 
\[
\{P_{S_1} \Gamma_t^\MM P_{S_2}\}_{t\in \MM}
\]
has SAP, where $P_{S_1}, P_{S_2}$ are othogonal projections from $\ell^2(\MM)$ onto $\ell^2(S_1)$ and $\ell^2(S_2)$ respectively (see Definition~\ref{def-hsap} for the definition of hereditary SAP of a general system).   In particular, any hereditary SAP monoid is a SAP monoid. 
\end{itemize}

One of our main results--Theorem \ref{thm-lunar-monoid}--follows from the following 
\begin{theorem}\label{thm-lunar-sap}
Any lunar monoid is a hereditary SAP monoid. 
\end{theorem}

The proof of Theorem~\ref{thm-lunar-sap} is postponed to \S \ref{sec-pf-thm-lunar-sap}.

 In general, it  seems to be  quite  non-trivial to determine whether a given monoid is a  SAP monoid or not. However, we have the following stability result.

\begin{proposition}\label{lem-monoid}
Cartesian product preserves the  class of SAP monoids.
\end{proposition}

\begin{proof}
Observe that 
$
\Gamma_{(t_1, t_2)}^{\MM_1 \times \MM_2} = \Gamma_{t_1}^{\MM_1} \otimes \Gamma_{t_2}^{\MM_2}.
$
The result follows from Proposition~\ref{prop-tensor}. 
\end{proof}

\begin{remark}
In the sequel to this paper, we shall prove the following more involved fact: 
\begin{center} 
{\it 
free product preserves the class of SAP monoids. }
\end{center}
\end{remark}

In Problem~\ref{prob-sap-lunar}, we asked whether the class of SAP monoids coincides with that of lunar ones. A more tractable open problem is 

\begin{problem}\label{prob-hsap-lunar}
Does the class of hereditary SAP monoids coincides with that of lunar ones ?
\end{problem}

Recall that  the lunar condition of a monoid $\MM$ is a local condition: to check whether $\MM$ satisfies the lunar condition, it suffices to check the configurations of all $4\times 4$ blocks in the multiplication table of $\MM$.  Therefore, a positive answer to Problem~\ref{prob-hsap-lunar} would follow from a positive answer to Problem~\ref{prob-hsap} in \S \ref{sec-hank-monoid}.

\section{Hankel systems of lunar maps}\label{sec-hank-monoid}
As explained in the introduction of this paper, the definition of  Hankel systems determined by the multiplication tables of monoids can be extended to the ones determined by level sets of  abstract two-variable maps.   This section is devoted to such non-classical Hankel systems. 

By generalizing the definition of lunar monoids (Definition~\ref{def-AAM}), we introduce a notion of lunar maps (Definition~\ref{def-SD}). Then we prove that the  Hankel system associated to any lunar map has SAP (see  Theorem~\ref{thm-Phi}) and in fact has a stronger hereditary SAP (see Definition~\ref{def-hsap} and Corollary~\ref{cor-comp}).  

Given any lunar map, we are able to construct a coupled discrete foliation decomposition (see the equalities \eqref{AX-fol} below). And this coupled foliation structure allows us to obtain   simultaneous block-diagonalization of all the  tensorized operators of the associated Hankel system (see the discussion in the beginning of \S \ref{pf-main}).

\begin{definition}[Lunar maps]\label{def-SD}
Consider three sets $\mathcal{A}, \mathcal{X}, \mathcal{L}$. 
A two-variable map $\Phi: \mathcal{A}\times \mathcal{X}\rightarrow \mathcal{L}$ is called a {\it lunar map} if 
\begin{itemize}
\item[(i)]$\Phi$ is {\it coordinatewise injective}, that is,  the map $\Phi(a,\cdot)$ is injective on $\mathcal{X}$ for each $a\in \mathcal{A}$ and so is the map $\Phi(\cdot,x)$ on $\mathcal{A}$ for   each $x\in \mathcal{X}$.
\item[(ii)]$\Phi$ satisfies the {\it generalized lunar condition}:  
 for any $a,b,c,d\in \mathcal{A}$ and $ x, y, z, w\in \mathcal{X}$, the following implication  holds: 
\begin{align}\label{def-g-lunar-c}
\Big(\Phi(a,x) = \Phi(b,y), \Phi(c,x)=\Phi(d,y), \Phi(a,z)=\Phi(b,w)\Big)\Longrightarrow \Big(\Phi(c,z)=\Phi(d,w)\Big). 
\end{align}
\end{itemize}
\end{definition}

Given a  coordinatewise  injective  map  $\Phi: \mathcal{A}\times \mathcal{X} \rightarrow \mathcal{L}$, by the  Hankel system associated to  the map $\Phi$, we mean the family  $\{\Gamma_\ell^\Phi\}_{\ell\in \mathcal{L}}\subset B(\ell^2(\mathcal{X}), \ell^2(\mathcal{A}))$  consisting of operators $\Gamma_\ell^\Phi$ : 
\begin{align}\label{def-G-Phi}
\Gamma_\ell^\Phi(a, x)= \mathds{1}(\Phi(a, x)= \ell), \quad (a, x)\in \mathcal{A}\times \mathcal{X}. 
\end{align}
Note that   $\Phi$ is coordinatewise injective if and only if   the operators $\Gamma_\ell^\Phi \in  B(\ell^2(\mathcal{X}), \ell^2(\mathcal{A}))$ are  of  unit norm  $\| \Gamma_\ell^\Phi \|=1$ for all $\ell$ in the image of $\Phi$.

\begin{mainthm}\label{thm-Phi}
For any lunar map $\Phi: \mathcal{A}\times \mathcal{X}\rightarrow \mathcal{L}$, 
the Hankel system $\{\Gamma^\Phi_\ell \}_{\ell\in \mathcal{L}}$  has SAP. 
\end{mainthm}

One may wonder whether the lunar condition on a two-variable map $\Phi$ is also necessary for the Hankel system $\{\Gamma_\ell^\Phi\}_\ell$ to have SAP. However, even for maps on finite sets of very small size, we are not able to answer this question.  
\begin{problem}\label{prob-hsap}
Consider any  the map $\Phi: \{1,2,3,4\} \times \{1,2,3,4\}\rightarrow \mathcal{L}$. Assume that the Hankel system of associated to  $\Phi$ has SAP. Does it follow that $\Phi$ is a lunar map ? 
\end{problem}

\subsection{Hereditary SAP}
In general, the SAP is not preserved by compression (see Example~\ref{ex-compression}). Therefore, we introduce a stronger notion of SAP in the following 
\begin{definition}[Hereditary SAP]\label{def-hsap}
A system $\{x_i\}_{i\in I} \subset B(\ell^2)$ is said to have hereditary SAP if for any  subsets $S_1,S_2 \subset \N$, the compressed system 
\[
\{P_{S_1} x_i P_{S_2}\}_{i\in I}
\]
has SAP. Here $P_{S_1}, P_{S_2}$ are othogonal projections onto $\ell^2(S_1)$ and $\ell^2(S_2)$ respectively.  More generally, 
the hereditary SAP for systems in $B(\ell^2(\mathcal{X}), \ell^2(\mathcal{A}))$ is defined in a similar way. 
\end{definition}

\begin{corollary}\label{cor-comp}
For any lunar map $\Phi: \mathcal{A}\times \mathcal{X}\rightarrow \mathcal{L}$, 
the Hankel system $\{\Gamma^\Phi_\ell \}_{\ell\in \mathcal{L}}$  has hereditary SAP. 
\end{corollary}

As a consequence,  we prove that 
the regular representation system of any countable discrete group (see Example~\ref{ex-regular}) has hereditary SAP. We mention that, it   does not follow from the standard Fell's absorption principle.  In particular, it implies  \cite[Cor. 5.3]{Katsoulis:2023aa}.

\begin{corollary}\label{cor-gp-compression}
The regular representation system $\{ \lambda_G(g)\}_{g\in G} \subset B(\ell^2(G))$ of any countable discrete group $G$ has hereditary SAP. 
\end{corollary}

Recall that the big Hankel operators on the Hardy spaces over higher dimensional or infinite dimensional  torus $\T^d$, where $d$ is an integer $d\ge 2$ or $d = \infty$ and $\T^\infty$ means $\T^\N$ (all the notation is classical and will be recalled in \S \ref{sec-app} below): 
\[
\Gamma_\varphi^{\mathrm{big}}: H^2(\T^d) \xrightarrow{\,M_\varphi\,} L^2(\T^d) \xrightarrow[\textit{projection}]{\textit{orthogonal}}   L^2(\T^d) \ominus H^2(\T^d), 
\]
where $M_\varphi$ is initially densely defined by $M_\varphi(f)= \varphi f$  with $\varphi \in L^2(\T^d)$.  By the big Hankel system, we mean 
\begin{align}\label{def-big-hank}
\{\Gamma_{e^{i n\cdot \theta}}^{\mathrm{big}}\}_{n\in \Z^d\setminus \N^d}  \subset B\big(H^2(\T^d), L^2(\T^d) \ominus H^2(\T^d)\big), 
\end{align}
when $d = \infty$, the index set $\Z^d\setminus\N^d$ should be replaced by $\Z^{(\N)}\setminus \N^{(\N)}$. 

\begin{corollary}\label{cor-big-hankel}
For any integer $d \ge 2$ or $d = \infty$, the big Hankel system  \eqref{def-big-hank} has SAP. 
\end{corollary}

\subsection{Constructions of lunar maps}\label{sec-lunar-map}  
Here we  give a convenient reformulation of the generalized lunar condition.  

For any pair $(a, b)\in \mathcal{A}^2$, define the solution set of the equation $\Phi(a,x)= \Phi(b,y)$ by 
\begin{align}\label{def-sol-phi}
\Sol_\Phi(a,b):= \{(x,y)\in \mathcal{X}^2| \Phi(a,x)= \Phi(b,y)\}.
\end{align}

\begin{lemma}\label{lem-lunar-eq}
A map $\Phi$ satisfies the generalized lunar condition \eqref{def-g-lunar-c} if and only if for any $a, b, c, d\in \mathcal{A}$, either $\Sol_\Phi(a,b)= \Sol_\Phi(c
,d)$ or $\Sol_\Phi(a
,b)\cap \Sol_\Phi(c
,d)  = \emptyset$. 
\end{lemma}

 The proof of Lemma~\ref{lem-lunar-eq} is direct and is omitted.  This reformulation provides an  algorithm  for checking the generalized lunar condition.

 We now turn to basic constructions of lunar maps.

\begin{example}[Rational maps]\label{ex-polynomial}
Given $a,b, m, n \in  \Z^*$, define
$
\Phi: \N^*\times \N^* \rightarrow \Q$ by 
 \[
\Phi(x,y)= a x^m + b y^n.
\] Then $\Phi$ is a lunar map. 
\end{example}

\begin{example}[Tensor products]
If two maps $\Phi_i: \mathcal{A}_i \times \mathcal{X}_i\rightarrow \mathcal{L}_i$ ($i=1,2$) are lunar, then so is  their tensor product $\Phi_1\otimes \Phi_2$ given by 
\[
(\Phi_1\otimes \Phi_2)((a_1, a_2), (x_1, x_2)) : = (\Phi_1(a_1, x_1), \Phi_2(a_2, x_2)). 
\]
\end{example}

\begin{example}[Refinement of level sets]
Given two maps defined on the same product space $\Phi_i: \mathcal{A} \times \mathcal{X} \rightarrow \mathcal{L}_i (i = 1, 2)$, define $\Phi: \mathcal{A}\times \mathcal{X} \rightarrow \mathcal{L}_1\times \mathcal{L}_2$ by $\Phi(a,x)= (\Phi_1(a,x), \Phi_2(a,x))$.   If $\Phi_1$ and $\Phi_2$ are lunar maps, then so is $\Phi$. 

Note that the level sets of $\Phi$ are refinement of level sets of both $\Phi_1$ and $\Phi_2$ and the Hankel system $\{\Gamma_{(\ell_1, \ell_2)}^\Phi\}_{(\ell_1, \ell_2)}$ is obtained as  Hadamard-Schur products of operators from the two systems $\{\Gamma_{\ell_1}^{\Phi_1}\}_{\ell_1}$ and $\{\Gamma_{\ell_2}^{\Phi_2}\}_{\ell_2}$:
\[
\Gamma_{(\ell_1, \ell_2)}^\Phi  (a,x)= \Gamma_{\ell_1}^{\Phi_1}(a,x) \Gamma_{\ell_2}^{\Phi_2}(a,x). 
\]

\end{example}

\begin{example}[Restrictions]\label{ob-restriction} 
Let  $\Phi: \mathcal{A}\times \mathcal{X}\rightarrow \mathcal{L}$ be a lunar map. For any subsets  $\mathcal{A}_1\subset \mathcal{A}$ and $\mathcal{X}_1\subset \mathcal{X}$, the restriction  $\Phi|_{\mathcal{A}_1\times \mathcal{X}_1}:  \mathcal{A}_1 \times \mathcal{X}_1 \rightarrow \mathcal{L}$ is again a lunar map. 
\end{example}

\begin{example}[Transposition]
Given a  map $\Phi: \mathcal{A}\times \mathcal{X}\rightarrow \mathcal{L}$, its transposition  $\Phi^t$ is defined by $\Phi^t(x,a)= \Phi(a,x)$. 
It can be checked that $\Phi$ is a lunar map if and only if so is $\Phi^t$. 
\end{example}

\begin{example}[Group operations]\label{ex-group}
For any group  $G$, let $\Phi_1, \Phi_2: G\times G \rightarrow G$ be defined by  
\begin{align}\label{def-Phi-12}
\Phi_1(g_1, g_2)= g_1 g_2 \an \Phi_2(g_1, g_2)= g_1g_2^{-1}, \quad \forall g_1, g_2 \in G. 
\end{align}
Then  both $\Phi_1$ and $\Phi_2$ are lunar maps.
\end{example}

\begin{example}[Restriction of group operations]
By the constructions in  Examples \ref{ob-restriction} and \ref{ex-group}, for any $S_1, S_2 \subset G$, the corresponding  restricted maps $\Phi_1|_{S_1\times S_2},\Phi_2|_{S_1\times S_2}: S_1 \times S_2 \rightarrow G$ of $\Phi_1, \Phi_2$ defined in \eqref{def-Phi-12}  are lunar maps.
\end{example}

\begin{example}[Lunar monoids]\label{ex-lunar-monoid}
Let $\MM$ be a monoid and define $\Phi_\MM: \MM\times \MM\rightarrow \MM$ by 
\begin{align}\label{def-phi-M}
\Phi_\MM(x,y) = xy.
\end{align}
 It can be  checked directly that $\Phi_\MM$ is a lunar map if and only if $\MM$ is a lunar monoid in the sense of Definition~\ref{def-AAM}. 
\end{example}

\begin{example}[Restriction of monoid operations]
Let $\MM$ be a lunar monoid and let $S_1, S_2\subset \MM$. Then, by the constructions in  Examples \ref{ob-restriction} and \ref{ex-lunar-monoid},  the restriction onto the subset $S_1\times S_2$ of the map $\Phi_\MM$ defined in \eqref{def-phi-M} is a lunar map.  In particular, the class of lunar monoids is closed under the operation of  taking the submonoids. 
\end{example}

\begin{example}[Group-embeddable monoids]
By the constructions in Examples \ref{ob-restriction}, \ref{ex-group} and \ref{ex-lunar-monoid}, any  group-embeddable monoid is a lunar monoid in the sense of  Definition~\ref{def-AAM}.  Therefore, any cancellative Abelian monoid is a  lunar monoid, since it can be embedded into its Grothendieck group (see, e.g., \cite[Section 2.A, page 52]{Bruns-Polytopes-K-theory}).
\end{example}

For non-lunar maps, we have 
\begin{example}[Non-lunar polynomials]
The polynomial  $P(x,y)= x^2 + y^2 +xy$ defines a non-lunar map on $\N^*\times \N^*$. 
\end{example}

Therefore, it is natural to consider 
\begin{problem}\label{prob-lunar-poly}
Determine all integer-coefficient polynomials $P\in \Z[x, y]$ such that their restrictions on  $\N\times \N$ (or $\Z\times \Z$) are lunar maps. 
\end{problem}

\subsection{Proof of Theorem \ref{thm-Phi}}\label{pf-main}

Fix a lunar map  $\Phi: \mathcal{A}\times \mathcal{X} \rightarrow \mathcal{L}$ in the sense of  Definition~\ref{def-SD}. 

The proof is  divided into the following ten steps.

{\flushleft \it Step 1.  An equivalence relation on $\mathcal{A}^2$.}  Recall the notation introduced in \eqref{def-sol-phi}. The map $\Phi$ induces  an equivalence relation on $\mathcal{A}^2$  given  by 
\begin{align}\label{def-rel-Phi}
(a, b) \sim_\Phi (c, d) \Longleftrightarrow \Sol_\Phi(a, b)    =  \Sol_\Phi(c, d). 
\end{align}
Let $[\mathcal{A}^2]  = \mathcal{A}^2/\!{\sim_\Phi}$ be the corresponding quotient space whose general elements  will be denoted as $[a, b]$. Note that all pairs $(a, b)$ with $\Sol_\Phi(a,b)= \emptyset$ are equivalent under the relation $\sim_\Phi$.  The corresponding equivalent class will be denoted by $
[\star]:=\big\{(a,b)\in \mathcal{A}^2| \Sol_\Phi(a,b)= \emptyset\big\}$. 
 Set \[
[\mathcal{A}^2]^\times = [\mathcal{A}^2] \setminus \{[\star]\}.
\] By Lemma \ref{lem-lunar-eq},   if $\Sol_\Phi(a,b) \ne \emptyset$, then $(a, b) \sim_\Phi (c, d)$ if and only if 
$
\Sol_\Phi(a,b)\cap \Sol_\Phi(c,d)\ne \emptyset$. 

{\flushleft\it Step 2. Coupled foliation decompositions.} Define 
\begin{align*}
(\mathcal{X}^2)_{\text{left}}:&  = \{(x,y)\in \mathcal{X}^2\big| \exists (a,b)\, s.t.\, \Phi(a,x)= \Phi(b,y)\}. 
\\
(\mathcal{A}^2)_{\text{right}}: & = \{(a,b) \in \mathcal{A}^2\big| \exists (x,y)\, s.t.\, \Phi(a,x)= \Phi(b,y)\}.
\end{align*}
Then we have the following coupled foliation  decompositions: 
\begin{align}\label{AX-fol}
\begin{split}
(\mathcal{X}^2)_{\text{left}} & = \bigsqcup_{[a,b]\in [\mathcal{A}^2]^\times} \Sol_\Phi(a,b), 
\\
 (\mathcal{A}^2)_{\text{right}} &= \bigsqcup_{[a,b] \in [\mathcal{A}^2]^\times} [a,b] = \bigsqcup_{[a,b]\in [\mathcal{A}^2]^\times} \big\{(c,d)\in \mathcal{A}^2\big| (c,d)\sim_\Phi (a,b)\big\}. 
\end{split}
\end{align}
 The two  decompositions \eqref{AX-fol} are coupled in the sense that their components are in natural  one-to-one correspondance. This correspondance will be useful later.

{\flushleft\it Step 3. Coupled orthogonal decompositions.} The coupled decompositions in \eqref{AX-fol} induce coupled  orthogonal decompositions of $\ell^2(\mathcal{X})\otimes_2\ell^2(\mathcal{X})$ and $\ell^2(\mathcal{A})\otimes_2\ell^2(\mathcal{A})$ as follows. 

For any equivalence class $[a, b]\in [\mathcal{A}^2]^\times$, define  a closed subspace of $\ell^2(\mathcal{X})\otimes_2 \ell^2(\mathcal{X})
$ by 
\begin{align*}
\mathcal{H}_\Phi([a,b]):=\overline{\spann}\{e_{x} \otimes e_{y}| (x, y) \in \Sol_\Phi(a,b)\} \subset \ell^2(\mathcal{X})\otimes_2 \ell^2(\mathcal{X}).
\end{align*}
Clearly,  if $[a, b] \ne [c,d]$ in $[\mathcal{A}^2]^\times$, then
$
\mathcal{H}_\Phi([a,b]) \perp \mathcal{H}_\Phi([c,d]).
$
Therefore, 
\begin{align*}
\mathcal{H}_\Phi: =  \ell^2((\mathcal{X}^2)_{\text{left}}) =  \bigoplus_{[a,b]\in [\mathcal{A}^2]^\times} \mathcal{H}_\Phi([a,b])   \subset \ell^2(\mathcal{X})\otimes_2 \ell^2(\mathcal{X}).
\end{align*}
Here $\bigoplus_{[a,b]\in [\mathcal{A}^2]^\times} \mathcal{H}_\Phi([a,b])$ denotes the closed subspace  of $\ell^2(\mathcal{X})\otimes_2 \ell^2(\mathcal{X})$ given by 
\[
\bigoplus_{[a,b]\in [\mathcal{A}^2]^\times} \mathcal{H}_\Phi([a,b])   = \overline{\spann} \Big\{ \bigcup_{[a,b]\in  [\mathcal{A}^2]^\times} \mathcal{H}_\Phi([a,b])\Big\}. 
\]
Denoting by $\mathcal{H}_\Phi^{\perp}$  the orthogonal complement of the subspace $\mathcal{H}_\Phi$ in $\ell^2(\mathcal{X})\otimes_2 \ell^2(\mathcal{X})$,  we obtain  the orthogonal decomposition:
\begin{align}\label{X-decomp}
\ell^2(\mathcal{X})\otimes_2 \ell^2(\mathcal{X}) = \mathcal{H}_\Phi^\perp \oplus \mathcal{H}_\Phi =\mathcal{H}_\Phi^\perp \oplus \bigoplus_{[a,b]\in [\mathcal{A}^2]^\times} \mathcal{H}_\Phi([a,b]). 
\end{align}
 
Similarly,  for any $[a, b]\in [\mathcal{A}^2]^\times$, define  a closed subspace  of $\ell^2(\mathcal{A})\otimes_2 \ell^2(\mathcal{A})$  by 
\begin{align}\label{def-K-Phi}
\mathcal{K}_\Phi([a,b]):  = \overline{\spann}\{e_c\otimes e_d| (c,d) \sim_\Phi (a,b) \} \subset \ell^2(\mathcal{A})\otimes_2\ell^2(\mathcal{A}).
\end{align}
Then 
\begin{align*}
\mathcal{K}_\Phi: = \ell^2((\mathcal{A}^2)_{\text{right}}) =   \bigoplus_{[a,b]\in [\mathcal{A}^2]^\times} \mathcal{K}_\Phi([a,b]) = \overline{\spann} \Big\{ \bigcup_{[a,b]\in  [\mathcal{A}^2]^\times} \mathcal{K}_\Phi([a,b])\Big\} \subset \ell^2(\mathcal{A})\otimes_2 \ell^2(\mathcal{A})
\end{align*}
and  we  obtain  the orthogonal decomposition:
\begin{align}\label{A-decomp}
\ell^2(\mathcal{A})\otimes_2 \ell^2(\mathcal{A})  = \mathcal{K}_\Phi^\perp \oplus \mathcal{K}_\Phi =\mathcal{K}_\Phi^\perp \oplus \bigoplus_{[a,b]\in [\mathcal{A}^2]^\times} \mathcal{K}_\Phi([a,b]) . 
\end{align}

{\flushleft \it Step 4. The space $\mathcal{H}_\Phi^\perp$ is in  the simultaneous kernel of all $\Gamma_\ell^\Phi\otimes \Gamma_\ell^\Phi$.}
 We claim that 
\begin{align}\label{simul-kernel}
\mathcal{H}_\Phi^\perp   \subset  \bigcap_{\ell\in \mathcal{L}}\ker(\Gamma_\ell^\Phi\otimes \Gamma_\ell^\Phi). 
\end{align}

Indeed, note that 
\[
\mathcal{H}_\Phi^\perp = \overline{\spann}\Big\{e_x\otimes e_y\Big| (x,y) \in   \mathcal{X}^2\setminus \bigcup_{[a,b]\in [\mathcal{A}^2]} \Sol_\Phi(a,b)  \Big\}.
\]
Fix any pair $(x,y) \in   \mathcal{X}^2\setminus \bigcup_{[a,b]\in [\mathcal{A}^2]} \Sol_\Phi(a,b)$.  Assume by contradiction that there exists $\ell\in \mathcal{L}$ such that $(\Gamma_{\ell}^\Phi \otimes \Gamma_{\ell}^\Phi)(e_x \otimes e_y) \ne 0$. Then by the definition \eqref{def-G-Phi} of $\Gamma_\ell^\Phi$, there exists $(a, b)\in \mathcal{A}^2$, such that $\Phi(a,x)= \Phi(b,y)$ and $(\Gamma_\ell^\Phi \otimes \Gamma_\ell^\Phi)(e_x \otimes e_y) = e_a \otimes e_b$. But this implies that $(x,y)\in \Sol_\Phi(a,b)$ and we obtain a contradiction. 

{\flushleft \it Step 5. Simultaneous block-diagonalization  of all $\Gamma_\ell^\Phi \otimes \Gamma_\ell^\Phi$.} We claim that, with respect to the orthogonal decompositions \eqref{X-decomp} and \eqref{A-decomp}, all operators $\Gamma_\ell^\Phi\otimes \Gamma_\ell^\Phi$ for $\ell \in \mathcal{L}$ are simultaneously block-diagonalized. In notation,  we write this simultaneous block-diagonalization as 
\begin{align}\label{simul-diag}
\Gamma_\ell^\Phi\otimes \Gamma_\ell^\Phi = \Big(\mathcal{H}_\Phi^\perp \xrightarrow{0} \mathcal{K}_\Phi^\perp \Big)\oplus \bigoplus_{[a,b]\in [\mathcal{A}^2]^\times} \Big( \mathcal{H}_\Phi([a,b])\xrightarrow{\Gamma_\ell^\Phi\otimes \Gamma_\ell^\Phi} \mathcal{K}_\Phi([a,b])\Big).
\end{align}
Here by simultaneous block-diagonalization, we mean that the orthogonal decompositions \eqref{X-decomp} and \eqref{A-decomp} are indepedent of $\ell \in \mathcal{L}$.

 Indeed,  by \eqref{simul-kernel},  we have 
\[
(\Gamma_\ell^\Phi\otimes \Gamma_\ell^\Phi)|_{\mathcal{H}_\Phi^\perp} \equiv 0 \quad \text{for all $\ell \in \mathcal{L}$}. 
\]
It remains to show that for any $[a,b]\in [\mathcal{A}^2]^\times$,
\begin{align}\label{H-Phi-K}
(\Gamma_\ell^\Phi\otimes \Gamma_\ell^\Phi) (\mathcal{H}_\Phi([a,b])) \subset \mathcal{K}_\Phi([a,b]) \quad \text{ for all $\ell \in \mathcal{L}$.}
\end{align}
Equivalently, we need to show that for any $(x,y)\in \Sol_\Phi(a,b)$, 
\[
(\Gamma_\ell^\Phi\otimes \Gamma_\ell^\Phi) (e_x \otimes e_y) \in  \mathcal{K}_\Phi([a,b]) \quad \text{ for all $\ell \in \mathcal{L}$.}
\]
But by the definition \eqref{def-G-Phi} of $\Gamma_\ell^\Phi$, either $(\Gamma_\ell^\Phi \otimes \Gamma_\ell^\Phi)(e_x\otimes e_y) = 0$ (the zero-vector $0$ clearly belongs to $\mathcal{K}_\Phi([a,b])$) or there exists $(c,d)\in \mathcal{A}^2$ such that $\Phi(c,x)= \Phi(d,y)=\ell$ and 
\[
(\Gamma_\ell^\Phi \otimes \Gamma_\ell^\Phi)(e_x \otimes e_y) = e_c\otimes e_d. 
\]
Now the condition $\Phi(c,x)= \Phi(d,y) = \ell$ implies that $(x,y) \in \Sol_\Phi(c,d)$ and hence $(x,y)\in \Sol_\Phi(a,b) \cap \Sol_\Phi(c,d)$. Therefore, $\Sol_\Phi(a,b) \cap \Sol_\Phi(c,d) \neq \emptyset$.   By assumption,  $\Phi$ satisfies the generalized lunar condition, hence by  Lemma \ref{lem-lunar-eq},  $\Sol_\Phi(a,b) = \Sol_\Phi(c,d)$ and then, by the definition  \eqref{def-rel-Phi} of the equivalence relation $\sim_\Phi$, we  obtain  $(a,b) \sim_\Phi (c,d)$. Consequently, by the definition \eqref{def-K-Phi}, we have  $e_c\otimes e_d \in \mathcal{K}_\Phi([a,b])$. That is, 
\[
(\Gamma_\ell^\Phi \otimes \Gamma_\ell^\Phi)(e_x \otimes e_y) =e_c\otimes e_d \in \mathcal{K}_\Phi([a,b]).
\] 
This completes the proof of the desired relation \eqref{H-Phi-K}. 

{\flushleft\it Step 6. Restriction of $\Gamma_\ell^\Phi\otimes \Gamma_\ell^\Phi$ onto the diagonal subspace.} Define the diagonal subspace of $\ell^2(\mathcal{X}) \otimes_2 \ell^2(\mathcal{X})$ and that of $\ell^2(\mathcal{A})\otimes_2 \ell^2(\mathcal{A})$ as 
\begin{align*}
\mathcal{H}^{\diag}:& = \overline{\spann}\{e_x\otimes e_x| x\in \mathcal{X}\} \subset \ell^2(\mathcal{X})\otimes_2 \ell^2(\mathcal{X});
\\
\mathcal{K}^{\diag}:& = \overline{\spann}\{e_a\otimes e_a| a\in \mathcal{A}\} \subset \ell^2(\mathcal{A})\otimes_2 \ell^2(\mathcal{A}).
\end{align*}
Let $U: \mathcal{H}^\diag \rightarrow \ell^2(\mathcal{X})$ be the unitary operator such that   $U (e_x\otimes e_x) = e_x$ for all $x\in \mathcal{X}$. Similarly, define a unitary operator $V: \mathcal{K}^\diag \rightarrow \ell^2(\mathcal{A})$ such that $V(e_a\otimes e_a) = e_a$ for all $a\in \mathcal{A}$.  Then we have the following simultaneous commutative diagram for all $\ell\in \mathcal{L}$ (here by simultaneous commutative diagram, we mean that the interwining operators $U, V$ are independent of $\ell$): 
\begin{equation}\label{CD-diag-subspace}
\begin{tikzcd}
\mathcal{H}^\diag \arrow{r}{ \Gamma_\ell^\Phi\otimes \Gamma_\ell^\Phi} \arrow[d, "\simeq", "U" swap]& \mathcal{K}^\diag \arrow[d, "V", "\simeq" swap]
\\
\ell^2(\mathcal{X}) \arrow{r}{\Gamma_\ell^\Phi} & \ell^2(\mathcal{A})
\end{tikzcd}.
\end{equation}

Indeed, if $\Gamma_\ell^\Phi (e_x) =0$, then 
$
V(\Gamma_\ell^\Phi\otimes \Gamma_\ell^\Phi) (e_x\otimes e_x) = 0 = (\Gamma_\ell^\Phi U )(e_x\otimes e_x). 
$
If $\Gamma_\ell^\Phi(e_x) \ne 0$, then there exists a unique $a\in \mathcal{A}$ such that $\Phi(a, x)= \ell, \Gamma_\ell^\Phi (e_x)= e_a$ and 
\[
[V (\Gamma_\ell^\Phi \otimes \Gamma_\ell^\Phi)](e_x\otimes e_x) = V(e_a \otimes e_a) = e_a =  \Gamma_\ell^\Phi (e_x)=  (\Gamma_\ell^\Phi U) (e_x \otimes e_x). 
\]
Thus the commutative diagram \eqref{CD-diag-subspace}  follows. 

\begin{remark}\label{rem-aa}
Note that the coordinatewise injectivity of $\Phi$ implies   $\Sol_\Phi(a,a) = \{(x,x)| x\in \mathcal{X}\}$ for any $a\in \mathcal{A}$ and $\Sol_\Phi(a, b) \cap \Sol_\Phi(a,a)  = \emptyset$ for any distinct $a, b\in \mathcal{A}$. It follows that 
\[
\text{$
\mathcal{H}^\diag= \mathcal{H}_\Phi([a,a])
$ and $\mathcal{K}^\diag = \mathcal{K}_\Phi([a,a])$. }
\]
\end{remark}

{\flushleft\it Step 7. Definitions of interwining operators for $\mathcal{H}_\Phi([a,b])$ and $\mathcal{K}_\Phi([a,b])$.} 
For any $[a,b]\in [\mathcal{A}^2]^\times$, define a bounded linear operator $P^{[a,b]}:  \mathcal{H}_\Phi([a,b])  \rightarrow \ell^2(\mathcal{X})$ such that 
\[
P^{[a,b]} (e_x\otimes e_y) =  e_x \quad \text{for all $(x,y)\in \Sol_\Phi(a,b)$}. 
\]
And  define  also a bounded linear operator $Q^{[a,b]}: \ell^2(\mathcal{A}) \rightarrow \mathcal{K}_\Phi([a,b])$ such that 
\begin{align}\label{def-Qab}
Q^{[a,b]} (e_c) = \left\{ 
\begin{array}{ll}
  e_c \otimes e_d, & \text{if there exists $d\in \mathcal{A}$ such that $(a,b) \sim_\Phi (c,d)$}
\\
0,& \text{otherwise}
\end{array}
\right..
\end{align}
We need to check that the above  $Q^{[a,b]}(e_c)$ is well-defined. Indeed, assume that $d_1, d_2 \in \mathcal{A}$ such that $(a,b)\sim_\Phi(c,d_1) \sim_\Phi(c,d_2)$. Since $[a, b]\in [\mathcal{A}^2]^\times$, we have $ \Sol_\Phi(c,d_1) = \Sol_\Phi(c,d_2) = \Sol_\Phi(a,b) \ne \emptyset$. Take any $(x_0, y_0) \in \Sol_\Phi(c,d_1) = \Sol_\Phi(c,d_2) = \Sol_\Phi(a,b)$, then $
\Phi(c, x_0)= \Phi(d_1, y_0)$ and $\Phi(c, x_0) = \Phi(d_2, y_0)$. Hence $\Phi(d_1, y_0) = \Phi(d_2, y_0)$. Consequently, by the coordinatewise injectivity assumption  on $\Phi$, we must have $d_1 = d_2$.  

Clearly, both operators $P^{[a,b]}$ and $Q^{[a,b]}$ are partial isometries (unitary operators are also considered as partial isometries). In particular, their  operator norms satisfy
\begin{align}\label{PQ-norm1}
\| P^{[a,b]}\| \le 1 \an \| Q^{[a,b]}\| \le 1. 
\end{align}

{\flushleft \it Step 8. Restriction of $\Gamma_\ell^\Phi \otimes \Gamma_\ell^\Phi$ onto $\mathcal{H}_\Phi([a,b])$.} We claim that, for any $[a,b]\in [\mathcal{A}^2]^\times$,  simultaneously for all $\ell\in \mathcal{L}$, the following diagram commutes (here again, we emphasize that, the interwining operators $P^{[a,b]}$ and $Q^{[a,b]}$ are both independent of $\ell\in \mathcal{L}$): 
\begin{equation}\label{CD-Hab}
\begin{tikzcd}
\mathcal{H}_\Phi([a,b]) \arrow{r}{  \Gamma_\ell^\Phi\otimes \Gamma_\ell^\Phi} \arrow[d, "P^{[a,b]}" swap]& \mathcal{K}_\Phi([a,b])
\\
\ell^2(\mathcal{X}) \arrow{r}{\Gamma_\ell^\Phi} & \ell^2(\mathcal{A})  \arrow[u, "Q^{[a,b]}" swap]
\end{tikzcd}.
\end{equation}

Indeed, take any $\ell\in \mathcal{L}$. Then for any $(x,y) \in \Sol_\Phi(a,b)$,  one  and  only one of the following  three cases happens:
\begin{itemize}
\item[(i)]  There does not exist $c\in \mathcal{A}$ such that $\Phi(c,x)= \ell$. In this case,  $\Gamma_\ell^\Phi(e_x)= 0$. Hence,  on the one hand,  
\[
(\Gamma_\ell^\Phi \otimes \Gamma_\ell^\Phi)(e_x\otimes e_y)  = (\Gamma_\ell^\Phi  e_x)\otimes (\Gamma_\ell^\Phi e_y) =  0
\] and on the other hand, 
\[
Q^{[a,b]} \Gamma_\ell^\Phi P^{[a,b]} (e_x\otimes e_y) =  Q^{[a,b]} \Gamma_\ell^\Phi (e_x) = 0.
\]
\item[(ii)] There exists $c \in \mathcal{A}$ with $\Phi(c,x) = \ell$ but there does not exist $d \in \mathcal{A}$ with $\Phi(d,y)= \ell$. In this case, $\Gamma_\ell^\Phi (e_x)= e_c$ and $\Gamma_\ell^\Phi(e_y) = 0$.  Moreover, there is no element $d\in \mathcal{A}$ such that $(a, b)\sim_\Phi (c,d)$. Hence by the definition \eqref{def-Qab} of $Q^{[a,b]}$, we must have $Q^{[a,b]}(e_c) = 0$. Therefore, on the one hand, 
\[
(\Gamma_\ell^\Phi \otimes \Gamma_\ell^\Phi)(e_x\otimes e_y)  = (\Gamma_\ell^\Phi  e_x)\otimes (\Gamma_\ell^\Phi e_y) =  0
\] and on the other hand, 
\[
Q^{[a,b]} \Gamma_\ell^\Phi P^{[a,b]} (e_x\otimes e_y) =  Q^{[a,b]} \Gamma_\ell^\Phi (e_x) = Q^{[a,b]}(e_c) =  0.
\]
\item[(iii)] There exist $c,d\in \mathcal{A}$ such that $\Phi(c,x) = \Phi(d,y) = \ell$ (in particular, it implies that $(a,b) \sim_\Phi (c,d)$). In this case, $\Gamma_\ell^\Phi(e_x) = e_c, \Gamma_\ell^\Phi(e_y) = e_d$ and $Q^{[a,b]}(e_c)= e_c \otimes e_d$. Consequently, on the one hand, 
\[
(\Gamma_\ell^\Phi \otimes \Gamma_\ell^\Phi)(e_x\otimes e_y)  = (\Gamma_\ell^\Phi  e_x)\otimes (\Gamma_\ell^\Phi e_y) =  e_c \otimes e_d
\]
and on the other hand, 
\[
Q^{[a,b]} \Gamma_\ell^\Phi P^{[a,b]} (e_x\otimes e_y) =  Q^{[a,b]} \Gamma_\ell^\Phi (e_x) = Q^{[a,b]}(e_c) =  e_c\otimes e_d.
\]
\end{itemize}
Clearly, in all these three cases, $(\Gamma_\ell^\Phi \otimes \Gamma_\ell^\Phi)(e_x\otimes e_y) = Q^{[a,b]} \Gamma_\ell^\Phi P^{[a,b]} (e_x\otimes e_y)$. Since $\ell \in \mathcal{L}$ and $(x,y)\in \Sol_\Phi(a,b)$ are chosen arbitrarily, we prove that the diagram \eqref{CD-Hab} commutes.  

{\flushleft \it Step 9. The proof of the inequality $\| \sum_{\ell} a_\ell \otimes \Gamma_\ell^\Phi\| \le \| \sum_{\ell} a_\ell \otimes \Gamma_\ell^\Phi \otimes \Gamma_\ell^\Phi\|$.} Take any finitely supported sequence $(a_\ell)_{\ell\in \mathcal{L}}$ in an arbitrary unital $C^*$-algebra $B$ with unit $1_B$.  By the simultaneous commutative diagram \eqref{CD-diag-subspace}, we have 
\begin{align*}
\sum_\ell a_\ell \otimes \Gamma_\ell^\Phi &= \sum_\ell a_\ell \otimes  V [(\Gamma_\ell^\Phi \otimes \Gamma_\ell^\Phi)|_{\mathcal{H}^\diag} ]U^{-1}  
\\
& = (1_B \otimes V) \Big( \sum_\ell a_\ell \otimes (\Gamma_\ell^\Phi \otimes \Gamma_\ell^\Phi)|_{\mathcal{H}^\diag}\Big) (1_B\otimes U^{-1}). 
\end{align*}
Consequently, 
\begin{align*}
\Big\|\sum_\ell a_\ell \otimes \Gamma_\ell^\Phi\Big\| = \Big\| \sum_\ell a_\ell \otimes (\Gamma_\ell^\Phi \otimes \Gamma_\ell^\Phi)|_{\mathcal{H}^\diag}\Big\| \le  \Big\| \sum_\ell a_\ell \otimes \Gamma_\ell^\Phi \otimes \Gamma_\ell^\Phi\Big\|.
\end{align*}

{\flushleft \it Step 10. The proof of the converse inequality $\| \sum_{\ell} a_\ell \otimes \Gamma_\ell^\Phi \otimes \Gamma_\ell^\Phi\| \le \| \sum_{\ell} a_\ell \otimes \Gamma_\ell^\Phi\| $.}
Take any finitely supported sequence $(a_\ell)_{\ell\in \mathcal{L}}$ in an arbitrary unital  $C^*$-algebra $B$ with unit $1_B$. The simultaneous block-diagonalization  \eqref{simul-diag}  together with the commutative diagram \eqref{CD-Hab} and the contractivity of all the interwining operators $P^{[a,b]}, Q^{[a,b]}$  in \eqref{PQ-norm1} imply that 
\begin{align*}
\Big\|\sum_{\ell} a_\ell \otimes \Gamma_\ell^\Phi \otimes \Gamma_\ell^\Phi \Big\| & = \sup_{[a,b]\in [\mathcal{A}^2]^\times}  \Big\|  \sum_\ell a_\ell \otimes   \Big( \mathcal{H}_\Phi([a,b])\xrightarrow{\Gamma_\ell^\Phi\otimes \Gamma_\ell^\Phi} \mathcal{K}_\Phi([a,b])\Big) \Big\|
\\
& = \sup_{[a,b]\in [\mathcal{A}^2]^\times}  \Big\|  \sum_\ell a_\ell \otimes   (Q^{[a,b]} \Gamma_\ell^\Phi P^{[a,b]}) \Big\|
\\
& = \sup_{[a,b]\in [\mathcal{A}^2]^\times}  \Big\|  (1_B \otimes Q^{[a,b]})\Big( \sum_\ell a_\ell \otimes  \Gamma_\ell^\Phi  \Big) (1_B\otimes P^{[a,b]})\Big\|
\\
&\le\Big\|\sum_\ell a_\ell \otimes \Gamma_\ell^\Phi\Big\| .
\end{align*}
This completes the proof of the converse inequality. 

\medskip

We thus complete the whole proof of Theorem \ref{thm-Phi}.

\subsection{Proof of Corollary \ref{cor-comp}}
Corollary \ref{cor-comp} follows immediately from Theorem~\ref{thm-Phi} by  using  Example~\ref{ob-restriction}: the lunar map is preserved by restriction.  

\subsection{Proof of Corollary \ref{cor-gp-compression}}
Corollary \ref{cor-gp-compression} follows from Corollary \ref{cor-comp}. Indeed,  for any group $G$, the map $\Phi: G\times G \rightarrow G$ defined by $\Phi(a,x) = ax^{-1}$  is a lunar map. Moreover, the corresponding operator $\Gamma_g^\Phi$ coincides with $\lambda_G(g)$ for any $g\in G$.

\subsection{Proof of Corollary \ref{cor-big-hankel}}
For any integer $d\ge 2$,  by using the standard orthonormal basis $\{e^{i n \cdot \theta}\}_{n\in \Z^d}$ of $L^2(\T^d)$, one can easily show that the big Hankel system \eqref{def-big-hank} is unitarily equivalent to the compression of the regular representation of the group $\Z^d$. When $d = \infty$, we only need to replace  $T^d$ by $\T^\N$ and replace $\Z^d$ by  the countable group $\Z^{(\N)}$.   Hence Corollary \ref{cor-big-hankel} follows from Corollary \ref{cor-gp-compression}.

\subsection{Proof of Theorem \ref{thm-lunar-sap}}\label{sec-pf-thm-lunar-sap}
By the construction of lunar maps via  lunar monoids (see Example~\ref{ex-lunar-monoid}), Theorem~\ref{thm-lunar-sap} follows immediately from Corollary~\ref{cor-comp}.

\section{Applications}\label{sec-app}

\subsection{Notation and preliminaries}\label{sec-notation}
\subsubsection{Notation}
  Let $\D = \{z\in \C: |z|<1\}$ be the unit disk endowed  with the normalized area measure $dA(z)= \frac{dxdy}{\pi}$ and  $\T=\partial \D$ be the unit circle endowed with the normalized Haar measure $dm$.  Recall that all the classical function spaces such as $L^p(\T, m),  H^p(\T),  \BMOA(\T)$   on $\T$ are simply denoted by: 
\[
 L^p  = L^p(\T, m),  \quad H^p = H^p(\T),  \quad \BMOA= \BMOA(\T). 
\]

Given any pair of Banach spaces $X, Y \subset H^1$ and a complex sequence $m = (m_n)_{n=0}^\infty$, the associated Fourier multiplier  $T_m$ (a priori densely defined) is the linear map:
\[
X \ni \sum_{n=0}^\infty \widehat{f}(n) e^{in\theta}  = f  \xrightarrow{\quad T_m\quad} T_m f =  f *\varphi_m =  \sum_{n  =0}^\infty m_n  \widehat{f}(n) e^{in\theta} \in Y,
\]
 here  the  associated formal Fourier series $\varphi_m = \sum_{n\ge 0} m_n e^{in\theta}$ is called the symbol of $T_m$ and will always be identified with $T_m$.   The space of all  bounded Fourier multipliers from $X$ to $Y$   is a Banach space with respect to the natural norm and will be denoted by 
\[(X, Y).\] 
If $X$ and $Y$ are both endowed with an \oss, then the space  of all  completely  bounded (abbrev. {\it cb}) Fourier multipliers from $X$ to $Y$, denoted by 
\[
(X,Y)_{cb},
\]
is equipped with a natural \oss  through isometries
\[
M_d((X, Y)_{cb}) = (X, M_d(Y))_{cb} \quad  \text{for all $d \ge 1$}.
\] Here $M_d(Z)$ denotes the space of $d\times d$ matrices with entries in $Z$.

\medskip
We shall also use the following notation:
\begin{itemize}
\item $M_m = M_m(\C)$ for the set of complex matrices of size $m\times m$; 
\item  $\ell^2_n = \ell^2(\{1, 2, \cdots, n\})$; 
\item   $H_0^p = \{f\in H^p| \widehat{f}(0)=0\}$; 
\item  $\overline{H^p}=\{\bar{f}| f\in H_p\}$ and similar notation like $\overline{H_0^\infty}$;
\item $S_p$  denotes the $p$-Schatten-von-Neumann  class of operators on $\ell^2$ for $1\le p<\infty$ and $S_\infty$ denotes the class of compact operators on $\ell^2$; in the case of $n\times n$ matrices,  the corresponding spaces will be denoted by $S_p^n$. In particular, $S_\infty^n = M_n$;  
\item $H^1(S_1)$ denotes the  space of $S_1$-vector-valued $H^1$-functions and similar notation like  $L^\infty(B(\ell^2))$,  $H^p(S_p)$ etc;
\item The discrete non-commutative  vector-valued $L_p$-spaces (introduced by Pisier in \cite[Chapter 1]{Pisier-Non-Commutative-vector}) are denoted by $S_p[E]$ (or $S_p^n[E]$), where $E$ is some operator space; recall Pisier's descriptions (\cite[Prop. 2.3]{Pisier-Non-Commutative-vector}) on the complete bounded norms: let $u: E\rightarrow F$ be a map between two operator spaces,  then, for any $1\le p\le\infty$, 
\begin{align}\label{cb-Sp}
\|u\|_{cb}= \| Id\otimes u: S_p[E] \rightarrow S_p[F]\| = \sup_{n\ge 1}\| Id\otimes u: S_p^n [E] \rightarrow S_p^n[F]\|.
\end{align}
\item For any two operator spaces $E, F$, we denote by $CB(E, F)$ the space of completely bounded operators from $E$ to $F$.
\item For any  $C^*$-algebra $A$, denoted by $\overline{A}$ its complex conjugate. Under the non-canonical identification $\ell^2 \simeq  \overline{\ell^2}$, we have $\overline{B(\ell^2)} \simeq B(\ell^2)$. In particular, if $a = [a_{ij}]_{i,j\in\N}\in B(\ell^2)$,  we define  $\overline{a}= [\bar{a}_{ij}]_{i,j\in\N} \in B(\ell^2)$. See  \cite[Section 2.3]{newbook-pisier}.
\end{itemize}

\subsubsection{Nehari-Sarason-Page Theorem}\label{sec-nsp} 
Let $\hank(\overline{H^2}, H^2) \subset B(\overline{H^2}, H^2)$ denote the set of all  bounded Hankel operators obtained by 
\[
\Gamma_\varphi: \overline{H^2} \xrightarrow{\,M_\varphi\,} L^2 \xrightarrow{\,\mathcal{R}_{+}\,} H^2, 
\]
where $M_\varphi$ is initially densely defined by $M_\varphi(f)= \varphi f$  with $\varphi \in L^2$ and $\mathcal{R}_{+}$ is the classical Riesz projection. 
With respect to the orthonormal basis $\{e^{-in\theta}\}_{n \ge 0}$ and $\{ e^{in\theta}\}_{n \ge 0}$  for $\overline{H^2}$ and $H^2$ respectively, any operator in $\hank(\overline{H^2}, H^2)$ has a matrix representation (here we slightly abuse the notation $\Gamma$): 
\[
\Gamma_a = [a_{i+j}]_{i,j\ge 0}, \quad \text{where $a= (a_i)_{i\ge 0} \in \C^\N$ is a complex sequence.}
\]  
 In particular,  we use identification: 
$
\Gamma_\varphi =  \Gamma_{(\widehat{\varphi}(n))_{n\ge 0}} = [\widehat{\varphi}(i+j)]_{i,j\ge 0}.$
Recall that we denote by  $\hank(\ell^2)$ the collection  of all bounded $\Gamma_a \in B(\ell^2)$.

The Nehari-Sarason-Page Theorem \cite{Nehari-Annofmath, Sarason-Nehari-vectorial, Page-Nehari-vectorial} asserts that, by equipping the space $L^\infty/\overline{H_0^\infty}$  with the quotient \oss,  the map 
$
L^\infty/\overline{H_0^\infty} \ni \varphi \mapsto \Gamma_\varphi\in \hank(\overline{H^2}, H^2)
$
is a complete isometric isomorphism. 
Throughout the whole paper, we will use the following equivalent norm on  $\BMOA$:
\begin{align}\label{bmoa-Linf}
\| \varphi\|_{\BMOA}:= \|\varphi\|_{L^\infty/\overline{H_0^\infty}}.
\end{align}
This choice of the equivalent  norm on $\BMOA$ will  allow us to compute the precise cb-norms (not just the equivalent ones) of multipliers in various situations. 

Thus,   by equipping   $\BMOA$ with the \oss induced by  $L^\infty/\overline{H_0^\infty}$, we have complete isometric identifications: 
\begin{align}\label{PSN-bis}
\begin{array}{ccccc}
\BMOA  \simeq L^\infty/\overline{H_0^\infty}&  \xrightarrow{\simeq}& \hank(\overline{H^2}, H^2) & \xrightarrow{\simeq}&  \hank(\ell^2)
\vspace{3mm}
\\
\varphi & \mapsto & \Gamma_\varphi & \mapsto & [\widehat{\varphi}(i+j)]_{i,j\ge 0}. 
\end{array}
\end{align}

\subsection{The multipliers in $(H^s, H^r)_{cb}$}
Recall the definition of $\BMOA^{(p)}$ ($1\le p \le  \infty$) introduced in \eqref{def-bmoap} and \eqref{def-bmoa-inf}.

\begin{theorem}\label{cor-interpolation}
For any $q\in [2, \infty]$, 
 \[
\BMOA^{(q')} \xhookrightarrow[\textit{inclusion}]{\textit{contractive}} (H^1, H^q)_{cb} = (H^{q'}, \BMOA)_{cb} \quad \text{with $\frac{1}{q}  + \frac{1}{q'}  = 1$}.
\]
More generally, for any $1\le r\le 2 \le s<\infty$, 
\begin{align}\label{gen-rs}
\BMOA^{(u)} \xhookrightarrow[\textit{inclusion}]{\textit{bounded}} (H^r, H^s)_{cb} \quad \text{with $\frac{1}{u} = \frac{1}{r} - \frac{1}{s}$}. 
\end{align}
\end{theorem}

\begin{remark}
The inclusion \eqref{gen-rs} can not be generalized to the case $s= \infty$ since
\[
\BMOA^{(2)} \not\subset H^2 = (H^2, H^\infty)_{cb} = (H^2, H^\infty). 
\]
\end{remark}

\subsubsection{The end point situations: $(H^1, \BMOA)_{cb}$ and  $(H^1, H^2)_{cb}$}
Using the self-absorption property introduced in Definition~\ref{def-self-tensor} and   Proposition~\ref{thm-doubling},  we  may fix an   \oss on  $\BMOA^{(2)}$ by requiring (for any integer $d \ge 1$)
\begin{align}\label{BMOA-os}
\Big\|\sum_{n=0}^\infty a_n \otimes e^{i n \theta}\Big\|_{M_d\otimes_{min} \BMOA^{(2)}}: =  \Big\| \sum_{n=0}^\infty a_n \otimes \overline{a}_n e^{i n \theta}\Big\|_{L^\infty(\T; \,M_d \otimes \overline{M}_d)/\overline{H_0^\infty}(\T;\, M_d\otimes \overline{M}_d)}^{1/2},
\end{align}
where $(a_n)_{n\ge 0}$ is any finitely supported sequence in $M_d$.

\begin{proposition}\label{thm-hardy}
We have  the following  complete isometric equalities: 
\[
(H^1, \BMOA)_{cb} = \BMOA \an (H^1, H^2)_{cb}  = (H^2, \BMOA)_{cb}= \BMOA^{(2)}. 
\]
\end{proposition}

\medskip

{\flushleft \bf (i) Characterization of $(H^1, \BMOA)_{cb}$.} We proceed to the proof of the complete isometric equality
$
(H^1, \BMOA)_{cb} = \BMOA
$ in Proposition~\ref{thm-hardy}. 
Namely, we shall prove 
\begin{align}\label{eq-HBMOA-cb}
(H^1, M_d(\BMOA))_{cb} \xrightarrow[\simeq]{\textit{isometric}} M_d(\BMOA) \quad \text{for all $d\ge 1$}.
\end{align}

Using  the complete isometric isomorphism 
$
(H^1)^* = \BMOA
$ and the definining property of \oss on operator space dual (see \cite[Section 2.3]{Pisier-Operator-space-book}),  we have
\begin{align}\label{tensor-embedding}
\BMOA\otimes_{min} M_d(\BMOA)  \xhookrightarrow[\text{isometric}]{\text{complete}} CB(H^1, M_d(\BMOA)).
\end{align}
We shall use the following  Lemma~\ref{lem-dilation} to reduce the computation of the cb-norm of any $T_m \in (H^1, M_d(\BMOA))_{cb} \subset CB(H^1, M_d(\BMOA))$ to the computation  of the cb-norms of associated Fourier multipliers lying in the image of the embedding \eqref{tensor-embedding}.

Given any sequence $m = (m_n)_{n=0}^\infty$ in $M_d$ and any $0< r<1$, let $m(r) = (m_n(r))_{n=0}^\infty$ denote the dilated sequence
\begin{align}\label{def-mr}
m_n(r): = r^n m_n, \quad n \in\N. 
\end{align}

\begin{lemma}\label{lem-dilation}
For any multiplier $T_m: H^1\rightarrow M_d(\BMOA)$, 
\[
\| T_m\|_{(H^1, M_d(\BMOA))_{cb}} = \sup_{0<r<1}   \| T_{m(r)}\|_{(H^1, M_d(\BMOA))_{cb}}.
\]
\end{lemma}

\begin{proof}
By \eqref{cb-Sp} and the non-commutative Fubini theorem (see \cite[Prop. 2.1]{Pisier-Non-Commutative-vector})
\begin{align}\label{nc-fubini}
S_p[H^p]   \xlongequal{\textit{isometric}}  H^p(S_p), \quad 1\le p <\infty, 
\end{align}
we have 
\begin{align}\label{Tm-cb-S1}
\begin{split}
\| T_m\|_{(H^1, M_d (\BMOA))_{cb}} &=   \| Id\otimes T_m: S_1[H^1] \rightarrow S_1[M_d(\BMOA)] \|
\\
& = \sup \Big\{ \Big\| \sum_{n=0}^\infty m_n \otimes \widehat{f}(n) e^{in\theta}\Big \|_{S_1[M_d(\BMOA)]} \Big| \| f\|_{H^1(S_1)} < 1 \Big\}.
\end{split}
\end{align}
Note that for any $g\in H^1(S_1)$ with $\| g\|_{H^1(S_1)} < 1$,  
\[
\lim_{r\to 1^{-}} \Big\| \sum_{n=0}^\infty r^n \widehat{g}(n) e^{i n \theta} - \sum_{n=0}^\infty \widehat{g}(n) e^{i n \theta}\Big\|_{H^1(S_1)} =0.
\]
 Hence the following subset is dense in the open unit  ball of $H^1(S_1)$: 
\[
\bigcup_{0<r <1}\Big\{ P_r *g \Big| \| g\|_{H^1(S_1)}<1 \Big\}  \stackrel{\text{dense}}{\subset} \Big\{ f \Big| \| f\|_{H^1(S_1)}<1 \Big\},
\]
where $P_r*g$ is the vector-valued version of  the standard  Poisson convolution $P_r: L^1\rightarrow L^1$  on $\T$ given by 
\begin{align}\label{def-Poisson}
(P_r * h) (e^{i\theta})  =  \sum_{n\in \Z} r^{|n|}  \widehat{h}(n) e^{in \theta}.
\end{align}
Therefore, the equality \eqref{Tm-cb-S1}  implies 
\begin{align*}
\| T_m\|_{(H^1, M_d(\BMOA))_{cb}} 
 =    \sup_{0<r<1 \an \| g \|_{H^1(S_1)} < 1}  \Big\| \sum_{n=0}^\infty m_n \otimes  r^n \widehat{g}(n) e^{in\theta}\Big \|_{S_1[M_d(\BMOA)]}.
\end{align*}
For any $0<r<1$, the equality \eqref{Tm-cb-S1} applied to $m(r)$ implies  that 
\[
  \sup_{\| g\|_{H^1(S_1)} < 1} \Big\| \sum_{n=0}^\infty m_n \otimes r^n \widehat{g}(n) e^{in\theta}\Big \|_{S_1[M_d(\BMOA)]}    =  \| T_{m(r)}\|_{(H^1, M_d(\BMOA))_{cb}}.
\]
We thus complete the whole proof.  
\end{proof}

\begin{proof}[Proof of the isometric equality \eqref{eq-HBMOA-cb}]
Clearly, if $T_m\in (H^1, M_d(\BMOA))_{cb}$, then the sequence $m = (m_n)_{n\in \N}$ is bounded and hence
\begin{align}\label{dial-sum}
\sum_{n= 0}^\infty \| m_n(r)\|<\infty \quad \text{for all $0<r<1$}.
\end{align}
Thus  $T_{m(r)}$ coincides with the image of following tensor under the embedding \eqref{tensor-embedding}: 
\[
\sum_{n=0}^\infty     e^{in \theta} \otimes  m_n(r) \otimes e^{in \theta} \in \BMOA\otimes_{min} M_d(\BMOA), 
\]
where $e^{in\theta}$ denotes the function $e^{i\theta}\mapsto e^{in\theta}$. It follows that
\begin{align*}
\| T_{m(r)}\|_{(H^1, \BMOA)_{cb}} &= \Big\| \sum_{n=0}^\infty   e^{in \theta} \otimes m_n(r)  \otimes e^{in\theta}\Big\|_{ \BMOA\otimes_{min}M_d(\BMOA)}  &  \text{(by \eqref{tensor-embedding})}
\\
&  =  \Big\| \sum_{n=0}^\infty  \Gamma_n \otimes  m_n(r) \otimes \Gamma_n \Big\|_{ \hank(\ell^2)\otimes_{min} M_d(\hank(\ell^2))}    &  \text{(by \eqref{PSN-bis})}
\\
& = \Big\| \sum_{n=0}^\infty  m_n(r)    \otimes  \Gamma_n \Big\|_{ M_d(\hank(\ell^2))} &  \text{(by Proposition \ref{thm-doubling})}
\\
&  = \Big\| \sum_{n=0}^\infty   m_n(r) \otimes e^{in\theta} \Big\|_{M_d(\BMOA)} &  \text{(again by \eqref{PSN-bis})}.
 \end{align*}
Now by Lemma \ref{lem-Pr} and Lemma \ref{lem-dilation}, 
\[
\| T_m\|_{(H^1, M_d(\BMOA))_{cb}} = \Big\| \sum_{n=0}^\infty  m_n\otimes  e^{in\theta} \Big\|_{M_d(\BMOA)}.
\]
Conversely, if $\sum_{n=0}^\infty  m_n\otimes  e^{in\theta}\in M_d(\BMOA)$, then  $(m_n)_{n=0}^\infty$ is also bounded. We may repeat the above argument and obtain the desired isometric equality \eqref{eq-HBMOA-cb}. 
\end{proof}

\medskip

{\flushleft \bf (ii) Characterization of $(H^1, H^2)_{cb}$.} 
We now turn to the proof of the  second  complete isometric equality of Proposition~\ref{thm-hardy}:
\[
(H^1, H^2)_{cb} = (H^2, \BMOA)_{cb} = \BMOA^{(2)},
\]  where the \oss on $H^2$ is Pisier's operator Hilbert space determined by \eqref{def-OH}.  In other words, we shall prove the following isometric equalities: 
\begin{align}\label{12-cb-eq}
(H^1, M_d(H^2))_{cb} = (H^2, M_d(\BMOA))_{cb} = M_d(\BMOA^{(2)}) \quad \text{for all $d\ge 1$}, 
\end{align}
where the norm on $M_d(\BMOA^{(2)})$ is determined by \eqref{BMOA-os}.

\begin{lemma}\label{lem-dilation-bis}
For any multiplier $T_m:  H^1\rightarrow M_d(H^2)$, 
\[
\| T_m\|_{(H^1, M_d(H^2))_{cb}} = \sup_{0<r<1}   \| T_{m(r)}\|_{(H^1, M_d(H^2))_{cb}}.
\]
\end{lemma}
\begin{proof}
The proof  follows almost verbatim that of Lemma \ref{lem-dilation}.
\end{proof}

\begin{proof}[Proof of the isometric equality \eqref{12-cb-eq}]
 Fix any integer $d\ge 1$. 
By the duality $(H^1)^* = \BMOA$, the  equality $(H^1, M_d(H^2))_{cb} = (H^2, M_d(\BMOA))_{cb}$ holds isometrically. Hence it suffices to prove the isometric equality $(H^1, M_d(H^2))_{cb} = M_d(\BMOA^{(2)})$.    Following along the same lines of the proof of the equality \eqref{eq-HBMOA-cb}, by Proposition \ref{thm-doubling},  we only need to prove that for any bounded  sequence $m=(m_n)_{n=0}^\infty$  in $M_d$ and any $0<r<1$,   
\[
\| T_{m(r)}\|_{(H^1, M_d(H^2))_{cb}} =  \Big\| \sum_{n=0}^\infty m_n(r) \otimes \overline{m_n(r)}\otimes e^{in\theta} \Big\|_{M_d \otimes_{min}  \overline{M}_d \otimes_{min}\BMOA}^{1/2}.
\]
Now fix any bounded  sequence $m=(m_n)_{n=0}^\infty$  in $M_d$ and any $0<r<1$. 
The condition \eqref{dial-sum} implies  $T_{m(r)} \in (H^1, M_d(H^2))_{cb}$. It also implies that,   under the complete isometric embedding 
\begin{align}\label{tensor-H1-H2}
 \BMOA\otimes_{min} M_d(H^2) = (H^1)^* \otimes_{min} M_d(H^2) \xhookrightarrow[\textit{embedding}]{\textit{complete isometric}} CB(H^1, M_d(H^2)),
\end{align}
the multiplier $T_{m(r)}$ corresponds to the tensor 
\[
\sum_{n=0}^\infty e^{in\theta}\otimes m_n(r)  \otimes  e^{in\theta}\in \BMOA\otimes_{min} M_d \otimes_{min} H^2  = \BMOA\otimes_{min} M_d ( H^2).
\]
Therefore, we have 
\begin{align*}
 & \| T_{m(r)}\|_{(H^1, M_d(H^2))_{cb}}  
 =  \Big\| \sum_{n=0}^\infty  e^{in\theta}\otimes m_n(r) \otimes e^{in\theta} \Big\|_{\BMOA \otimes_{min}M_d(H^2)}  &  (\text{by \eqref{tensor-H1-H2}})
\\
& = \Big\| \sum_{n=0}^\infty \Gamma_n \otimes   m_n(r) \otimes e^{in\theta} \Big\|_{\hank(\ell^2) \otimes_{min} M_d(H^2)}  & \text{(by \eqref{PSN-bis})}
\\
& =  \Big\| \sum_{n=0}^\infty  m_n(r) \otimes \overline{m_n(r)} \otimes \Gamma_n \otimes \Gamma_n \Big\|_{M_d \otimes_{min} \overline{M}_d \otimes_{min}\hank(\ell^2) \otimes_{min}\hank(\ell^2)}^{1/2}  & \text{(by \eqref{def-OH})}
\\
&= \Big\| \sum_{n=0}^\infty   m_n(r) \otimes \overline{m_n(r)} \otimes \Gamma_n  \Big\|_{M_d\otimes_{min} \overline{M}_d\otimes_{min}\hank(\ell^2)}^{1/2}  &   \text{(by Proposition \ref{thm-doubling})}
\\
&  = \Big\| \sum_{n=0}^\infty  m_n(r) \otimes \overline{m_n(r)} \otimes  e^{in\theta} \Big\|_{M_d \otimes_{min} \overline{M}_d\otimes_{min}\BMOA}^{1/2}  &  \text{(again by \eqref{PSN-bis})}.
\end{align*}
This completes the whole proof. 
\end{proof}

\subsubsection{General situations}

For using the complex interpolation method, we  now  prove in Lemmas \ref{prop-bmoap} and \ref{prop-scale} that,  the quantity $\|\cdot\|_{\BMOA^{(p)}}$ given in \eqref{def-bmoap} defines a norm on $\BMOA^{(p)}$ for any $1\le p < \infty$ and the family $\{\BMOA^{(p)}: 1\le p\le \infty\}$ forms a  complex interpolation scale.

\begin{lemma}\label{prop-bmoap}
Let $1\le  p< \infty$ and $q = p/(p-1) \in (1,\infty]$.  Then  for any $\varphi\in \BMOA^{(p)}$ and $\psi\in \BMOA^{(q)}$, we have 
$
\| \varphi*\psi\|_{\BMOA} \le \| \varphi\|_{\BMOA^{(p)}}  \|\psi\|_{\BMOA^{(q)}}.  
$
Moreover, 
\[
\| \varphi\|_{\BMOA^{(p)}} = \sup \big\{\| \varphi*\psi\|_{\BMOA}:  \| \psi\|_{\BMOA^{(q)}}\le 1\big\}. 
\]
In particular, $\| \cdot \|_{\BMOA^{(p)}}$ defines a norm on $\BMOA^{(p)}$. 
\end{lemma}

\begin{proof}
We give the proof for  $1< p<\infty$,  the proof for $p = 1$ is similar.  It suffices to show that for any sequences $(a_n)_{n=0}^\infty, (b_n)_{n=0}^\infty$ of complex numbers, 
\begin{align}\label{bmo-holder}
\Big\|\sum_{n=0}^\infty a_n b_n e^{in\theta}\Big\|_{\BMOA} \le  \Big\|\sum_{n=0}^\infty |a_n|^p e^{in\theta}\Big\|_{\BMOA}^{1/p} \cdot \Big\|\sum_{n=0}^\infty |b_n|^q e^{in\theta}\Big\|_{\BMOA}^{1/q}
\end{align}
and 
\begin{align}\label{sup-norm}
\Big\|\sum_{n=0}^\infty |a_n|^p e^{in\theta}\Big\|_{\BMOA}^{1/p} = \sup\Big\{  \Big \|\sum_{n=0}^\infty a_n b_n e^{in\theta}\Big\|_{\BMOA} :   \Big\|\sum_{n=0}^\infty |b_n|^q e^{in\theta}\Big\|_{\BMOA} \le 1 \Big\}. 
\end{align}
By Nehari's Theorem, the inequality \eqref{bmo-holder} is equivalent to 
\[
\| [a_{i+j}b_{i+j}]\|\le  \| [|a_{i+j}|^p]\|^{1/p}  \cdot \| [|b_{i+j}|^q]\|^{1/q},
\]
where $[c_{i+j}]  = [c_{i+j}]_{i,j \ge 0}$ denotes the infinite Hankel matrix  and  $\|\cdot\|$ denotes the operator norm of $B(\ell^2)$. Since $\| [a_{i+j} b_{i+j}]\| \le \| [|a_{i+j}b_{i+j}|]\|$, we may assume $a_n \ge 0$ and $b_n\ge 0$ for all $n\ge 0$. Now for non-negative coefficients,  the operator $[a_{i+j}b_{i+j}]$ is  self-adjoint and  
\[
\| [a_{i+j}b_{i+j}]\| =  \sup \sum_{i,j\ge 0} a_{i+j} b_{i+j} v_i v_j,
\]
where  the supremum runs over the set of all vectors $v \in \ell^2$ with  non-negative coefficients $v_n\ge 0$ and $\| v\|_{\ell^2} \le 1$.  But for any such vector $v\in \ell^2$,  by  H\"older's inequality,  
\[
\sum_{i,j\ge 0} a_{i+j} b_{i+j} v_i v_j \le  \Big(\sum_{i,j\ge 0} a_{i+j}^p   v_i v_j\Big)^{1/p}  \Big(\sum_{i,j\ge 0}  b_{i+j}^q v_i v_j \Big)^{1/q}
 \le \| [a_{i+j}^p]\|^{1/p} \| [b_{i+j}^q]\|^{1/q}. 
\]
This completes the proof of the  inequality \eqref{bmo-holder}.  

Finally, by considering 
$
b_n =  \frac{\bar{a}_n}{|a_n|} | a_n|^{p-1} \mathds{1}(a_n\ne 0),
$ we obtain the  equality  \eqref{sup-norm}
\end{proof}

\begin{lemma}\label{prop-scale}
Let $1\le p_0, p_1\le  \infty$ and $\alpha \in (0,1)$. Set $1/p_\alpha = (1-\alpha)/p_0 + \alpha/p_1$,  then 
\[
(\BMOA^{(p_0)}, \BMOA^{(p_1)})_\alpha = \BMOA^{(p_\alpha)} \quad \text{with equal norms.}
\]
\end{lemma}

\begin{proof}
Assume that $p_0 \ne  p_1$ and hence $p_\alpha \in (1, \infty)$.  Let $S = \{ z\in \C| 0< \Re(z)<1\}$. 
Assume that $\| \sum_{n\ge 0} |m_n|^{p_\alpha} e^{in\theta}\|_{\BMOA} < 1$. For any $z \in \overline{S}$, set
\[
m_n(z) =  \frac{m_n}{|m_n|} |m_n|^{p_\alpha (1-z)/p_0  + p_\alpha z /p_1} \mathds{1}(m_n \ne 0), \quad n \ge 0. 
\]
Then $m_n(\alpha)= m_n$. Clearly, for any $t\in \R$, we have 
\[
\Big\| \sum_{n\ge 0} m_n(it) e^{in \theta}\Big\|_{\BMOA^{(p_0)}} <1 \an  \Big \| \sum_{n\ge 0}m_n(1+it) e^{in \theta}\Big\|_{\BMOA^{(p_1)}} < 1. 
\]
The contractive embedding $\BMOA^{(p_\alpha)} \subset (\BMOA^{(p_0)}, \BMOA^{(p_1)})_\alpha$ then  follows.

 Conversely, assume that 
$
\| \sum_{n\ge 0} m_n e^{in\theta} \|_{(\BMOA^{(p_0)}, \BMOA^{(p_1)})_\alpha} < 1.
$
Then there exist functions $m_n: \overline{S}\rightarrow \C$, analytic on $S$ with $m_n(\alpha) = m_n$, and 
\[
\sup_{t\in \R} \Big\|\sum_{n\ge 0} |m_n(it)|^{p_0} e^{in\theta}\Big\|_{\BMOA}<1 \an   \sup_{t\in \R}\Big\|\sum_{n\ge 0} |m_n(1+it)|^{p_1} e^{in\theta} \Big\|_{\BMOA}<1. 
\] 
Now take any sequence $(a_n)_{n\ge 0}$ in $\C$ with  $\| \sum_{n\ge 0} |a_n|^{q_\alpha} e^{in\theta}\|_{\BMOA} < 1$ and $q_\alpha= p_\alpha/(p_\alpha-1)$. By the previous argument, we can construct functions $a_n: \overline{S}\rightarrow \C$, analytic on $S$ with $a_n(\alpha)= a_n$, and 
\[
\sup_{t\in \R} \Big\|\sum_{n\ge 0} |a_n(it)|^{q_0} e^{in\theta}\Big\|_{\BMOA}<1 \an   \sup_{t\in \R}\Big\|\sum_{n\ge 0} |a_n(1+it)|^{q_1} e^{in\theta} \Big\|_{\BMOA}<1,
\] 
where $q_0= p_0/(p_0-1), q_1 = p_1/(p_1-1)$. Therefore, by the inequality \eqref{bmo-holder}, 
\[
\sup_{z\in \partial S}\Big\|\sum_{n=0}^\infty a_n (z) m_n(z) e^{in\theta}\Big\|_{\BMOA} <1. 
\]
It follows immediately that 
\[
\Big\|\sum_{n=0}^\infty a_n m_n e^{in\theta}\Big\|_{\BMOA} <1. 
\]
Then, by applying the equality \eqref{sup-norm}, we obtain $\|\sum_{n=0}^\infty m_n e^{in\theta}\|_{\BMOA^{(p_\alpha)}} <1$. The desired inverse contractive embedding $ (\BMOA^{(p_0)}, \BMOA^{(p_1)})_\alpha \subset \BMOA^{(p_\alpha)}$ now follows.
\end{proof}

\begin{proof}[Proof of Theorem \ref{cor-interpolation}]
Observe that, for any $1\le r, s\le\infty$, the equality 
\[
(H^r, H^s)_{cb} = (H^r, L^s)_{cb}
\]
 holds  completely isometrically.   By Proposition~\ref{thm-hardy},  $\BMOA^{(2)}   = (H^1, H^2)_{cb}  = (H^1, L^2)_{cb}$  with equal norms (in fact, the equalities hold completely isometrically). On the other hand,  
\[
\BMOA^{(1)} \xrightarrow{\textit{contractive}} \BMOA  = (H^1, L^\infty)_{cb}. 
\]
Thus,   by the standard complex interpolation method, we obtain  that 
\[
(\BMOA^{(1)}, \BMOA^{(2)})_\theta \xrightarrow{\textit{contractive}}  (H^1, (L^\infty, L^2)_\theta)_{cb} \quad \text{for all $\theta \in [0,1]$}. 
\]
Therefore, by Lemma \ref{prop-scale}, for any $2\le  q \le \infty$, we have 
\begin{align}\label{1-to-p}
\BMOA^{(q')} \xrightarrow{\textit{contractive}} (H^1, L^q)_{cb} = (H^1, H^q)_{cb}  \text{\,\,\,with\,\,\,} \frac{1}{q}+\frac{1}{q'}= 1. 
\end{align}
This completes the first assertion of Theorem \ref{cor-interpolation}. 

Now  we proceed to the proof of the bounded inclusion \eqref{gen-rs}. By  interpolating \eqref{1-to-p} with $\BMOA^{(2)} = (H^2, \BMOA)_{cb}$,  we obtain 
\begin{align}\label{inter-c-b}
(\BMOA^{(2)}, \BMOA^{(q')})_\theta \xrightarrow{\textit{contractive}}  \big((H^2, H^1)_\theta, (\BMOA, H^q)_\theta\big)_{cb}.
\end{align}

{\flushleft\bf Observation: } for any $1< r < \infty$ and any $\theta\in (0, 1)$, we have 
\begin{align}\label{c-iso}
H^{r_\theta} \xrightarrow[\simeq]{\textit{complete isomorphism}}  (H^1, H^r)_\theta
\end{align}
and its dual form 
 \begin{align}\label{dual-c-iso}
H^{r_\theta'} \xrightarrow[\simeq]{\textit{complete isomorphism}} (\BMOA, H^{r'})_\theta,
\end{align}
where \[
r' = \frac{r}{r-1}, \quad    \frac{1}{r_\theta} = \frac{1-\theta}{1} + \frac{\theta}{r}  \an   r_\theta' =\frac{r_\theta}{r_\theta-1}.
\]

Lemma \ref{prop-scale} combined with  \eqref{inter-c-b}, \eqref{c-iso} and  \eqref{dual-c-iso} implies 
\[
\BMOA^{(p_\theta)}  \xhookrightarrow[\textit{inclusion}]{\textit{bounded}}  (H^{r_\theta}, H^{s_\theta})_{cb}, 
\]
where $\theta \in (0,1)$ and 
\begin{align}\label{exp-inter}
\frac{1}{p_\theta} = \frac{1-\theta}{2} + \frac{\theta}{q'}, \quad \frac{1}{r_\theta} = \frac{1-\theta}{2} + \frac{\theta}{1}, \quad \frac{1}{s_\theta} = \frac{1- \theta}{\infty} +\frac{\theta}{q}
\end{align}
with $2\le q\le \infty$ and $q'  = q/(q-1) \in [1, 2]$. 
\begin{figure}[h]
\begin{center}
\includegraphics[width=0.28\linewidth]{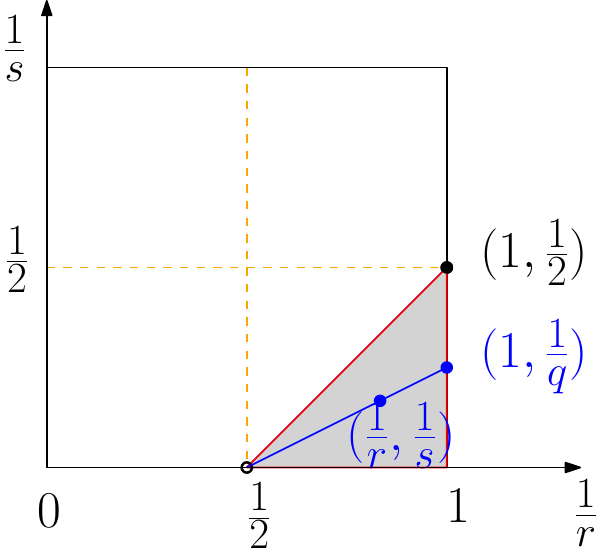}
\caption{$(\frac{1}{r}, \frac{1}{s}) = (1- \theta) (\frac{1}{2}, 0) + \theta (1, \frac{1}{q})$ with $q \in [2, \infty],\theta \in (0,1)$.}\label{fig-interpolation}
\end{center}
\end{figure}

\noindent Now from Figure \ref{fig-interpolation}, it is clear that for any pair of exponents $(r, s)$ satisfying
\begin{align}\label{rs-low}
\frac{1}{r} \ge \frac{1}{s}+ \frac{1}{2} \an  1< r <  2< s<\infty, 
\end{align}   
we can find 
\[
\theta \in (0,1), \, r_\theta = r, \, s_\theta= s, \, \frac{1}{p_\theta} = \frac{1}{r}-\frac{1}{s}\an q\in [2, \infty], 
\]
satisfying  the equation system \eqref{exp-inter}. The desired bounded inclusion \eqref{gen-rs} follows for any pair of exponents $(r, s)$ satisfying \eqref{rs-low}.  

Note also that for $r=1$ and any $2\le s\le\infty$, the bounded inclusion  \eqref{gen-rs} holds.  Finally, by interpolating the above result with the  trivial isometric equality 
$
(H^2, H^2)_{cb}  = \BMOA^{(\infty)},
$
we extend the bounded inclusion \eqref{gen-rs} to all pairs $(r, s)$ with $1\le r\le 2 \le s<\infty$.

\begin{proof}[Proof of the Observation] Indeed, by \cite[Thm. 4.3]{Xu-studia1990} (see also  \cite[Remark 5.3]{Xu-Blasco-interpolation1991}), 
\begin{align}\label{Xu-id}
(H^1(S_1), H^r(S_r))_\theta \xlongequal[\textit{norms}]{\textit{equivalent}} H^{r_\theta}(S_{r_\theta}).
\end{align}
  On the other hand,   by \cite[Cor. 1.4, formula (1.5)]{Pisier-Non-Commutative-vector}  and the equality \eqref{nc-fubini}, we have 
\begin{align}\label{pisier-id}
(H^1(S_1), H^r(S_r))_\theta = (S_1[H^1], S_r[H^r])_\theta = S_{r_\theta} [(H^1, H^r)_\theta] \quad \text{with equal norms}. 
\end{align}
The relations \eqref{Xu-id} and \eqref{pisier-id} together imply 
\[
H^{r_\theta}(S_{r_\theta}) = S_{r_\theta}[H^{r_\theta}] = S_{r_\theta} [(H^1, H^r)_\theta] \quad \text{with equivalent norms}.
\]
Consequently, by \eqref{cb-Sp}, we obtain  \eqref{c-iso}.

For the dual form \eqref{dual-c-iso}, since  $H^r(S_r)  = S_r[H^r]$ is reflexive for any $1<r<\infty$, we may apply the  duality theory for complex interpolation \cite[Thm. 8.37]{Pisier-martingale-in-Banach2016} to obtain the dual of $(S_1^n[H^1], S_r^n[H^r])_\theta$. Thus, by \cite[Cor. 1.4, formula (1.5)]{Pisier-Non-Commutative-vector}, for any $n\ge 1$, 
\[
S_{r_\theta'}^n[(\BMOA, H^{r'})_\theta] = (S_\infty^n[\BMOA], S_{r'}^n[H^{r'}])_\theta  =  \Big((S_1^n[H^1], S_r^n[H^r])_\theta\Big)^* \,\, \text{with equal norms.}
\]
Then by  \eqref{Xu-id} and \eqref{pisier-id} and the classical result on the dual of $H^p(S_p)$, we have 
\[
S_{r_\theta'}^n[(\BMOA, H^{r'})_\theta]  \xrightarrow[\simeq]{\textit{isomorphism}}   \Big(H^{r_\theta}(S_{r_\theta}^n) \Big)^*  \xrightarrow[\simeq]{\textit{isomorphism}}   H^{r_\theta'} (S_{r_\theta'}^n)  = S_{r_\theta'}^n[H^{r_\theta'}], 
\]
where the  norm equivalence constants in the above isomorphisms are uniformly bounded.  Therefore,  by \eqref{cb-Sp}, we complete the proof of the complete isomorphism \eqref{dual-c-iso}. 
\end{proof}

This completes the whole proof of the theorem. 
\end{proof}

\subsection{Bounded Hankel operators induce cb-Fourier multipliers}
The second complete isometric equality in Proposition~\ref{thm-hardy} has a higher dimensional analogue. Here we only state the isometric counterpart, the completely isometric counterpart requires a definition of the \oss on the corresponding space and will not be treated here.

For any  scalar function $\varphi$ in $H^1(\T^d)$, set 
\[
\varphi^\dag(z_1, \cdots, z_d):= \overline{ \varphi(z_1, \cdots, z_d)}.
\]

\begin{corollary}\label{cor-higher-2-inf}
A function $\varphi \in H^1(\T^d)$ induces a completely bounded Fourier multiplier in $(H^2(\T^d), \hank(\overline{H^2(\T^d)}, H^2(\T^d)))_{cb}$ if and only if $\Gamma_{\varphi*\varphi^\dag} \in \hank(\overline{H^2(\T^d)}, H^2(\T^d))$. Moreover, 
\[
\| \varphi\|_{(H^2(\T^d), \hank(\overline{H^2(\T^d)}, H^2(\T^d)))_{cb}} =  \| \Gamma_{\varphi*\varphi^\dag}\|_{\hank(\overline{H^2(\T^d)}, H^2(\T^d))}^{1/2}. 
\]
\end{corollary} 
\begin{proof}
The proof  follows almost verbatim that of the second isometric equality  in Proposition~\ref{thm-hardy}, where the SAP of $\N$ is replaced by the SAP of $\N^d$. 
\end{proof}

\begin{corollary}\label{cor-S4-hankel}
Assume that   $\varphi\in H^1(\T^d)$ with $\Gamma_{\varphi*\varphi^\dag}\in \hank(\overline{H^2(\T^d)}, H^2(\T^d))$. Then for any sequence $(f_k)_{k\ge 1}$ in $H^2(\T^d)$, we have 
\begin{align*}
\big\|\sum_{k=1}^\infty (\Gamma_{f_k*\varphi})^*  \Gamma_{f_k*\varphi} \big\|_{B(\overline{H^2(\T^d)})}^{1/2} \le  \|\Gamma_{\varphi*\varphi^\dag} \|^{1/2}_{\hank(\overline{H^2(\T^d)}, H^2(\T^d))} \big(\sum_{k,l=1}^\infty |\langle f_k, f_l\rangle|^2\big)^{1/4}.   
\end{align*}
In particular, if $(f_k)_{k\ge 1}$ are mutually orthogonal, then 
\[
\big\|\sum_{k=1}^\infty (\Gamma_{f_k*\varphi})^*  \Gamma_{f_k*\varphi} \big\|_{B(\overline{H^2(\T^d)})}^{1/2} \le  \|\Gamma_{\varphi*\varphi^\dag} \|^{1/2}_{\hank(\overline{H^2(\T^d)}, H^2(\T^d))} \big(\sum_{k=1}^\infty \| f_k\|_{H^2(\T^d)}^4\big)^{1/4}.   
\]
\end{corollary}

\begin{proof}
 Take any  $\varphi\in H^1(\T^d)$ with $\Gamma_{\varphi*\varphi^\dag}\in \hank(\overline{H^2(\T^d)}, H^2(\T^d))$.
By Corollary~\ref{cor-S4-hankel}, the Fourier multiplier $T_\varphi$ belongs to $(H^2(\T^d), \hank(\overline{H^2(\T^d)}, H^2(\T^d)))_{cb}$ and we have 
\[
\big\|  C\otimes_{min} H^2(\T^d)    \xrightarrow{Id \otimes T_\varphi} C\otimes_{min} \hank(\overline{H^2(\T^d)}, H^2(\T^d))\big \| \le  \|\Gamma_{\varphi*\varphi^\dag} \|_{\hank(\overline{H^2(\T^d)}, H^2(\T^d))}^{1/2}.
\]
Recall that for any operator space $E$, the  tensor product $C\otimes_{min} E$ coincides with the Haagerup tensor product $C\otimes_h E$.  Since $H^2(\T^d)$ is equipped with $OH$ structure, then by applying the complex interpolation (see  \cite[Thm. 5.22 and Cor. 7.11]{Pisier-Operator-space-book}), we have 
\[
C\otimes_h H^2(\T^d) \xrightarrow[\textit{isometric}]{\textit{complete}} C\otimes_h (R, C)_{\frac{1}{2}}  \xrightarrow[\textit{isometric}]{\textit{complete}}(C\otimes_h R, C\otimes_h C)_{\frac{1}{2}}. 
\]
In the setting of Banach spaces, we have the isometric isomorphism 
\[
(C\otimes_h R, C\otimes_h C)_{\frac{1}{2}} \xrightarrow[\simeq]{\textit{isometric}} S_4. 
\]
The above isometric isomorphism leads to the equality: 
\[
\|(f_1, f_2, \cdots, f_n, \cdots)^T\|_{C\otimes_{min} H^2(\T^d)} = \big(\sum_{k,l=1} | \langle f_k, f_l\rangle|^2 \big)^{1/4}
\]
and we  complete the whole proof. 
\end{proof}

\subsection{Fourier-Schur multiplier inequalities with critical exponent $4/3$}

\begin{corollary}\label{cor-FS-mul}
Let $\varphi \in \BMOA^{(2)}$. Then for any  $\ell^2$-column-vector-valued  $f\in H^1(\ell^2)$, 
\[
\Big\| \Big(\int_\T \big[(f*\varphi)  (f*\varphi)^*\big] dm \Big)^{1/2} \Big\|_{S_{4/3}} \le   \| \varphi \|_{\BMOA^{(2)}} \cdot \| f\|_{H^1(\ell^2)}. 
\]
\end{corollary}

\begin{remark}\label{rk-FS}
The inequalities obtained in Corollary \ref{cor-FS-mul} are called  Fourier-Schur multiplier inequalities since they  can be rewritten as 
\begin{align*}
\Big\| \Big(  \widehat{\varphi}(0) \widehat{f}(0),  \widehat{\varphi}(1) \widehat{f}(1), \cdots, \widehat{\varphi}(n) \widehat{f}(n), \cdots\Big)\Big\|_{S_{4/3}} \le  \| \varphi\|_{\BMOA^{(2)}} \cdot \| f\|_{H^1(\ell^2)}. 
\end{align*}
\end{remark}

The exponent $4/3$ in Corollary  \ref{cor-FS-mul} is optimal, this can be seen from 
\begin{proposition}\label{prop-43-optimal}
Let $p>0$. Assume that there exist $q\in [1,\infty)$ and a numerical constant $c> 0$ such that for any $\varphi\in \BMOA^{(2)}$ and any $f\in H^q(\ell^2)$, 
\[
\Big\| \Big(\int_\T \big[(f*\varphi)  (f*\varphi)^*\big] dm \Big)^{1/2} \Big\|_{S_{p}} \le   c  \| \varphi \|_{\BMOA^{(2)}} \cdot \| f\|_{H^q(\ell^2)}. 
\]
Then $p \ge 4/3$. 
\end{proposition}

\begin{proof}[Proof of Corollary \ref{cor-FS-mul}]
  Let $C$ (resp. $R$) denote the column (resp. row) Hilbert space, which by definition, consists of all bounded operators in $B(\ell^2)$ whose all but the first column  (resp. row) entries vanish. 

Take any $\varphi\in \BMOA^{(2)}$. By Proposition~\ref{thm-hardy},  the Fourier multiplier $T_{\widehat{\varphi}}$ associated with the sequence $(\widehat{\varphi}(n))_{n\ge 0}$  belongs to $(H^1, H^2)_{cb}$.  Hence, by using the Haagerup tensor product $\otimes_h$ (see \cite[Chapter 5]{Pisier-Operator-space-book}),  we have 
\[
\| H^1 \otimes_h C \xrightarrow{T_{\widehat{\varphi}} \otimes Id} H^2 \otimes_h C\| \le \| T_{\widehat{\varphi}}\|_{(H^1, H^2)_{cb}} = \|\varphi\|_{\BMOA^{(2)}}. 
\]
By noticing the following relations (see \cite[p.18, formula (1.1) and Prop. 2.1]{Pisier-Non-Commutative-vector})
\[
H^1\otimes_h C \xhookrightarrow[\textit{embedding}]{\textit{isometric}} R\otimes_h H^1 \otimes_h C  \xrightarrow[\simeq]{\textit{isometric}} S_1[H^1] \xrightarrow[\simeq]{\textit{isometric}} H^1(S_1),
\]
we obtain an isometric isomorphism (as Banach spaces): 
\[
H^1 \otimes_h C \xrightarrow[\simeq]{\textit{isometric}} H^1(\ell^2). 
\]
Now since $H^2$ is equipped with the $OH$ structure, by applying the complex interpolation (see  \cite[Thm. 5.22 and Cor. 7.11]{Pisier-Operator-space-book}), we have 
\[
H^2 \otimes_h C \xrightarrow[\textit{isometric}]{\textit{complete}}(R, C)_{\frac{1}{2}} \otimes_h C \xrightarrow[\textit{isometric}]{\textit{complete}}(R\otimes_h C, C\otimes_h C)_{\frac{1}{2}}. 
\]
In the setting of Banach spaces, we have the following isometric isomorphisms (in fact, the first one is completely isometric): 
\[
R\otimes_h C \xrightarrow[\simeq]{\textit{isometric}} S_1 \an   C\otimes_h C \xrightarrow[\simeq]{\textit{isometric}} S_2,
\]
hence 
\[
(R\otimes_h C, C\otimes_h C)_{\frac{1}{2}} \xrightarrow[\simeq]{\textit{isometric}} (S_1, S_2)_{\frac{1}{2}} = S_{4/3}. 
\]
It follows that 
\[
\Big\|H^1(\ell^2) \xrightarrow[\simeq]{\textit{isometric}}   H^1 \otimes_h C \xrightarrow{T_{\widehat{\varphi}} \otimes Id} H^2 \otimes_h C  \xrightarrow[\simeq]{\textit{isometric}}S_{4/3}\Big\| \le \|\varphi\|_{\BMOA^{(2)}}. 
\]
In other words,  the following  map 
\begin{align*}
\begin{array}{ccc}
H^1(\ell^2)&  \xrightarrow{\qquad} & S_{4/3}
\\
f & \mapsto &  (\widehat{\varphi}(0) \widehat{f}(0), \cdots, \widehat{\varphi}(n) \widehat{f}(n), \cdots)
\end{array}
\end{align*}
is bounded with norm no larger than $\|\varphi\|_{\BMOA^{(2)}}$.  This completes the whole proof. 
\end{proof}

For proving  Proposition~\ref{prop-43-optimal}, we need the following classical results on the lacunary Fourier series. 

\begin{lemma}[{see, e.g., \cite[Prop. 3.1]{Blasco-TAMS}}]\label{lem-lacunary-Hq}
For any $q\in [1, \infty)$, there exists a numerical constant $c \ge 1$ such that for any sequence $(v_n)_{n=0}^\infty$ of vectors in $\ell^2$, 
\[
 \frac{1}{c}   \Big(\sum_{n = 0}^\infty \| v_n\|_{\ell^2}^2 \Big)^{1/2} \le \big\| \sum_{n = 0}^\infty v_n e^{i 2^n \theta}\big\|_{H^q(\ell^2)} \le c \Big(\sum_{n = 0}^\infty \| v_n\|_{\ell^2}^2 \Big)^{1/2}. 
\]
\end{lemma}

\begin{lemma}[{see, e.g., \cite[Thm. 6.12]{Pavlovic-book}}]\label{lem-lacunary-bmo}
There exists a numerical constant $c>0$ such that for any  complex sequence $(a_n)_{n = 0}^\infty$,
\[
\big\| \sum_{n= 0}^\infty a_n e^{i 2^n \theta}\big\|_{\BMOA} \le   c \Big(\sum_{n = 0}^\infty |a_n|^2 \Big)^{1/2}.
\]
\end{lemma}

\begin{proof}[Proof of Proposition \ref{prop-43-optimal}]
Fix any $q \in [1, \infty)$.  By assumption, 
\[
\Big\| \Big(\sum_{n=0}^\infty |\widehat{\varphi}(n)|^2 \widehat{f}(n) \widehat{f}(n)^*\Big)^{1/2}\Big\|_{S_p} \le c \Big\| \sum_{n  =0}^\infty |\widehat{\varphi}(n)|^2 e^{in \theta}\Big\|_{\BMOA}^{1/2}  \| f\|_{H^q(\ell^2)}.
\]
In the above inequality, by taking $\varphi\in \BMOA^{(2)}$ and $f\in H^q(\ell^2)$ with  Fourier supports on the set of dyadic integers, we obtain 
\[
\Big\| \Big(\sum_{n=0}^\infty |d_n|^2 v_n v_n^*\Big)^{1/2}\Big\|_{S_p} \le c \Big\| \sum_{n  =0}^\infty |d_n|^2 e^{i 2^n \theta}\Big\|_{\BMOA}^{1/2}  \Big\| \sum_{n = 0}^\infty  v_n e^{i 2^n \theta}\Big\|_{H^q(\ell^2)}.
\]
Then by Lemmas \ref{lem-lacunary-Hq} and \ref{lem-lacunary-bmo}, there exists a numerical constant $c'> 0$ such that 
\begin{align*}
  \Big\| \Big(\sum_{n=0}^\infty |d_n|^2 v_n v_n^*\Big)^{1/2}\Big\|_{S_p}  \le c' \Big(\sum_{n  =0}^\infty |d_n|^4 \Big)^{1/4}  \Big(\sum_{n = 0}^\infty  \|v_n\|_{\ell^2}^2\Big)^{1/2}.
\end{align*}
Hence, by setting  $v_n = \lambda_n e_n$ with $\lambda_n \in \C$ and $e_n\in \ell^2$ the natural basis vector of $\ell^2$, we have 
\[
\Big(\sum_{n =0}^\infty |d_n\lambda_n|^p \Big)^{1/p} \le c' \Big(\sum_{n  =0}^\infty |d_n|^4 \Big)^{1/4}  \Big(\sum_{n = 0}^\infty  |\lambda_n|^2\Big)^{1/2}. 
\]
Clearly, the above inequality holds for all $(d_n)_{n\ge 0}$ and $(\lambda_n)_{n\ge 0}$ only if $p \ge 4/3$.  
\end{proof}

\subsection{Exact complete bounded norms of Carleman embeddings}
Let $L_a^2(\D)$ be the Bergman space on $\D$ with respect to  the normalized area measure $dA(z)= \frac{dxdy}{\pi}$. Recall the famous Carleman inequality: for any $f\in H^1$, we have $
\| f\|_{L_a^2(\D)} \le \| f\|_{H^1}$. 
 Namely, the Carleman embedding 
\[
i: H^1\rightarrow L_a^2(\D)
\] is contractive (see \cite{Gamelin-Varleman-embedding-monthly}  for an elementary proof). 

The space $L_a^2(\D)$ carries the \oss  inherited from that of $L^2(\D)$. Since the natural \oss on $L^2(\D)$ is the Pisier's operator Hilbert space structure $OH$, by the uniqueness of $OH$,  we have  
\begin{align}\label{berg-hardy-ciso}
\begin{array}{cccc}
M: & L_a^2(\D)& \xrightarrow[\simeq]{\textit{completely isometric}}& H^2
\\
&  \sqrt{n+1} z^n & \mapsto &  e^{in\theta}
\end{array}.
\end{align}

\begin{corollary}\label{cor-carleman}
$\|i: H^1\rightarrow L_a^2(\D)\|_{cb} = \sqrt{\pi}$. 
\end{corollary}

\begin{remark}
Applying   \eqref{cb-Sp} to the Carleman embedding $i: H^1\rightarrow L_a^2(\D)$ and using  the duality $(S_\infty[L_a^2(\D)])^* = S_1[L_a^2(\D)]$ together with  the non-commutative Fubini theorem \eqref{nc-fubini}, we can reformulate Corollary \ref{cor-carleman} as follows:  for any $f\in H^1(S_1)$ and $g \in  \C[z] \otimes S_\infty$ (with $\sqrt{\pi}$ optimal):
\begin{align}\label{cb-carleman-ineq}
\Big| \int_\D \Tr \big(f(z) \overline{g(z)}\big) dA(z) \Big| \le \sqrt{\pi} \| f\|_{H^1(S_1)}  \Big\|  \int_\D  \big[ g(z) \otimes \overline{g(z)}  \big] dA(z) \Big\|_{B(\ell^2\otimes_2 \ell^2)}^{1/2}.
\end{align}
\end{remark}

\begin{proof}[Proof of Corollary \ref{cor-carleman}] 
By \eqref{berg-hardy-ciso}, $\| i: H^1\rightarrow L_a^2(\D)\|_{cb} = \| M\circ i: H^1\rightarrow H^2\|_{cb}$. Clearly, $M\circ i$ coincides with the Fourier multiplier corresponding to  the sequence $(1/\sqrt{n+1})_{n=0}^\infty$. Hence, by the second isometric equality in Proposition~\ref{thm-hardy}, we have 
\begin{align}\label{carleman-bmoa-fn}
\| i: H^1\rightarrow L_a^2(\D)\|_{cb}  =  \Big\| \sum_{n=0}^\infty \frac{1}{\sqrt{n+1}} e^{i n\theta} \Big\|_{\BMOA^{(2)}} = \Big\| \left [\frac{1}{i+j+1} \right]_{i, j \ge 0} \Big\|_{B(\ell^2)}^{1/2} = \sqrt{\pi},
\end{align}
where  the operator norm of the  Hilbert matrix $[1/(i+j+1)]_{i,j\ge 0}$ is determined by Schur and can be found in, e.g.,  \cite[p.226]{HLP-inequality} or \cite{Magnus-AJM}. 
\end{proof}

For any integer $d\ge 2$, write $L_a^2(\D^d)=L_a^2(\D^d, dA^{\otimes d})$ with $dA^{\otimes d}(z) = dA(z_1)\cdots dA(z_d)$. Let  
$
i_d: H^1(\T^d) \rightarrow L_a^2(\D^d)
$
 dentoe the canonical embedding.  
The  higher dimensional analogue of Corollary~\ref{cor-carleman} is 

\begin{corollary}\label{cor-higher-carleman}
For any integer $d \ge 2$, we have $\|i_d: H^1(\T^d) \rightarrow L_a^2(\D^d)\|_{cb} =  \pi^{d/2}$. 
\end{corollary}

\begin{proof}  
By  the norm equality \eqref{carleman-bmoa-fn} of the Hilbert matrix  and the Nehari Theorem, there exists a function $\psi \in L^\infty(\T)$ with $\| \psi\|_{L^\infty(\T)} = \pi$ such that 
$\widehat{\psi}(n) = \frac{1}{n+1}$ for all $n\ge 0$.
Define $\psi_d(z) = \psi(z_1) \cdots \psi(z_d)$. Then 
$
\| \psi_d\|_{L^\infty(\T^d)} = \pi^d. 
$  Write $e_n(z_j) = z_j^n$ for $n\ge 1$ and $e_n(z_j)= \overline{z_j}^{|n|}$ for $n<0$.  For any $\alpha = (\alpha_1,\cdots, \alpha_d)\in\Z^d$, denote $|\alpha| = |\alpha_1|+\cdots+|\alpha_d|$ and $z^\alpha = e_{\alpha_1}(z_1) \cdots e_{\alpha_d}(z_d)$. Since $L_a^2(\D^d)$ carries Pisier's operator Hilbert space structure and the operator space dual of $H^1(\T^d)$ is given by $(H^1(\T^d))^*  = L^\infty(\T^d)/H^1(\T^d)^\circ$ (where $H^1(\T^d)^\circ$ is the annihilator of $H^1(\T^d)$ in $L^\infty(\T^d)$), we have 
\begin{align*}
& \|i_d: H^1(\T^d) \rightarrow L_a^2(\D^d)\|_{cb}
\\
   =&   \sup_{0<r<1} \big\| \sum_{\alpha \in \N^d}  \prod_{j=1}^d \frac{ r^{|\alpha|}}{\alpha_j+1}  z^\alpha \otimes  \overline{z^\alpha} \big\|_{(L^\infty(\T^d)/H^1(\T^d)^\circ) \otimes_{min}( \overline{ L^\infty(\T^d)/H^1(\T^d)^\circ})}^{1/2}
\\
\le & \sup_{0<r<1} \big\| \sum_{\alpha \in \Z^d}  r^{|\alpha|}  \widehat{\psi_d}(\alpha)   z^\alpha \otimes  \overline{z^\alpha} \big\|_{L^\infty(\T^d) \otimes_{min} \overline{L^\infty(\T^d)}}^{1/2}  
\\
=&   \sup_{0<r<1} \big\| \sum_{\alpha \in \Z^d}  r^{|\alpha|}\widehat{\psi_d}(\alpha)   z^\alpha   \overline{w^\alpha} \big\|_{L^\infty(\T^d \times \T^d)}^{1/2} 
\le    \| \psi_d\|_{L^\infty(\T^d \times \T^d)}^{1/2} = \pi^{d/2}. 
\end{align*}

Conversely, by Corollary~\ref{cor-carleman}, for any $\varepsilon>0$,  there exist  functions  $F\in S_1[H^1(\T)]=  H^1(\T; S_1)$  and $G\in S_\infty[L_a^2(\D)]$ such that 
\[
\| F \|_{S_1[H^1(\T)]} = 1, \quad \big \| \int_{\D} G(z_1) \otimes \overline{G(z_1)} dA(z_1) \big \|_{B(\ell^2\otimes_2 \ell^2)} = 1
\]
and 
\[
\int_\D\Tr  \big(F(z_1) \overline{G(z_1)}\big) dA(z_1)  \ge \pi^{1/2}-\varepsilon.
\]
Now we define an $S_1(\ell^2\otimes_2 \cdots \otimes_2 \ell^2)$-valued function 
$
F_d(z): = F(z_1)\otimes \cdots \otimes F(z_d)
$ and an $S_\infty(\ell^2\otimes_2 \cdots \otimes_2 \ell^2)$-valued function $G_d(z): = G(z_1)\otimes \cdots \otimes G(z_d)$. 
Then, 
\[
\| F_d\|_{H^1(\T^d; S_1(\ell^2\otimes_2 \cdots \otimes_2 \ell^2))} = \int_{\T^d} \|F(z_1)\|_{S_1} \cdots \| F(z_d)\|_{S_1} dm(z_1)\cdots dm(z_d) = 1
\]
and 
\begin{align*}
&\big\|\int_{\D^d} G_d(z) \otimes \overline{G_d(z)} dA^{\otimes d}(z)\big\|_{B(\ell^2 \otimes_2 \cdots \otimes_2\ell^2)}
\\
 =&\big\|\big[\int_{\D} G(z_1) \otimes \overline{G(z_1)} dA(z_1)\big]  \otimes \cdots \otimes \big[\int_{\D} G(z_d) \otimes \overline{G(z_d)} dA(z_d)\big] \big\|_{B(\ell^2 \otimes_2 \cdots \otimes_2\ell^2)} = 1.
\end{align*}
On the other hand, 
\begin{align*}
\|F_d\|_{S_1[L_a^2(\D^d)} & \ge \int_{\D^d}\Tr  \big(F_d(z) \overline{G_d(z)}\big) dA^{\otimes d}(z)  
\\
& = \big(\int_\D\Tr  \big(F(z_1) \overline{G(z_1)}\big) dA(z_1)\big)^d \ge (\pi^{1/2}-\varepsilon)^d.
\end{align*}
Since $\varepsilon>0$ is arbitrary, this completes the proof of the inverse inequality. 
\end{proof}

\subsection{Lifting property of Fourier multipliers in $(H^1, \BMOA)_{cb}$}

\begin{corollary}\label{cor-hardy-one-inf}
A  Fourier multiplier $T_m$ belongs to $(H^1, \BMOA)_{cb}$ if and only if there exists a function $\psi\in L^\infty$ such that $T_m$ satisfies the following commutative diagram : 
\begin{equation}\label{cd-diag}
\begin{tikzcd}
L^1  \arrow{r}{K_\psi}& L^\infty\arrow[d, "Q" swap, "\text{quotient}"]
\\
H^1 \arrow{r}{T_m}\arrow[u, "\text{inclusion}" swap, "\mathcal{I}"]&  L^\infty/\overline{H_0^\infty} \simeq \BMOA 
\end{tikzcd},
\end{equation}
where $K_\psi(f)= f*\psi$. Moreover, 
$
\|T_m\|_{cb} = \inf \big\{ \|\psi\|_{\infty}\big| \text{$\psi$ satisfies \eqref{cd-diag}}\big\}$. 
\end{corollary}

The above lifting property \eqref{cd-diag} of Fourier multipliers in $(H^1, \BMOA)_{cb}$ implies
\begin{corollary}\label{cor-iso}
$
(H^1, \BMOA)_{cb} = (H^1, H^\infty)$ with equal norms.  
\end{corollary}

\begin{remark}
Corollary~\ref{cor-iso} can also be proved by combining  Proposition~\ref{thm-hardy} and 
the  equality $(H^1, H^\infty) = \BMOA$.
It is worthwhile to mention that  a general multiplier in $(H^1, \BMOA)$   need not to be  liftable to a bounded one from $H^1$ to $H^\infty$, since 
\[
(H^1, H^\infty) \xlongequal{isometric} \BMOA  \ne  \mathcal{B}  \xlongequal{isomorphic} (H^1, \BMOA)  , 
\]
 where $\mathcal{B}$ is the Bloch space on $\D$ (see \cite[Section 1.3]{Hedenmalm-Zhukehe-Berman-space-GTM}) and  the last isomorphic equality is due to  Mateljevi\'c and  Pavlovi\'c ( \cite{Mateljevic-Pacific-journal-math1990} and  \cite[Thm. 11.10]{Pavlovic-book}).  
\end{remark}

\begin{proof}[Proof of Corollary \ref{cor-hardy-one-inf}]
For any $T_m\in (H^1, \BMOA)_{cb}$, by Proposition~\ref{thm-hardy}, we have 
\[
\| T_m\|_{(H^1, \BMOA)_{cb}}  = \Big\| \sum_{n=0}^\infty m_n e^{in \theta}\Big\|_{\BMOA}. 
\]
Therefore,  by recalling our choice  \eqref{bmoa-Linf} of the norm on $\BMOA$, for any $\varepsilon> 0$, we can find $\psi\in L^\infty$ such that 
\[
\| \psi\|_{L^\infty} \le \| T_m\|_{(H^1, \BMOA)_{cb}} +\varepsilon  \an   \mathcal{R}_{+} (\psi) (e^{i\theta})= \sum_{n=0}^\infty m_n e^{in\theta},
\]
where $\mathcal{R}_{+}$ is the Riesz projection. Define $K_\psi: L^1 \rightarrow L^\infty$ by  $K_\psi(f)= f*\psi$ and the  commutative diagram \eqref{cd-diag} follows immediately. 

Conversely, assume that $T_m$ satisfies \eqref{cd-diag}. Since both the embedding map $H^1 \rightarrow L^1$ and the quotient map $L^\infty \rightarrow L^\infty/\overline{H_0^\infty}$ are completely contractive, by the property for minimal \oss on $L^\infty$ (or the maximal \oss on $L^1$) (see \cite[Chapter 3]{Pisier-Operator-space-book}),  we have 
\[
\| T_m\|_{(H^1, \BMOA)_{cb}} \le \|K_\psi\|_{CB(L^1, L^\infty)} = \| K_\psi\|_{B(L^1, L^\infty)} = \| \psi\|_{L^\infty}.
\] 
This completes the whole proof. 
\end{proof}

\begin{proof}[Proof of Corollary \ref{cor-iso}]
Fix any Fourier multiplier $T_m$ defined on $H^1$. Since the subspace $H^\infty \subset L^\infty$ carries the miminal \oss (see \cite[Chapter 3]{Pisier-Operator-space-book}), we have 
\[
\| T_m\|_{(H^1, H^\infty)} = \| T_m\|_{(H^1, H^\infty)_{cb}}.
\]
The  quotient map $Q: L^\infty\rightarrow \BMOA = L^\infty/\overline{H_0^\infty}$ is  completely contractive. Hence, by factorizing $T_m: H^1 \rightarrow \BMOA$ as  
\[
H^1\xrightarrow{T_m} H^\infty \xrightarrow{Q|_{H^\infty}} \BMOA,
\]
we obtain 
\begin{align*}
\|T_m\|_{(H^1, \BMOA)_{cb}}  \le \| T_m\|_{(H^1, H^\infty)_{cb}} \cdot \| Q|_{H^\infty}: H^\infty\rightarrow \BMOA\|_{cb}
 \le \| T_m\|_{(H^1, H^\infty)}.
\end{align*}

Conversely,  assume that $\|T_m\|_{(H^1, \BMOA)_{cb}}< 1$. By Corollary \ref{cor-hardy-one-inf}, there exists a function  $\psi$ on $\T$ with   $\| \psi\|_{L^\infty}< 1$ and 
\begin{align}\label{psi-riesz}
 \mathcal{R}_{+} (\psi) (e^{i\theta})= \sum_{n=0}^\infty m_n e^{in\theta}, 
\end{align}
such that $T_m: H^1\rightarrow \BMOA$  factorizes as  
\[
H^1 \xhookrightarrow[\textit{embedding}]{\mathcal{I}} L^1\xrightarrow[\textit{convolution}]{K_\psi} L^\infty \xtwoheadrightarrow[\textit{quotient}]{Q} \BMOA.
\]
Observe that $(K_\psi \circ \mathcal{I}) (H^1) \subset H^\infty$.  By \eqref{psi-riesz},   the multiplier $T_m: H^1 \rightarrow H^\infty$  coincides with  the map $K_\psi \circ \mathcal{I}: H^1 \rightarrow H^\infty$. Consequently, 
\[
\| T_m\|_{(H^1, H^\infty)} \le \| K_\psi \circ \mathcal{I}: H^1 \rightarrow H^\infty\| \le \| \psi\|_{L^\infty}<1. 
\]
This completes the whole proof.
\end{proof}

\subsection{Failure of hyper-complete-contractivity for  the Poisson semigroup.}

Recall the Poisson convolution $P_r$  on $\T$  given by \eqref{def-Poisson}.  The hypercontractivity result due to Janson (see \cite{Janson-Poisson-hypercontractivity} or \cite{Weissler-Poisson-hypercontractivity}) says that $P_r: H^1\rightarrow H^2$ is contractive if and only if $r\le \sqrt{1/2}$. However, Proposition~\ref{thm-hardy} implies that there is no completely contractive version of hypercontractivity for Poisson semigroup.  

\begin{corollary}\label{cor-hyper}
 $\| P_r\|_{(H^1, H^2)_{cb}} = (1-r^4)^{-1/2}>1$ for any $r\in (0, 1)$.  
\end{corollary}

\begin{proof}
Given any $r\in (0, 1)$, the Poisson convolution $P_r: H^1\rightarrow H^2$ is just the Fourier multiplier associated with the sequence $(r^n)_{n=0}^\infty$. Therefore, by  Proposition~\ref{thm-hardy}, 
\[
\| P_r\|_{(H^1, H^2)_{cb}} = \Big\| \sum_{n = 0}^\infty r^{n} e^{in \theta} \Big\|_{\BMOA^{(2)}} = \big\|[ r^{2i + 2j} ]_{i, j \ge 0} \big\|_{B(\ell^2)}^{1/2}. 
\]
Since $[ r^{2i + 2j} ]_{i, j \ge 0} = v_r v_r^*$ with $ v_r= (1, r^2, r^4, \cdots, r^{2i}, \cdots )$ a row vector, we have
\[
\big\|[ r^{2i + 2j} ]_{i, j \ge 0} \big\|_{B(\ell^2)} = \| v_r\|_{\ell^2}^2 =  \sum_{n =0}^\infty r^{4n}  = \frac{1}{1-r^4}. 
\]
This completes the whole proof. 
\end{proof}


\begin{thebibliography}{HKZ00}
\bibitem[AMZ22]{arnold2022s1bounded}
Loris Arnold, Christian~Le Merdy, and Safoura Zadeh.
\newblock $S^1$-bounded fourier multipliers on $H^1({\mathbb R})$ and
functional calculus for semigroups.
\newblock {\em arXiv: 2203.16829, to appear in J.Anal. Math.}

\bibitem[AMZ23]{arnold2023hankel}
Loris Arnold, Christian~Le Merdy, and Safoura Zadeh.
\newblock Hankel operators on $L^p(\mathbb{R}_+)$ and their $p$-completely
bounded multipliers.
\newblock {\em arXiv: 2301.09481.}


\bibitem[BR11]{book-control}
Jean Berstel and Christophe Reutenauer.
\newblock {\em Noncommutative rational series with applications}, volume 137 of
  {\em Encyclopedia of Mathematics and its Applications}.
\newblock Cambridge University Press, Cambridge, 2011.

\bibitem[BPc91]{Blasco-TAMS}
Oscar Blasco and Aleksander Pe\l~\!\!czy\'{n}ski.
\newblock Theorems of {H}ardy and {P}aley for vector-valued analytic functions
  and related classes of {B}anach spaces.
\newblock {\em Trans. Amer. Math. Soc.}, 323(1):335--367, 1991.

\bibitem[BX91]{Xu-Blasco-interpolation1991}
Oscar Blasco and Quan~Hua Xu.
\newblock Interpolation between vector-valued {H}ardy spaces.
\newblock {\em J. Funct. Anal.}, 102(2):331--359, 1991.


\bibitem[BOCS21]{Brevig-2021}
Ole~Fredrik Brevig, Joaquim Ortega-Cerd\`a, and Kristian Seip.
\newblock Idempotent {F}ourier multipliers acting contractively on {$H^p$}
  spaces.
\newblock {\em Geom. Funct. Anal.}, 31(6):1377--1413, 2021.

\bibitem[BG09]{Bruns-Polytopes-K-theory}
Winfried Bruns and Joseph Gubeladze.
\newblock {\em Polytopes, rings, and {$K$}-theory}.
\newblock Springer Monographs in Mathematics. Springer, Dordrecht, 2009.

\bibitem[Bus71]{Bush-George-TAMS}
George~C. Bush.
\newblock The embeddability of a semigroup---{C}onditions common to {M}alcev
  and {L}ambek.
\newblock {\em Trans. Amer. Math. Soc.}, 157:437--448, 1971.





\bibitem[CAGPPT23]{Parcet-annals}
Jos{\'e}~M. Conde-Alonso, Adri{\'a}n~M. Gonz{\'a}lez-P{\'e}rez, Javier Parcet,
  and Eduardo Tablate.
\newblock {Schur multipliers in Schatten-von Neumann classes}.
\newblock {\em Annals of Mathematics}, 198(3):1229 -- 1260, 2023.

\bibitem[Con94]{Connes}
Alain Connes. 
\newblock {\em Noncommutative geometry}.
\newblock  San Diego, CA, Academic Press Inc., 1994.




\bibitem[DR97]{invention-connes}
G\'{e}rard Duchamp and Christophe Reutenauer.
\newblock Un crit\`ere de rationalit\'{e} provenant de la g\'{e}om\'{e}trie non
  commutative.
\newblock {\em Invent. Math.}, 128(3):613--622, 1997.



\bibitem[ER22]{Ruan-book}
Edward~G. Effros and Zhong-Jin Ruan.
\newblock {\em Theory of operator spaces}.
\newblock AMS Chelsea Publishing, Providence, RI, 2022.

\bibitem[Fli74]{Fliess}
Michel Fliess.
\newblock Matrices de {H}ankel.
\newblock {\em J. Math. Pures Appl. (9)}, 53:197--222, 1974.

\bibitem[GK89]{Gamelin-Varleman-embedding-monthly}
Theodore~W. Gamelin and Dmitry Khavinson.
\newblock The isoperimetric inequality and rational approximation.
\newblock {\em Amer. Math. Monthly}, 96(1):18--30, 1989.

\bibitem[GS03]{Bernoulli-hankel-semigroup}
Valerie Girardin and Rachid Senoussi.
\newblock Semigroup stationary processes and spectral representation.
\newblock {\em Bernoulli}, 9(5):857--876, 2003.

\bibitem[Gro53]{Grothendieck-1953}
Alexander Grothendieck.
\newblock R\'{e}sum\'{e} de la th\'{e}orie m\'{e}trique des produits tensoriels
topologiques.
\newblock {\em Bol. Soc. Mat. S\~{a}o Paulo}, 8:1--79, 1953.

\bibitem[HLP52]{HLP-inequality}
Godfrey~H. Hardy, John~E. Littlewood, and George P\'{o}lya.
\newblock {\em Inequalities}.
\newblock Cambridge, at the University Press, 1952.
\newblock 2d ed.


\bibitem[HKZ00]{Hedenmalm-Zhukehe-Berman-space-GTM}
Haakan Hedenmalm, Boris Korenblum, and Kehe Zhu.
\newblock {\em Theory of {B}ergman spaces}, volume 199 of {\em Graduate Texts
  in Mathematics}.
\newblock Springer-Verlag, New York, 2000.

\bibitem[Hel10]{Helson2010}
Henry Helson.
\newblock Hankel forms.
\newblock {\em Studia Math.}, 198(1):79--84, 2010.




\bibitem[Jan83]{Janson-Poisson-hypercontractivity}
Svante Janson.
\newblock On hypercontractivity for multipliers on orthogonal polynomials.
\newblock {\em Ark. Mat.}, 21(1):97--110, 1983.


\bibitem[JVA16]{Mulitipier-Fefferman-condition}
Miroljub Jevti\'{c}, Dragan Vukoti\'{c}, and Milo\v{s} Arsenovi\'{c}.
\newblock {\em Taylor coefficients and coefficient multipliers of {H}ardy and
  {B}ergman-type spaces}, volume~2 of {\em RSME Springer Series}.
\newblock Springer, Cham, 2016.


\bibitem[Kat23]{Katsoulis:2023aa}
Elias Katsoulis.
\newblock Fell's absorption principle for semigroup operator algebras.
\newblock {\em ar{X}iv: 2302.12809, to appear in J. Noncommutative Geometry}.

\bibitem[Lam51]{Lambk-immersibility}
Joachim Lambek.
\newblock The immersibility of a semigroup into a group.
\newblock {\em Canad. J. Math.}, 3:34--43, 1951.



\bibitem[LM71]{duke-hankel-semigroup}
Robert~J. Lindahl and Peter~H. Maserick.
\newblock Positive-definite functions on involution semigroups.
\newblock {\em Duke Math. J.}, 38:771--782, 1971.

\bibitem[Mag50]{Magnus-AJM}
Wilhelm Magnus.
\newblock On the spectrum of {H}ilbert's matrix.
\newblock {\em Amer. J. Math.}, 72:699--704, 1950.

\bibitem[Mal39]{Malcev-I}
Anatoly I.~Malcev.
\newblock \"{U}ber die {E}inbettung von assoziativen {S}ystemen in {G}ruppen.
\newblock {\em Rec. Math. [Mat. Sbornik] N.S.}, 6 (48):331--336, 1939.

\bibitem[Mal40]{Malcev-II}
Anatoly I.~Malcev.
\newblock \"{U}ber die {E}inbettung von assoziativen {S}ystemen in {G}ruppen.
  {II}.
\newblock {\em Rec. Math. [Mat. Sbornik] N.S.}, 8 (50):251--264, 1940.



\bibitem[MP90]{Mateljevic-Pacific-journal-math1990}
Miodrag Mateljevi\'{c} and Miroslav Pavlovi\'{c}.
\newblock Multipliers of {$H^p$} and {BMOA}.
\newblock {\em Pacific J. Math.}, 146(1):71--84, 1990.



\bibitem[Miy22]{MIYAGAWA2022109609}
Akihiro Miyagawa.
\newblock A characterization of rationality in free semicircular operators.
\newblock {\em Journal of Functional Analysis}, 283(8):109609, 2022.

\bibitem[Neh57]{Nehari-Annofmath}
Zeev Nehari.
\newblock On bounded bilinear forms.
\newblock {\em Ann. of Math. (2)}, 65:153--162, 1957.

\bibitem[OCS12]{Ortega-2012}
Joaquim Ortega-Cerd\`a and Kristian Seip.
\newblock A lower bound in {N}ehari's theorem on the polydisc.
\newblock {\em J. Anal. Math.}, 118(1):339--342, 2012.

\bibitem[Pag70]{Page-Nehari-vectorial}
Lavon~B. Page.
\newblock Bounded and compact vectorial {H}ankel operators.
\newblock {\em Trans. Amer. Math. Soc.}, 150:529--539, 1970.

\bibitem[Pav14]{Pavlovic-book}
Miroslav Pavlovi\'{c}.
\newblock {\em Function classes on the unit disc: An introduction}, volume~52 of {\em De Gruyter
  Studies in Mathematics}.
\newblock De Gruyter, Berlin, 2014.


\bibitem[Pel03]{Peller-book}
Vladimir~V. Peller.
\newblock {\em Hankel operators and their applications}.
\newblock Springer Monographs in Mathematics. Springer-Verlag, New York, 2003.

\bibitem[Pis96]{Pisier-OH-Memoirs-book}
Gilles Pisier.
\newblock The operator {H}ilbert space {${\rm OH}$}, complex interpolation and
  tensor norms.
\newblock {\em Mem. Amer. Math. Soc.}, 122(585):viii+103, 1996.

\bibitem[Pis98]{Pisier-Non-Commutative-vector}
Gilles Pisier.
\newblock Non-commutative vector valued {$L_p$}-spaces and completely
  {$p$}-summing maps.
\newblock {\em Ast\'{e}risque}, (247):vi+131, 1998.

\bibitem[Pis01]{Pisier-similar-cb}
Gilles Pisier.
\newblock {\em Similarity problems and completely bounded maps}, volume 1618 of
{\em Lecture Notes in Mathematics}.
\newblock Springer-Verlag, Berlin, expanded edition, 2001.
\newblock Includes the solution to ``The Halmos problem''.

\bibitem[Pis03]{Pisier-Operator-space-book}
Gilles Pisier.
\newblock {\em Introduction to operator space theory}, volume 294 of {\em
  London Mathematical Society Lecture Note Series}.
\newblock Cambridge University Press, Cambridge, 2003.

\bibitem[Pis16]{Pisier-martingale-in-Banach2016}
Gilles Pisier.
\newblock {\em Martingales in {B}anach spaces}, volume 155 of {\em Cambridge
  Studies in Advanced Mathematics}.
\newblock Cambridge University Press, Cambridge, 2016.

\bibitem[Pis20]{newbook-pisier}
Gilles Pisier.
\newblock {\em Tensor products of {$C^*$}-algebras and operator spaces---the
  {C}onnes-{K}irchberg problem}, volume~96 of {\em London Mathematical Society
  Student Texts}.
\newblock Cambridge University Press, Cambridge, 2020.

\bibitem[Rua88]{Ruan-JFA-Subspces}
Zhong-Jin Ruan.
\newblock Subspaces of {$C^*$}-algebras.
\newblock {\em J. Funct. Anal.}, 76(1):217--230, 1988.


\bibitem[Sar67]{Sarason-Nehari-vectorial}
Donald Sarason.
\newblock Generalized interpolation in {$H^{\infty }$}.
\newblock {\em Trans. Amer. Math. Soc.}, 127:179--203, 1967.

\bibitem[Sch77]{mathz-semigroup}
Walter Schempp.
\newblock On functions of positive type on commutative monoids.
\newblock {\em Math. Z.}, 156(2):115--121, 1977.

\bibitem[SZ67]{Stein-Zygmund-Annals}
Elias~M. Stein and Antoni Zygmund.
\newblock Boundedness of translation invariant operators on {H}\"{o}lder spaces
  and {$L^{p}$}-spaces.
\newblock {\em Ann. of Math. (2)}, 85:337--349, 1967.

\bibitem[Wei80]{Weissler-Poisson-hypercontractivity}
Fred~B. Weissler.
\newblock Logarithmic {S}obolev inequalities and hypercontractive estimates on
  the circle.
\newblock {\em J. Funct. Anal.}, 37(2):218--234, 1980.


\bibitem[Xu90]{Xu-studia1990}
Quan~Hua Xu.
\newblock Applications du th\'{e}or\`eme de factorisation pour des fonctions
  \`a valeurs op\'{e}rateurs.
\newblock {\em Studia Math.}, 95(3):273--292, 1990.

\end{thebibliography}

\end{document}